\documentclass[11pt]{article}

\usepackage[utf8]{inputenc}
\usepackage[T1]{fontenc}
\usepackage{lmodern}       
\usepackage[a4paper, top=3cm, bottom=3cm, left=2.5cm, right=2.5cm]{geometry}

\usepackage{amsmath,amssymb,amsfonts,amsthm,mathtools} 
\usepackage{mathrsfs}      
\usepackage{bm}            
\usepackage{stmaryrd}      
\usepackage{tikz, tikz-cd} 
\usepackage{graphicx}      
\usepackage{float}        
\usepackage{subcaption}    
\usepackage{xcolor}        
\usepackage{hyperref}      
\hypersetup{colorlinks=true, linkcolor=blue, citecolor=blue, urlcolor=blue}
\usepackage{enumitem}      
\usepackage{microtype} 
\usepackage{graphicx}
\usepackage{authblk} 
\DeclareUnicodeCharacter{2212}{\ensuremath{-}}
\usepackage[percent]{overpic}
\newtheorem{thm}{Theorem}[section]
\newtheorem{lem}[thm]{Lemma}
\newtheorem{pro}[thm]{Proposition}
\newtheorem{cor}[thm]{Corollary}

\newtheorem{RHP}[thm]{RH problem}
\theoremstyle{definition}

\newtheorem{assumption}{Assumption}
\theoremstyle{remark}
\newtheorem{rmk}[thm]{Remark}

\numberwithin{equation}{section}

\newcommand{\R}{\mathbb R}
\newcommand{\C}{\mathbb C}

\def\dn{\mathrm{dn}}
\newcommand{\bJ}{\mathbf{J}}

\title{\bf Direct Scattering for the KdV Equation with a Step-like Finite-Gap Potential: A Riemann--Hilbert Approach}
\author[1,2]{Xiaodong Zhu\thanks{Email: \texttt{xdzbnu@mail.bnu.edu.cn}}}
\affil[1]{Laboratory of Mathematics and Complex Systems (Ministry of Education), School of Mathematical Sciences, Beijing Normal University, Beijing 100875, China}
\affil[2]{{SISSA, via Bonomea 265, 34136 Trieste, Italy, INFN Sezione di Trieste}}
\date{}

\begin{document}
\maketitle

\begin{abstract}
We develop the direct scattering theory for the KdV equation with step-like finite-gap backgrounds under perturbations. More precisely, we consider initial data that asymptotically approach two distinct one-gap periodic travelling wave solutions as \(x \to \pm \infty\). Under suitable assumptions on the perturbation, we formulate the direct scattering problem and establish the analytic structure of the associated scattering data. In particular, we reformulate the problem in terms of a vector Riemann--Hilbert problem, which provides a foundation for the study of long-time asymptotics of perturbed finite-gap potentials. This formulation highlights the connection between step-like finite-gap scattering theory and the Riemann--Hilbert framework arising in soliton-gas type settings.
\end{abstract}
\tableofcontents
\section{Introduction}
In this paper, we study the direct and inverse scattering problems for the Korteweg–de Vries (KdV) equation with step-like oscillatory initial data, which asymptotically approach two distinct periodic wave solutions as $x \to \pm \infty$. 
The KdV equation is given by
\begin{equation}\label{eq:KdV}
    u_t - 6 u u_x + u_{xxx} = 0,\quad (x,t)\in\mathbb{R}\times\mathbb{R}_+,
\end{equation}
where $u=u(x,t)$ is a real-valued function, and the subscripts denote partial derivatives with respect to the corresponding variables. This equation was originally derived to model the propagation of long, weakly nonlinear dispersive shallow-water waves, and it admits \textit{soliton solution}
\begin{equation*}
    u(x,t) = -2 \eta^2 \operatorname{sech}^2 \!\left( \eta \left( x - 4 \eta^2 t - x_0 \right) \right),
\end{equation*}
where $\eta > 0$ and $x_0 \in \mathbb{R}$ denotes the phase shift.

The KdV equation \eqref{eq:KdV} is a well-known integrable system, whose spatial part is associated with the one-dimensional Schr\"odinger operator. It admits the following Lax pair in operator form \cite{gardner_method_1967}:
\begin{equation}\label{Lax:operator}
   \begin{aligned}
  &\mathscr{L}\varphi = \left(-\partial_x^{2} + u\right)\varphi = \lambda^{2}\varphi,\\
  &\varphi_{t} = \mathscr{A}\varphi = \left[-4\partial_x^{3} 
     + 3\bigl(u\,\partial_x + u_x\bigr)\right]\varphi .
\end{aligned} 
\end{equation}

On the other hand, the KdV equation \eqref{eq:KdV} also admits the following matrix-form Lax pair:
\begin{equation}\label{Lax:matrix}
    \begin{aligned}
        &\begin{cases}
        \Psi_x = L\,\Psi,\\
        \Psi_t = A\,\Psi,
        \end{cases}\\[2mm]
    \end{aligned}
\end{equation}
with
\begin{equation*}
    \begin{aligned}
        &L=
        \begin{pmatrix}
        0 & 1\\
        -\lambda^2+u & 0
        \end{pmatrix},\qquad
        A=
        \begin{pmatrix}
        -u_x & 4\lambda^2+2u\\
        (u-\lambda^2)(4\lambda^2+2u)-u_{xx} & u_x
        \end{pmatrix}.
    \end{aligned}
\end{equation*}
In particular, the KdV equation \eqref{eq:KdV} can be recovered from the zero-curvature condition
\begin{equation}\label{zero curvature equation}
    L_t - A_x + [L,A] = 0,
\end{equation}
where $[L,A]=LA-AL$ denotes the commutator of $L$ and $A$.

In addition, the KdV equation~\eqref{eq:KdV} admits the following periodic travelling-wave solution, commonly referred to as a \textit{(quasi-)periodic finite-gap solution}~\cite{gurevich_nonstationary_1974}:
\begin{equation}\label{solution:periodic solution}
    u(x,t) = -\beta_3 - (\beta_1-\beta_3)\,
    \mathrm{dn}^2\!\left(
        \frac{\sqrt{\beta_1-\beta_3}}{\sqrt{2}}
        \left(x - 2(\beta_1+\beta_2+\beta_3)t - x_0\right)
        + K(m) \,\middle|\, m
    \right),
\end{equation}
where $\beta_1>\beta_2>\beta_3$, $x_0\in\mathbb{R}$ is an arbitrary constant, and $\mathrm{dn}(s\,|\,m)$ denotes the Jacobi elliptic function with period $2K(m)$. Here
\[
K(m)=\int_{0}^{\pi/2}\frac{d\vartheta}{\sqrt{1-m^2\sin^2\vartheta}}
\]
is the complete elliptic integral of the first kind, and the elliptic modulus $m$ is given by
\(
m^2=\frac{\beta_1-\beta_2}{\beta_1-\beta_3}.
\)
See Appendix~\ref{appendix:periodic solution} for further details.

Moreover, the \emph{Lax spectrum} (or the \emph{Bloch spectrum}; see Appendix~\ref{appendix:Bloch spectrum}) 
of the Lax pair~\eqref{Lax:matrix} with the potential~\eqref{solution:periodic solution} is defined as the set of 
$\lambda\in\mathbb{C}$ for which the associated eigenfunctions are bounded for all $x\in\mathbb{R}$, and is given by
\begin{equation}\label{spectral curve}
    \Sigma(\lambda^2)
    = \left\{ \lambda^2 \in \mathbb{R} \;\bigg|\; 
      \lambda^2 \in \Big(-\infty, -\frac{\beta_1 + \beta_2}{2}\Big] 
      \;\cup\; 
      \Big[-\frac{\beta_1 + \beta_3}{2}, -\frac{\beta_3 + \beta_2}{2}\Big] 
    \right\}.
\end{equation}
In particular, let $\beta_1=\eta_2^2+\eta_1^2,\beta_2=\eta_2^2-\eta_1^2$ and $\beta_3=\eta_1^2-\eta_2^2$, with $\eta_2>\eta_1>0$, we can get the following solution 
\begin{equation*}
u(x,t) = \eta_2^2 - \eta_1^2 
- 2\eta_2^2\, 
\mathrm{dn}^2\!\left(
  \eta_2\,[\,x - 2(\eta_1^2 + \eta_2^2)t\,- x_0]  + K(m)\, \middle|\, m
\right),\quad m=\frac{\eta_1}{\eta_2}.
\end{equation*}


In this manuscript, we study the direct scattering problem for the KdV equation~\eqref{eq:KdV} with step-like finite-gap potentials. The direct scattering problem for step-like operators has a clear physical motivation. Indeed, it appears, for example, in the analysis of alloys formed by two different semi-infinite one-dimensional crystals \cite{gesztesy_one-dimensional_1997}.
More precisely, let \( 0 < \eta_1^l < \eta_2^l \) and \( 0 < \eta_1^r < \eta_2^r \), and define
\begin{equation}\label{initial:ulur}
\begin{aligned}
    u_0^l(x)
    &= (\eta_2^l)^2 - (\eta_1^l)^2
        - 2(\eta_2^l)^2 \,
        \dn^2\!\left(
            \eta_2^l (x - x_0^l) + K(m^l)
            \,\Big|\, m^l
       \right),\\[0.3em]
    u_0^r(x)
    &= (\eta_2^r)^2 - (\eta_1^r)^2
        - 2(\eta_2^r)^2 \,
        \dn^2\!\left(
            \eta_2^r (x - x_0^r) + K(m^r)
            \,\Big|\, m^r
       \right),
\end{aligned}
\end{equation}
where \( x_0^l, x_0^r \in \mathbb{R} \) are phase shifts, and 
\( m^l = \eta_1^l / \eta_2^l \), \( m^r = \eta_1^r / \eta_2^r \). Assume that the initial data \( u_0(x) = u(x,0) \) satisfies the step-like finite-gap condition
\begin{equation}\label{initial}
\begin{aligned}
      &\int_0^{\infty}
      \left|\frac{d^n}{dx^n}\!\big(u_0(x) - u_0^r(x)\big)\right|
      (1+|x|^{m_0})\,dx < \infty,\\[0.3em]
      &\int_{-\infty}^0
      \left|\frac{d^n}{dx^n}\!\big(u_0(x) - u_0^l(x)\big)\right|
      (1+|x|^{m_0})\,dx < \infty,
\end{aligned}
\end{equation}
for \( 0 \le n \le n_0 \), where \( m_0 \ge 8 \) and \( n_0 \ge m_0 + 5 \).
If the initial data \( u_0(x) \) satisfies~\eqref{initial}, then by~\cite{egorova_cauchy_2011} there exists a unique classical solution 
\(
u(x,t) \in C^{n_0 - m_0 - 2}(\mathbb{R})\times C^1(\R)
\)
to the KdV equation~\eqref{eq:KdV}.
In \cite{trogdon_numerical_2014}, the authors develop a unified numerical framework for the KdV equation that enables the accurate computation of solutions arising from initial data asymptotically approaching (quasi-)periodic finite-gap profiles. In particular, Grava \textit{et al.} studied the direct and inverse scattering theory for the focusing nonlinear Schrödinger equation with step-like oscillatory initial data; see \cite{grava_direct_nodate}.

\begin{rmk}
Notice that the KdV equation possesses Galilean invariance: if $u(x,t)$ is a solution of \eqref{eq:KdV}, then for any constant $c\in\mathbb{R}$,
$\tilde u(x,t)=u(x-6ct,t)-c$
is also a solution of \eqref{eq:KdV}. In view of the Lax pair \eqref{Lax:operator}, we then have
$(-\partial_x^2+\tilde u)\varphi = (-\partial_x^2+u-c)\varphi = (\lambda^2-c)\varphi$.
Hence the spectral parameter is shifted by a constant, and we may always use a Galilean transformation to set the parameter $\beta_2+\beta_3$ in \eqref{spectral curve} to $0$. The associated Lax spectrum $\Sigma(\lambda^2)$ is then
$\Sigma(\lambda^2)=(-\infty,-\eta_2^2]\cup[-\eta_1^2,0]$.
Accordingly, throughout this work we fix the spectra of $u_0^l(x)$ and $u_0^r(x)$ so that they share the common endpoint $0$, see, for example, Figure~\ref{fig:example}.
\end{rmk}
\begin{figure}[H]
    \centering
    \includegraphics[width=0.9\linewidth]{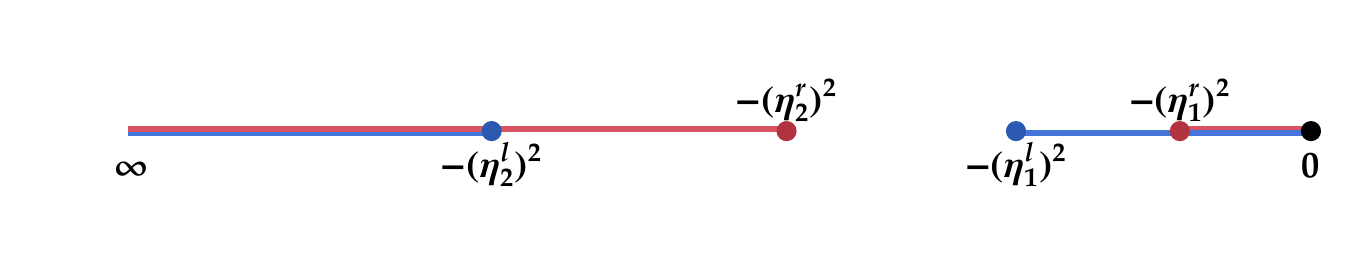}
 \caption{\small The Lax spectrum associated with the left background $u_0^l(x)$ (blue) and the right background $u_0^r(x)$ (red), which share the common endpoint $0$. The endpoints of the spectral bands for $u_0^l$ are given by $-(\eta_1^l)^2$ and $-(\eta_2^l)^2$, while the endpoints of the spectral bands for $u_0^r$ are given by $-(\eta_1^r)^2$ and $-(\eta_2^r)^2$.}
    \label{fig:example}
\end{figure}

\begin{rmk}
In what follows, we study the direct and inverse scattering problem in the $\lambda$-plane rather than working directly with the energy plane ($\lambda^2$). The Lax spectrum in energy plane illustrated in Figure~\ref{fig:example} corresponds to the configuration (iii) in Figure~\ref{fig:bands} in the $\lambda$-plane.
\end{rmk}

Although Egorova \textit{et al.}~\cite{egorova_cauchy_2009,egorova_cauchy_2011} studied the Cauchy problem for the KdV equation with step-like finite-gap initial data and proved that the KdV equation admits a unique classical solution $q(x,t)$ corresponding to the initial condition \eqref{initial} by means of the direct and inverse scattering method based on the Gelfand--Levitan--Marchenko (GLM) equations, the long-time asymptotic behavior of such solutions still deserves further investigation. In this manuscript, we formulate an associated Riemann--Hilbert (RH) problem, which provides a convenient framework for the analysis of the long-time asymptotics of these solutions.

\section{Main results}
Our results are formulated in terms of three reflection coefficients $r_1(\lambda)$, $r_2(\lambda)$, and $\rho(\lambda)$, which can be regarded as the corresponding reflection coefficients for the KdV equation \eqref{eq:KdV} determined by the initial data \eqref{initial}. These reflection coefficients are then used to formulate a RH problem, from which the solution $u(x,t)$ can be recovered.
\subsection*{The direct scattering problem}
Let $\Psi_0^l(x;\lambda)$ and $\Psi_0^r(x;\lambda)$, defined in \eqref{def:Psi0}, be the $2\times2$-valued eigenfunctions satisfy the spatial part \eqref{Lax:matrix} with respect to $u_0^l(x)$ and $u_0^r(x)$, respectively and their corresponding spectra are $\Sigma^l:=(i\eta_1^l,i\eta_2^l)\cup(-i\eta_2^l,-i\eta_1^l)$ and $\Sigma^r:=(i\eta_1^r,i\eta_2^r)\cup(-i\eta_2^r,-i\eta_1^r)$. Introduce the Jost solution $\Psi^s(x,t)$, with $s\in\{r,l\}$,  of \eqref{Lax:matrix} for the initial data $u_0(x)$ satisfying the condition \eqref{initial} such that
\begin{equation*}
    \Psi^l(x;\lambda)=\Psi_0^l(x;\lambda)(I+\mathcal{O}(|x|^{-1})),\quad x\to\infty_s,
\end{equation*}
with $\infty_{l}=-\infty$ and $\infty_{r}=+\infty$. Indeed, the Jost solutions $\Psi^s(x,\lambda)$ satisfy the integral equation
\begin{equation*}
    \Psi^s(x;\lambda)=\Psi_0^s(x;\lambda)+\int_{\infty_s}^x\Psi_0^s(x;\lambda)(\Psi_0^s(y;\lambda))^{-1}\Delta\tilde U^s(y)\Psi^s(y;\lambda)dy,
\end{equation*}
where 
$$
\Delta \tilde{U}^s(x):=\begin{pmatrix}
    0&0\\
    u_0^s(x)-u_0(x)&0
\end{pmatrix}.
$$

Define the transformed Jost solutions $\Phi^s(x;\lambda)$ by
\[
\Phi^s(x;\lambda):=
\begin{pmatrix}
1 & 1 \\
-i\lambda & i\lambda
\end{pmatrix}\Psi^s(x;\lambda),
\]
then $\Phi^s(x;\lambda)$ satisfies Proposition~\ref{prop:Phi}. In particular, let $\Phi_1^s(x;\lambda)$ and $\Phi_2^s(x;\lambda)$ denote the first and second columns of $\Phi^s(x;\lambda)$, respectively. Then the following properties hold:
\begin{itemize}
    \item $\Phi^l(x;\lambda)$ admits continuous boundary values on both sides of $\Sigma_1^l\cup\Sigma_2^l$ away from $\{\pm i\eta_1^l,\pm i\eta_2^l\}$. Moreover, $\Phi_1^l(x;\lambda)$ is analytic for $\lambda\in\mathbb{C}_+\setminus\Sigma_1^l$, and $\Phi_2^l(x;\lambda)$ is analytic for $\lambda\in\mathbb{C}_-\setminus\Sigma_2^l$.
    
    \item $\Phi^r(x;\lambda)$ admits continuous boundary values on both sides of $\Sigma_1^r\cup\Sigma_2^r$ away from $\{\pm i\eta_1^r,\pm i\eta_2^r\}$. Moreover, $\Phi_2^r(x;\lambda)$ is analytic for $\lambda\in\mathbb{C}_+\setminus\Sigma_1^r$, and $\Phi_1^r(x;\lambda)$ is analytic for $\lambda\in\mathbb{C}_-\setminus\Sigma_2^r$.
\end{itemize}

 We introduce the scattering data
\[
\begin{aligned}
a(\lambda)&=\det\bigl[\Phi_1^l(x;\lambda),\Phi_2^r(x;\lambda)\bigr],\qquad
b(\lambda)=\det\bigl[\Phi_1^r(x;\lambda),\Phi_1^l(x;\lambda)\bigr],\quad \lambda\in\mathbb{R}.
\end{aligned}
\]
For $\lambda\in\Sigma_1^r\cup\Sigma_1^l$, we define the boundary values of the scattering data $a(\lambda)$ and $b_1(\lambda)$ by
\[
a(\lambda_{\pm})=\det\bigl[\Phi_1^l(x;\lambda_{\pm}),\Phi_2^r(x;\lambda_{\pm})\bigr],\qquad
b_1(\lambda_{\pm})=\det\bigl[\Phi_1^r(x;\lambda_{\pm}),\Phi_1^l(x;\lambda_{\pm})\bigr],
\]
where the subscripts $\pm$ denote the limiting values from the left and right side of $\Sigma_1^r\cup\Sigma_1^k$, respectively. By Lemma~\ref{lem:ab}, it follows that $a(\lambda)$ admits an analytic continuation to
$\lambda\in\mathbb{C}_+\setminus(\Sigma_1^r\cup\Sigma_1^l)$. 

For simplicity, throughout this manuscript we restrict ourselves to the solitonless case, that is, we assume that $a(\lambda)$ has no zeros for
$\lambda\in\mathbb{C}_+\setminus(\Sigma_1^r\cup\Sigma_1^l)$.


\begin{assumption}[Absence of Solitons]\label{assumption}
    Assume the initial data $u_0(x)$ satisfies the condition \eqref{initial} and the corresponding  scattering data $a(\lambda)$ is nonzero for
$\lambda\in\mathbb{C}_+\setminus(\Sigma_1^r\cup\Sigma_1^l)$.
\end{assumption}

Using the auxiliary meromorphic function $h(\lambda)$ defined in \eqref{eq:h-lambda}, we introduce the modified scattering data $a_1(\lambda)$ and $a_2(\lambda)$ by
\begin{equation*}
    \begin{aligned}
        a_1(\lambda)&=\bigl(a(\lambda)/h(\lambda)\bigr)^{1/2},&&\lambda\in\C_+,\\
        a_2(\lambda)&=\bigl(a(\lambda)h(\lambda)\bigr)^{1/2},&&\lambda\in\C_+.
    \end{aligned}
\end{equation*}
These functions satisfy Proposition~\ref{pro:a12}. In terms of $a_1(\lambda)$ and $a_2(\lambda)$, we define the reflection coefficients $r_1(\lambda)$, $r_2(\lambda)$, and $\rho(\lambda)$ as
\begin{equation}\label{def:pectralfunction}
\begin{aligned}
    &r_1(\lambda)=\frac{a_2(\lambda_-)}{a_1(\lambda_-)}
    \frac{e^{-2ix_0^r\lambda}}{a_2(\lambda_-)a_1(\lambda_+)-ib_1^*(\lambda_-)}, 
    &&\lambda\in\Sigma_1^l,\\
    &r_2(\lambda)=\frac{a_1(\lambda_-)}{a_2(\lambda_-)}
    \frac{e^{2ix_0^r\lambda}}{a_1(\lambda_-)a_2(\lambda_+)+ib_1(\lambda_-)}, 
    &&\lambda\in\Sigma_1^r,\\
    &\rho(\lambda)=\frac{b(\lambda)}{a_1(\lambda)a_2^*(\lambda)}e^{-2ix_0^r\lambda},
    &&\lambda\in\R.
\end{aligned}
\end{equation}
We are now in a position to state the main lemma concerning the direct scattering problem.

\begin{lem}\label{lem:r1r2}
 Suppose that the initial datum $u_0$ satisfies the condition~\eqref{initial} and Assumption~\ref{assumption}.
The boundary values $r_{1}(\lambda_{\pm})$ are $m_0$--times differentiable on
\(
\Sigma_1^l \setminus \{ i\eta_1^l, i\eta_2^l \},
\)
and, as $\lambda \to i\eta_j^l$, $j=1,2$, they exhibit square--root behaviour.
Similarly, the boundary values $r_{2}(\lambda_{\pm})$ are $m_0$--times differentiable on
\(
\Sigma^r \setminus \{ i\eta_1^r, i\eta_2^r \},
\)
and, as $\lambda \to i\eta_j^r$, $j=1,2$, they exhibit square--root behaviour.

In particular, when $\lambda$ lies on the non--intersection part of the contour, the functions
$r_1(\lambda)$ and $r_2(\lambda)$ can be rewritten as
\begin{equation*}
    r_1(\lambda)
    = \frac{e^{-2i\lambda x_0^r}}{2\,a_1(\lambda_-)a_1(\lambda_+)},
    \qquad \lambda \in \Sigma_1^l \setminus \Sigma_1^r,
\end{equation*}
\begin{equation*}
    r_2(\lambda)
    = \frac{e^{-2i\lambda x_0^r}}{2\,a_2(\lambda_-)a_2(\lambda_+)},
    \qquad \lambda \in \Sigma_1^r \setminus \Sigma_1^l.
\end{equation*}

For $\lambda\in\mathbb{R}$, the spectral function $\rho(\lambda)$ is regular at $\lambda=0$ and belongs to $C^{m_0}(\mathbb{R})$. Moreover, as $\lambda\to\infty$, it satisfies
\(
|\rho(\lambda)|=\mathcal{O}\bigl(|\lambda|^{-4}\bigr).
\)
\end{lem}
\begin{proof}
    See section \ref{section:RHP}.
\end{proof}

\begin{cor}
	If the initial data $u_0(x)$ are a Schwartz-type perturbation, i.e., they satisfy \eqref{initial} for arbitrary $m_0,n_0\in\mathbb{N}$, then the spectral function $\rho(\lambda)$ in Lemma \ref{lem:r1r2} is smooth on $\mathbb{R}$ and decays rapidly as $\lambda\to\infty$. Moreover, the boundary values of the reflection coefficients $r_1(\lambda)$ and $r_2(\lambda)$ on $\Sigma_1^l$ and $\Sigma_1^r$, respectively, are also smooth, except possibly at the endpoints, where they still exhibit the same square-root behavior.
\end{cor}

\subsection*{Time evolution and the inverse problem}
We now consider the inverse scattering problem of recovering the solution $u(x,t)$ from the reflection coefficients $r_1(\lambda)$, $r_2(\lambda)$, and $\rho(\lambda)$. Moreover, we show that this inverse problem can be solved by formulating an associated RH problem in terms of $r_1(\lambda)$, $r_2(\lambda)$, and $\rho(\lambda)$.

\begin{RHP}\label{RHP:X}
    Seek a $1\times 2$ vector-valued meromorphic function $X(x,t;\lambda)$ satisfying the following conditions:
    \begin{enumerate}
        \item $X(x,t;\lambda)$ satisfies the jump condition
        \begin{equation}\label{jumps:X}
            X_+(x,t;\lambda) = X_-(x,t;\lambda) V(x,t;\lambda),
        \end{equation}
        where the jump matrix $V(x,t;\lambda)$ is given by
         \begin{equation*}
       V(x,t;\lambda)=\begin{cases}
            \begin{aligned}
                &\begin{pmatrix}
                    1 & -2i r_2(\lambda) e^{-2i \theta(x,t;\lambda)} \\
                    0 & 1
                \end{pmatrix},&&\lambda\in\Sigma_1^r\setminus\Sigma_1^l,\\
                &\begin{pmatrix}
                    \frac{1-r_1(\lambda)r_2(\lambda)}{1+r_1(\lambda)r_2(\lambda)} & \frac{-2i r_2(\lambda) }{1+r_1(\lambda)r_2(\lambda)} e^{-2i \theta(x,t;\lambda)}\\
                    \frac{-2i r_1(\lambda)}{1+r_1(\lambda)r_2(\lambda)} e^{2i \theta(x,t;\lambda)} & \frac{1-r_1(\lambda)r_2(\lambda)}{1+r_1(\lambda)r_2(\lambda)}
                \end{pmatrix},&&\lambda\in\Sigma_1^r\cap\Sigma_1^l,\\
                &\begin{pmatrix}
                    1 & 0 \\
                    -2i r_1(\lambda) e^{2i \theta(x,t;\lambda)} & 1
                \end{pmatrix},&&\lambda\in\Sigma_{1}^{l}\setminus\Sigma_1^r,\\
                &{\begin{pmatrix}
                    1-|\rho(\lambda)|^2&-\rho^*(\lambda)e^{-2i\theta(x,t;\lambda)}\\
                    \rho(\lambda)e^{2i\theta(x,t;\lambda)}&1
                \end{pmatrix}},&&\lambda\in\R,\\

                &\begin{pmatrix}
                    1 & 2i r_1(\lambda) e^{-2i \theta(x,t;\lambda)} \\
                    0 & 1
                \end{pmatrix},&&\lambda\in\Sigma_{2}^{l}\setminus\Sigma_2^r,\\
                 &\begin{pmatrix}
                    \frac{1-r_1(\lambda)r_2(\lambda)}{1+r_1(\lambda)r_2(\lambda)} & \frac{2i r_1(\lambda) }{1+r_1(\lambda)r_2(\lambda)} e^{-2i \theta(x,t;\lambda)}\\
                    \frac{2i r_2(\lambda)}{1+r_1(\lambda)r_2(\lambda)} e^{2i \theta(x,t;\lambda)} & \frac{1-r_1(\lambda)r_2(\lambda)}{1+r_1(\lambda)r_2(\lambda)}
                \end{pmatrix},&&\lambda\in\Sigma_2^{r}\cap\Sigma_2^{l},\\
                 &\begin{pmatrix}
                    1 & 0 \\
                    2i r_2(\lambda) e^{2i \theta(x,t;\lambda)} & 1
                \end{pmatrix},&&\lambda\in\Sigma_{2}^{r}\setminus\Sigma_2^l,\\
            \end{aligned}
        \end{cases}
    \end{equation*}
with $\theta(x,t;\lambda)=x\lambda+t\lambda^3$.
\item As $\lambda \to \infty$, the following asymptotic expansion holds:
        \begin{equation}
            X(x,t;\lambda) = \begin{bmatrix} 1 & 1 \end{bmatrix} + \mathcal{O}(\lambda^{-1}).
        \end{equation}

\item The function $X(x,t;\lambda)$ satisfies the symmetry condition
        \begin{equation}
            X^*(x,t;\lambda) = X(x,t; -\lambda) =  X(x,t;\lambda)\sigma_1.
        \end{equation}
    \end{enumerate}
\end{RHP}

\begin{rmk}
	The RH problem \ref{RHP:X} is closely related to the soliton-gas RH problem considered in \cite{dyachenko_primitive_2016}, but it additionally involves jump conditions on the real line. 
	Indeed, based on the spatial asymptotic analysis in \cite{girotti_rigorous_2021}, in the case $r_2(\lambda)=0$ and in the absence of jump conditions on the real line, the leading-order term of the corresponding initial data for such a KdV soliton gas is asymptotic to the periodic solution $u_0^l(x)$ in \eqref{initial:ulur}, while the error term is of order $\mathcal{O}(|x|^{-1})$. This indicates that the initial data for the KdV soliton gas do not satisfy the condition~\eqref{initial}; consequently, the classical inverse scattering transform (IST) method cannot be applied in a straightforward manner to the soliton gas case.
    
	Furthermore, the relationship between periodic solutions and soliton gases is investigated in \cite{nabelek_algebro-geometric_2020} via the dressing method. Moreover, in \cite{jenkins_approximation_2025}, the authors investigate the connection between finite-gap and multisoliton solutions of the focusing NLS equation, and suggest that, in the thermodynamic limit, both can be approximated by an appropriate generalization of Zakharov's primitive potentials~\cite{dyachenko_primitive_2016}.
\end{rmk}

Based on the RH problem~\ref{RHP:X}, we obtain the following main theorem.

\begin{thm}\label{thm:inverse}
Suppose that the initial data $u_0(x)$ satisfies the condition \eqref{initial} and Assumption~\ref{assumption}. Let the reflection coefficients $r_1(\lambda)$, $r_2(\lambda)$, and $\rho(\lambda)$ be defined in terms of $u_0(x)$ by \eqref{def:pectralfunction}. Then the RH problem $X(x,t;\lambda)$ admits a unique solution for each $(x,t)\in\R\times\R_+$.

Furthermore, the solution $u(x,t)$ of the KdV equation with initial data \eqref{initial} can be reconstructed from the solution $X(x,t;\lambda)$ of the RH problem~\ref{RHP:X} via
\begin{equation}
    u(x,t) = -2i\,\frac{d}{dx} \left( \lim_{\lambda \to \infty} \lambda \bigl( X_1(x,t;\lambda) - 1 \bigr) \right),
\end{equation}
where $X_1(x,t;\lambda)$ denotes the first component of the vector $X(x,t;\lambda)$.
\end{thm}
\begin{proof}
    See Section \ref{section:inverse}.
\end{proof}

\begin{rmk}
Although we have established the existence and uniqueness of the solution $u(x,t)$ for $t>0$ via the associated RH problem, we do not further investigate in this paper the precise functional space in which the solution is defined. Nevertheless, by the results of Egorova \textit{et al.}~\cite{egorova_cauchy_2009,egorova_cauchy_2011}, it is known that if the initial data satisfies \eqref{initial}, then the corresponding solution $u(x,t)$ remains a unique classical solution $u(x,t)\in C^{n_0-m_0-2}(\R)\times C^1(\R)$ and satisfying
$$
\left|\int_{0}^{\infty_s}\bigl|\partial_x^n\bigl(u(x,t)-u^s(x,t)\bigr)\bigr|\,(1+|x|^{\lfloor m_0/2\rfloor-2})\,dx\right|<\infty,\quad n\le n_0-m_0-2,
$$
and
$$
\left|\int_{0}^{\infty_s}\bigl|\partial_t\bigl(u(x,t)-u^{s}(x,t)\bigr)\bigr|\,(1+|x|^{\lfloor m_0/2\rfloor-2})\,dx\right|<\infty
$$
where $u^{s}(x,t),~s\in\{r,l\}$ denote the associated finite-gap background solutions.
\end{rmk}
\subsection*{Outline of the paper}
In Section \ref{section:directscattering}, we recall the RH problem associated with the periodic finite-gap potential. Based on this formulation, we construct the Jost solutions for the Lax pair \eqref{Lax:matrix} under the assumption that the initial data satisfy \eqref{initial}. Moreover, we derive the corresponding scattering data and study their properties in Lemma \ref{lem:ab}.  
In Section \ref{section:RHP}, using the scattering data, we introduce an auxiliary function, which enables us to define modified reflection coefficients and to formulate the RH problem associated with the initial data $u_0(x)$. 
Finally, in Section \ref{section:inverse}, incorporating the time evolution of the scattering data, we establish the existence and uniqueness of the solution to the RH problem \ref{RHP:X}, which in turn yields Theorem \ref{thm:inverse}. 
In Appendix \ref{appendix:periodic solution}, we review the periodic solutions, and in Appendix \ref{appendix:Bloch spectrum}, we study the Lax spectrum of the periodic KdV solution.

\section{Direct scattering}\label{section:directscattering}
\subsection*{RH problem for the KdV periodic potential}
Before addressing the direct scattering problem for the initial data~(\ref{initial}), we recall the (quasi-)periodic finite-gap RH problem for the genus-one case.

Consider the following periodic finite-gap potential for the KdV equation~\eqref{eq:KdV}:
\begin{equation}\label{u0}
u(x,t) = \eta_2^2 - \eta_1^2 
- 2\eta_2^2\, 
\mathrm{dn}^2\!\left(
  \eta_2\,[\,x - 2(\eta_1^2 + \eta_2^2)t\,- x_0]  + K(m)\, \middle|\, m
\right).
\end{equation}
where $0 < \eta_1 < \eta_2,m=\frac{\eta_1}{\eta_2}$ and $x_0 \in \mathbb{R}$ is the phase shift. The Lax spectrum of the Lax pair~\eqref{Lax:operator} associated with the genus-one finite-gap potential~\eqref{u0} is then given by
\(
[i\eta_1,\, i\eta_2] \cup [-i\eta_2,\, -i\eta_1].
\)
We then introduce the corresponding two-sheeted genus-one Riemann surface
\begin{equation}\label{surface:genus1}
 \mathscr{X}(\eta_{1},\eta_{2}) 
  := \bigl\{(y;\lambda) \in \mathbb{C}^2 \,\big|\, 
  y^2 = (\lambda^2 + \eta_1^2)(\lambda^2 + \eta_2^2) 
  \bigr\},
\end{equation}
equipped with its canonical homology basis $\mathscr{A},\mathscr{B}$, as illustrated in Figure.~\ref{fig:genus1}.
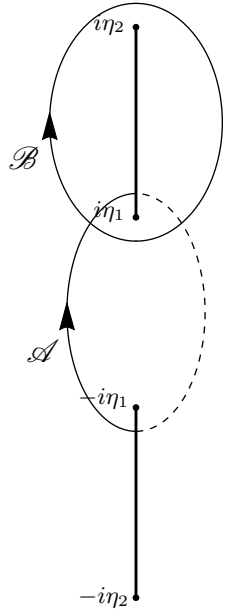
\begin{figure}[H]
    \centering
    \begin{overpic}[width=0.5\linewidth, angle=90]{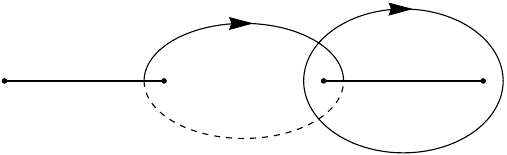}
        \put(-2,40){\large $\mathscr{A}$}  
        \put(-5,72){\large $\mathscr{B}$}  
        \put(9,95){\footnotesize $i\eta_2$}
        \put(9,64){\footnotesize $i\eta_1$}
        \put(6.5,33){\footnotesize $-i\eta_1$}
        \put(6.5,0){\footnotesize $-i\eta_2$}
    \end{overpic}
    \caption{The Riemann surface $\mathscr{X}(\eta_1,\eta_2)$ and its canonical homology basis \(\mathscr{A}, \mathscr{B}\).}
    \label{fig:genus1}
\end{figure}
Denote by $dp$ and $dq$ the quasi-momentum and quasi-energy differentials associated with the Riemann surface $\mathscr{X}(\eta_1,\eta_2)$.  
These can also be regarded as normalized Abelian differentials of the second kind.  
More precisely, introduce the function
\[
R(\lambda;\eta_1,\eta_2) = \sqrt{(\lambda^2+\eta_1^2)(\lambda^2+\eta_2^2)},
\]
and define the differentials $dp$ and $dq$ by
\begin{equation}\label{differential:dpdq}
    dp(\eta_1,\eta_2) = \frac{\zeta^2 + c_{1}}{R(\zeta;\eta_1,\eta_2)}\, d\zeta, 
    \qquad 
    dq(\eta_1,\eta_2) = 12\,\frac{\zeta^4 + \tfrac{\eta_1^2+\eta_2^2}{2}\,\zeta^2 + c_{2}}
    {R(\zeta;\eta_1,\eta_2)}\, d\zeta,
\end{equation}
where the constants $c_{1}$ and $c_{2}$ are determined by the normalization conditions, namely,
$$
    c_{1} = \eta_2^2 - \eta_2^2 \frac{E\!\left(m\right)}{K\!\left(m\right)},
    \qquad 
    c_{2} = \frac{\eta_1^2\eta_2^2}{3} - \frac{\eta_1^2 + \eta_2^2}{6}\,c_{1}(\eta_1,\eta_2).
$$
Here, $m=\frac{\eta_1}{\eta_2},\ K(m)= \int_{0}^{\pi/2} \frac{d\theta}{\sqrt{1 - m^2 \sin^2\theta}}$ and $E(m)=\int_{0}^{\pi/2} \sqrt{1 - m^2 \sin^2\theta}\, d\theta$ denote the complete elliptic integrals of the first and second kind, respectively.

Moreover, define the Abelian integrals $p(\lambda)$ and $q(\lambda)$ associated with 
the differentials \eqref{differential:dpdq} by
\begin{equation}\label{integral:pq}
\begin{aligned}
     p(\lambda;\eta_1,\eta_2) &= \int_{i\eta_2}^{\lambda} dp 
     \ = \lambda + \mathcal{O}(\lambda^{-1}), \qquad &\lambda \to \infty, \\   
     q(\lambda;\eta_1,\eta_2) &= \int_{i\eta_2}^{\lambda} dq 
     \ = 4\lambda^3 + \mathcal{O}(\lambda^{-1}), \qquad &\lambda \to \infty.
\end{aligned}
\end{equation}
In addition, fundamental frequencies $\Omega_1$ and $\Omega_2$ are given by the $\mathscr{B}$-periods of $dp$ and $dq$, respectively:
\begin{equation}\label{Omega:genus1}
    \Omega_1(\eta_1,\eta_2) = \oint_{\mathscr{B}} dp 
    = \frac{\pi \eta_2}{K\!\left( m \right)},
    \qquad 
    \Omega_2(\eta_1,\eta_2) = \oint_\mathscr{B} dq 
    = -\frac{2\pi \eta_2(\eta_1^2+\eta_2^2)}{K\!\left( m \right)}.
\end{equation}

Define the following RH problem associated with the finite gap solution $u(x,t)$ in (\ref{u0}), as in \cite{trogdon_riemannhilbert_2013}. 
Furthermore, in Lemma~\ref{lem:Lax-verify}, we will show that the solution of 
RH problem~\ref{RHP:g1-vector} coincides with the Baker–Akhiezer function, which satisfies the Lax pair \eqref{Lax:matrix}.
\begin{RHP}\label{RHP:g1-vector}
Let 
\(
\Omega := -x\,\Omega_1(\eta_1,\eta_2) - t\,\Omega_2(\eta_1,\eta_2),
\)
where $\Omega_1,\Omega_2$ are the real-valued fundamental frequencies defined in 
\eqref{Omega:genus1}, and let $\Delta=x_0 \Omega_1(\eta_1,\eta_2)\in\R$ denote a constant phase shift. Find a $1\times2$ vector-valued $m(\lambda)$, satisfying the following properties
		\begin{equation}\label{jump:g1-vector}
			\begin{aligned}
				&m(\lambda)\text{ is analytic for } \C\setminus{[-i\eta_2,i\eta_2]},\\
				&m_+(\lambda)=m_-(\lambda)\begin{cases}
					\begin{aligned}
						&\begin{pmatrix}
							0&-i\\
							-i&0
						\end{pmatrix},&&\lambda\in(i\eta_1,i\eta_2),\\
						&\begin{pmatrix}
							0&i\\
							i&0
						\end{pmatrix},&&\lambda\in(-i\eta_2,-i\eta_1),\\
						&e^{i(\Omega+\Delta)\sigma_3},&&\lambda\in(-i\eta_1,i\eta_1),\\
					\end{aligned}
				\end{cases}\\
				&m(\lambda)=\begin{bmatrix}
				    1&1
				\end{bmatrix}+\frac{1}{\lambda}\begin{bmatrix}
				     m_{1,1}(x,t)&\ m_{2,1}(x,t)
				\end{bmatrix}+\mathcal{O}(\lambda^{-2}),\text{ as } \lambda\to\infty,\\
				&m(-\lambda)=m(\lambda)\begin{pmatrix}
					0&1\\
					1&0
				\end{pmatrix},\\
			\end{aligned}
		\end{equation}
	\end{RHP}
Introduce the normalized holomorphic differential $\omega(\eta_1,\eta_2)$ associated with the Riemann surface $\mathscr{X}(\eta_1,\eta_2)$ in \eqref{surface:genus1},
$$
\begin{aligned}
&\omega(\eta_1,\eta_2) := \left( 2\int_{-i\eta_1}^{i\eta_1} \frac{d\zeta}{R(\zeta;\eta_1,\eta_2)} \right)^{-1}
\frac{d\zeta}{R(\zeta;\eta_1,\eta_2)}
= -\frac{i\eta_2}{4K\left(m\right)}\frac{d\zeta}{R(\zeta;\eta_1,\eta_2)},\\
\end{aligned}
$$
and
$$
\oint_{\mathscr{A}} \omega(\eta_1,\eta_2) = 1,
\qquad
\oint_{\mathscr{B}} \omega(\eta_1,\eta_2) = \frac{iK\left(1-m\right)}{2K\left(m\right)}
:= \tau,
$$
where $\Im \tau > 0$, so that $\tau$ is the (genus-1) period matrix element. The corresponding Abelian integral is defined by
\begin{equation}\label{Abelian:J}
J(\lambda;\eta_1,\eta_2) = \int_{i\eta_2}^{\lambda} \omega(\eta_1,\eta_2),
\end{equation}
with the following propositions:
 $$
	\begin{aligned}
		&J_+(\lambda;\eta_1,\eta_2)+J_-(\lambda;\eta_1,\eta_2)=0,
		&&\lambda\in[i\eta_1,i\eta_2],\\
		&J_+(\lambda;\eta_1,\eta_2)-J_-(\lambda;\eta_1,\eta_2)=-\tau, 
		&&\lambda\in[-i\eta_1,i\eta_1],\\
		&J_+(\lambda;\eta_1,\eta_2)+J_-(\lambda;\eta_1,\eta_2)=-1,
		&&\lambda\in[-i\eta_2,-i\eta_1],
	\end{aligned}
	$$
	and
	$$
	\begin{aligned}
		&J_+(i\eta_1;\eta_1,\eta_2)=-\frac{\tau}{2},
		&&J_+(-i\eta_1;\eta_1,\eta_2)=-\frac{\tau}{2}-\frac{1}{2},\\
		&J_+(-i\eta_2;\eta_1,\eta_2)=-\frac{1}{2}, 
		&&J_+(\infty;\eta_1,\eta_2)=-\frac{1}{4}.
	\end{aligned}
	$$
Introduce the Jacobi theta function as
	$$
	\vartheta_3(\lambda;\tau):=\sum_{n\in\mathbb Z}e^{2\pi i n\lambda+\pi i n^2\tau},\ \lambda\in \C,
	$$
and the solution of the RH problem \ref{RHP:g1-vector} is given by
\begin{equation}\label{solution:m}
    m(\lambda)=\frac{\gamma(\lambda;\eta_1,\eta_2)\vartheta_3(0;2\tau)}{\vartheta_3\left(\frac{\Omega+\Delta}{2\pi};2\tau\right)}
			\begin{bmatrix}
				\frac{\vartheta_3\left(2J(\lambda)+\frac{\Omega+\Delta}{2\pi}-\frac{1}{2};2\tau\right)}{\vartheta_3(2J(\lambda)-\frac{1}{2};2\tau)}&
				\frac{\vartheta_3\left(-2J(\lambda)+\frac{\Omega+\Delta}{2\pi}-\frac{1}{2};2\tau\right)}{\vartheta_3(-2J(\lambda)-\frac{1}{2};2\tau)}
			\end{bmatrix},
\end{equation}
where $$\Omega := -x\,\Omega_1(\eta_1,\eta_2) - t\,\Omega_2(\eta_1,\eta_2),\ \gamma(\lambda; \eta_1, \eta_2) := \left( \frac{\lambda^2 + \eta_1^2}{\lambda^2 + \eta_2^2} \right)^{1/4}.$$ 
Moreover, $\gamma(\lambda; \eta_1, \eta_2)$ obeys the following jump conditions:
$$
\begin{aligned}
    &\gamma_+(\lambda; \eta_1, \eta_2) = -i\, \gamma_-(\lambda; \eta_1, \eta_2), \quad \lambda \in (i\eta_1, i\eta_2),\\
    &\gamma_+(\lambda; \eta_1, \eta_2) = i\, \gamma_-(\lambda; \eta_1, \eta_2), \quad \lambda \in (-i\eta_2, -i\eta_1).
\end{aligned}
$$

In the following, we also need to construct a $2\times2$ matrix-valued RH problem. 
\begin{RHP}\label{RHP:genus1}
Let 
\(
\Omega := -x\,\Omega_1(\eta_1,\eta_2) - t\,\Omega_2(\eta_1,\eta_2),
\)
where $\Omega_1,\Omega_2$ are the real-valued fundamental frequencies defined in 
\eqref{Omega:genus1}, and let $\Delta=x_0 \Omega_1(\eta_1,\eta_2)\in\R$ denote a constant phase shift. 
For $\eta_2>\eta_1>0$, seek a $2\times 2$ matrix-valued function $O(x,t;\lambda)$ such  that \cite{girotti_rigorous_2021}:
\begin{enumerate}
    \item $O(x,t;\lambda)$ is analytic for $\lambda\in\C\setminus[-i\eta_2,i\eta_2]$, 
    with at most a simple pole at $\lambda=0$.
    
    \item The boundary values of $O(x,t;\lambda)$ on $[-i\eta_2,i\eta_2]$ satisfy the jump condition
    \[
    O_+(x,t;\lambda) = O_-(x,t;\lambda)\,V^{(O)}(\lambda), \qquad \lambda \in [-i\eta_2,i\eta_2],
    \]
    where
    \[
    V^{(O)}(\lambda)=
    \begin{cases}
        \begin{pmatrix}
            0 & -i \\
            -i & 0
        \end{pmatrix}, & \lambda \in (i\eta_1,i\eta_2), \\[1ex]
        \begin{pmatrix}
            0 & i \\
            i & 0
        \end{pmatrix}, & \lambda \in (-i\eta_2,-i\eta_1), \\[1ex]
        e^{i(\Omega+\Delta)\sigma_3}, & \lambda \in (-i\eta_1,i\eta_1).
    \end{cases}
    \]
    
    \item As $\lambda\to\infty$, $O(x,t;\lambda)$ satisfies the normalization
    \[
    O(x,t;\lambda) = I + \mathcal{O}\!\left(\frac{1}{\lambda}\right).
    \]
    
    \item $O(x,t;\lambda)$ obeys the symmetry relation
    \[ \overline{O(\bar{\lambda};x,t)}=
    O(-\lambda;x,t) = \begin{pmatrix}
        0 & 1 \\
        1 & 0
    \end{pmatrix}O(x,t;\lambda)
    \begin{pmatrix}
        0 & 1 \\
        1 & 0
    \end{pmatrix}.
    \]
\end{enumerate}
\end{RHP}
By the procedure in \cite{girotti_rigorous_2021}, we construct the solution 
$O(x,t;\lambda)$ to the $2\times2$ matrix-valued RH problem~\ref{RHP:genus1}.
Since $x_0 \Omega_1(\eta_1,\eta_2)=\Delta$ and then let
\begin{equation}\label{def:phixt}
    \varphi(x,t;\lambda) := m(\lambda)\, e^{-i((x-x_0)\,p(\lambda) + t\,q(\lambda))}.
\end{equation}
Then the jump conditions for $\varphi$ are independent of $x$ and $t$, 
so that $\partial_x \varphi$ satisfies the same jump conditions as $\varphi$. 
We define 
\[
\begin{aligned}
&\Psi(x,t;\lambda) := 
\begin{pmatrix}
\varphi(x,t;\lambda) \\[1ex]
\partial_x \varphi(x,t;\lambda)
\end{pmatrix}\\
&=
\begin{pmatrix}
m_1(\lambda) & m_2(\lambda) \\[1ex]
\partial_x m_1(\lambda) - i\,p(\lambda)\, m_1(\lambda) & 
\partial_x m_2(\lambda) + i\,p(\lambda)\, m_2(\lambda)
\end{pmatrix}
e^{-i ((x-x_0) p(\lambda)+t q(\lambda)) \sigma_3}.
\end{aligned}
\]
By normalizing $\Psi(x,t;\lambda) e^{i ((x-x_0) p(\lambda)+t q(\lambda)) \sigma_3}$ as $\lambda \to \infty$  to the identity matrix  and factoring out the exponential term, we obtain
\begin{equation}\label{O:genus1}
    O(x,t;\lambda)=\tfrac{1}{2}
    \begin{pmatrix}
        \bigl(1+\tfrac{p(\lambda)}{\lambda}\bigr) m_1(\lambda)
        - \tfrac{1}{i\lambda}\,\partial_x m_1(\lambda) &
        \bigl(1-\tfrac{p(\lambda)}{\lambda}\bigr) m_2(\lambda)
        - \tfrac{1}{i\lambda}\,\partial_x m_2(\lambda) \\[1ex]
        \bigl(1-\tfrac{p(\lambda)}{\lambda}\bigr) m_1(\lambda)
        + \tfrac{1}{i\lambda}\,\partial_x m_1(\lambda) &
        \bigl(1+\tfrac{p(\lambda)}{\lambda}\bigr) m_2(\lambda)
        + \tfrac{1}{i\lambda}\,\partial_x m_2(\lambda)
    \end{pmatrix},
\end{equation}
where $m_1(\lambda)$ and $m_2(\lambda)$ denote the first and second 
components of $m(\lambda)$, respectively. 
In particular, although $O(x,t;\lambda)$ has a simple singularity at $\lambda=0$, 
its determinant is identically equal to $1$, by Lemma~3.3 in \cite{girotti_rigorous_2021}.
\begin{lem}\label{lem:Lax-verify}
The matrix-valued function 
\begin{equation}\label{Psi}
\Psi(x,t;\lambda)
= 
\begin{pmatrix}
1 & 1 \\[0.5ex]
-i\lambda & i\lambda
\end{pmatrix}
O(x,t;\lambda)\, e^{-i[(x-x_0)p(\lambda)+tq(\lambda)]\sigma_3}
\end{equation}
with the matrix 
    $O(x,t;\lambda)$ as in \eqref{O:genus1} and $x_0=\Delta/\Omega_1(\eta_1,\eta_2)$,  with $\Delta=x_0 \Omega_1(\eta_1,\eta_2)$ and $\Omega_1$ defined in \eqref{Omega:genus1},  satisfies the matrix-form Lax pair~\eqref{Lax:matrix}.
\end{lem}
\begin{proof}
Recalling the definitions in \eqref{integral:pq} and \eqref{Omega:genus1}, we find that $\Psi(x,t;\lambda)$ satisfies the jump conditions: $\Psi_+(\lambda) = \Psi_-(\lambda) i\sigma_1$, $\lambda \in (i\eta_1, i\eta_2)$ and $\Psi_+(\lambda) = \Psi_-(\lambda) (-i\sigma_1)$, $\lambda \in (-i\eta_2, -i\eta_1)$. Since these jump matrices are independent of $x$ and $t$, it follows that $\partial_x\Psi(\Psi)^{-1}$ has no jumps for $\lambda \in \mathbb{C}$.

    As $\lambda \to \infty$, recalling the definition of $p(\lambda)$ in (\ref{integral:pq}), we have  
\[
p(\lambda) = \lambda + \sum_{k=1}^{\infty}\frac{p_k}{\lambda^k}, 
\qquad 
p_1 = \left[ \tfrac{\eta_1^2-\eta_2^2}{2} 
+ \eta_2^2 \tfrac{E\!\left(m\right)}{K\!\left(m\right)}  \right].
\]
Moreover, by the definition of $O(x,t,\lambda)$ in (\ref{O:genus1}),  
\begin{equation}\label{expansion:O}
    O(x,t;\lambda) = I + \frac{O^{(1)}}{\lambda} + \frac{O^{(2)}}{\lambda^2} + \mathcal{O}\!\left(\lambda^{-3}\right),
\end{equation}
where
\[
O^{(1)} =
\begin{pmatrix}
m_{1,1} & 0 \\[4pt]
0 & m_{2,1}
\end{pmatrix}, \qquad
O^{(2)} =
\begin{pmatrix}
m_{1,2} + \dfrac{p_1}{2} - \dfrac{\partial_x m_{1,1}}{2i} &
-\dfrac{p_1}{2} - \dfrac{\partial_x m_{2,1}}{2i} \\[10pt]
-\dfrac{p_1}{2} + \dfrac{\partial_x m_{1,1}}{2i} &
m_{2,2} + \dfrac{p_1}{2} + \dfrac{\partial_x m_{2,1}}{2i}
\end{pmatrix}.
\]
Substituting these expansions into $\partial_x\Psi(\Psi)^{-1}$, we obtain
\[
\begin{aligned}
\partial_x\Psi(\Psi)^{-1}
&= 
\begin{pmatrix}
1 & 1 \\ -i\lambda & i\lambda
\end{pmatrix}
\Big[ O_x O^{-1} + O \big( -i p(\lambda)\sigma_3 \big) O^{-1} \Big]
\begin{pmatrix}
1 & 1 \\ -i\lambda & i\lambda
\end{pmatrix}^{-1} \\[6pt]
&=
\begin{pmatrix}
1 & 1 \\ -i\lambda & i\lambda
\end{pmatrix}
\Bigg[
\Big(I + \tfrac{O^{(1)}}{\lambda} + \tfrac{O^{(2)}}{\lambda^2} + \mathcal{O}(\lambda^{-3})\Big)
\Big(-i\lambda\sigma_3 - i\tfrac{p_1}{\lambda}\sigma_3 + \mathcal{O}(\lambda^{-2})\Big) \\[-2pt]
&\quad \times
\Big(I - \tfrac{O^{(1)}}{\lambda} + \tfrac{(O^{(1)})^2 - O^{(2)}}{\lambda^2} + \mathcal{O}(\lambda^{-3})\Big)
+ \tfrac{O^{(1)}_x}{\lambda} + \mathcal{O}(\lambda^{-2})
\Bigg]
\begin{pmatrix}
1 & 1 \\ -i\lambda & i\lambda
\end{pmatrix}^{-1} \\[6pt]
&=
\begin{pmatrix}
1 & 1 \\ -i\lambda & i\lambda
\end{pmatrix}
\Bigg[
-i\lambda \sigma_3 + i[\sigma_3,O^{(1)}] \\
&+ \frac{1}{\lambda}\Big( (O^{(1)})_x- i p_1 \sigma_3-i[O^{(2)},\sigma_3] \Big)
+ \mathcal{O}(\lambda^{-2})
\Bigg]
\begin{pmatrix}
1 & 1 \\ -i\lambda & i\lambda
\end{pmatrix}^{-1} \\[6pt]
&=
\begin{pmatrix}
0 & 1 \\
-\lambda^2 - 2p_1 - 2i\,\partial_x m_{1,1} & 0
\end{pmatrix}
+ \mathcal{O}(\lambda^{-1}).
\end{aligned}
\]
Regarding the limit $\lambda \to 0$, we observe that
\begin{equation}\label{eq:m1m2}
    m_{1,\pm}(0) = m_{2,\pm}(0)\,e^{\pm i(\Omega+\Delta)}, 
    \qquad 
    p_{\pm}(0) = \mp \frac{\Omega_1}{2}.
\end{equation}
It then follows that
\begin{equation}\label{eq:m1m2partial}
\begin{aligned}
    \partial_x m_{1,\pm}(0) &= \partial_x m_{2,\pm} \, e^{\pm i ( \Omega + \Delta)} \mp i \Omega_1 m_{2,\pm} \, e^{\pm i ( \Omega + \Delta)}, \\
    \partial_x^2 m_{1,\pm}(0) &= \partial_x^2 m_{2,\pm} \, e^{\pm i ( \Omega + \Delta)} \mp 2i \Omega_1 \partial_x m_{2,\pm} \, e^{\pm i ( \Omega + \Delta)} - \Omega_1^2 m_{2,\pm} \, e^{\pm i ( \Omega + \Delta)}.
\end{aligned}  
\end{equation}
Then, as $\lambda \to 0$, a direct computation yields
\[
\partial_x \Psi \, (\Psi)^{-1}
= \frac{1}{\lambda}
\begin{pmatrix}
    0 & \alpha_{\pm} \\[0.3em]
    \beta_{\pm} & \partial_x \alpha_{\pm}
\end{pmatrix}
+ \mathcal{O}(1),
\]
where
\[
\begin{aligned}
    \alpha &= 2 i m_1 m_{2x} - 2 i m_{1x} m_2 + 4 m_1 m_2 p, \\
    \beta &= 2 \big( m_{1xx} (-i m_{2x} + m_2 p) + i m_{1x} (m_{2xx} + 4 i m_{2x} p - 3 m_2 p^2) \\
    &\quad + m_1 p (m_{2xx} + 3 i m_{2x} p - 2 m_2 p^2) \big).
\end{aligned}
\]
Using the relations in~\eqref{eq:m1m2} and~\eqref{eq:m1m2partial}, we conclude that $\alpha = \partial_x \alpha = 0$ and $\beta = 0$. Thus, as $\lambda \to 0$, the quantity $\partial_x \Psi \, (\Psi)^{-1}$ is regular.
Combining this with its behavior as $\lambda \to \infty$, we conclude that $\partial_x \Psi \, (\Psi)^{-1}$ is an entire function of $\lambda$. By Liouville's theorem, it must therefore satisfy the spatial part of the Lax pair given in \eqref{Lax:matrix}. 
\end{proof}
\begin{rmk}
In fact, the eigenvector $\varphi(x,t;\lambda)$ in \eqref{def:phixt} can be regarded as the \emph{Baker–Akhiezer function} of the KdV equation~\eqref{eq:KdV} (see~\cite{belokolos_algebro-geometric_nodate}), whose spectrum on the energy plane ($z = \lambda^2$) is given by $(-\infty, -\eta_2^2] \cup [-\eta_1^2, 0]$.
\end{rmk}
\begin{cor}\label{cor:u}
The solution of the KdV equation expressed via the RH problem~\ref{RHP:genus1} is given by
\begin{equation}\label{u:genus-1}
u(x,t) = -2p_1 - 2i\,\partial_x m_{1,1}(x,t),
\end{equation}
where $m_{1,1}$ denotes the coefficient in the large-$\lambda$ expansion of the first entry of 
$m(x,t,\lambda)$ in~\eqref{jump:g1-vector}. In particular,
\begin{equation*}
u(x,t) = \eta_2^2 - \eta_1^2
- 2\eta_2^2\, 
\mathrm{dn}^2\!\left(
  \eta_2\big[x - 2(\eta_1^2 + \eta_2^2)t - x_0\big] + K(m)\, \middle|\, m
\right),
\end{equation*}
with $m=\frac{\eta_1}{\eta_2}$.
\end{cor}
\begin{proof}
    As $\lambda \to \infty$, we have the expansion
    \[
        2J(\lambda; \eta_1, \eta_2)
        = -\frac{1}{2}
        + \frac{1}{\lambda}\!\left[\frac{i\eta_2}{4K(m)}\right]
        + \mathcal{O}(\lambda^{-2}).
    \]
    Combining this with $\Omega = -x\Omega_1 - t\Omega_2$, it follows that the large-$\lambda$ expansion of $m_1(\lambda)$ in~\eqref{solution:m} yields
    \[
        m_{1,1} = -i\,\partial_x \log \vartheta_3\!\left(\frac{\Omega + \Delta}{2\pi};\, 2\tau\right).
    \]
    Substituting $p_1$ and the above expansion into~\eqref{u:genus-1}, we obtain
    \begin{equation*}
        u(x,t)
        = \eta_2^2 - \eta_1^2
        - 2\eta_2^2 \frac{E(m)}{K(m)}
        - 2\,\partial_x^2
        \log \vartheta_3\!\left(\frac{\Omega + \Delta}{2\pi};\, 2\tau\right).
    \end{equation*}
    Moreover, using $2\tau = i\,\dfrac{K(1-m)}{K(m)}$ together with the identity given in~\cite[p.~45, Exercise~16]{lawden_elliptic_1989} and (3.5.5), we have
    \[
        \partial_x^2 \log \vartheta_3\!\left(\frac{\Omega + \Delta}{2\pi};\, 2\tau\right)
        = -\eta_2^2 \frac{E(m)}{K(m)}
        + \dn^2\!\left(
            2K(m)\!\left(\frac{\Omega + \Delta}{2\pi}\right)
            + K(m)\, \bigg|\, m
        \right).
    \]
    Consequently,
    \[
        u(x,t)
        = \eta_2^2 - \eta_1^2
        - 2\eta_2^2\,
        \dn^2\!\left(
            \eta_2\!\big[x - 2(\eta_1^2 + \eta_2^2)t - x_0\big]
            + K(m)\, \middle|\, m
        \right).
    \]
\end{proof}

\subsection*{Jost solution for step-like finite-gap potentials}
Recall that both $u_0^l(x)$ and $u_0^r(x)$ in \eqref{initial:ulur} are the periodic finite-gap solutions of the KdV equation associated with the spectral bands
\[
\Sigma^l = (i\eta_1^l, i\eta_2^l) \cup (-i\eta_2^l, -i\eta_1^l), \quad 
\Sigma^r = (i\eta_1^r, i\eta_2^r) \cup (-i\eta_2^r, -i\eta_1^r).
\]
In what follows, we denote by $O^{l}(x,t;\lambda)$ and $O^{r}(x,t;\lambda)$ the RH problems associated with $u^{l}(x,t)$ and $u^{r}(x,t)$, respectively. Let $s \in \{r, l\}$. For simplicity, we use $s$ to denote either case $r$ or $l$, and introduce the following abbreviations for various related functions and quantities:
\[
dp^{s} = dp(\eta_1^{s}, \eta_2^{s}), \qquad 
p^{s}(\lambda) = p(\lambda; \eta_1^{s}, \eta_2^{s}) = \int_{i\eta_2^{s}}^{\lambda} dp^{s},
\]
and similarly for $\Omega_1^{s}$ and other related quantities.

We introduce the Jost solutions associated with the period background solution $u^{l}(x,t)$ and $u^{r}(x,t)$ as follows:
\begin{equation}\label{def:Psi0}
\begin{aligned}
\Psi_0^{s}(x,t;\lambda) &:= \begin{pmatrix}
1 & 1 \\[0.5ex]
-i\lambda & i\lambda
\end{pmatrix}
O^{s}(x,t;\lambda)\, e^{-i[(x-x_0^l)p^{s}(\lambda)+tq^{s}(\lambda)]\sigma_3}. \\
\end{aligned}
\end{equation}
By Lemma~\ref{lem:Lax-verify}, the functions $\Psi_0^{l}(x,t;\lambda)$ and $\Psi_0^{r}(x,t;\lambda)$ satisfy the matrix Lax pair~\eqref{Lax:matrix} with potentials $u^{l}(x,t)$ and $u^{r}(x,t)$ in \eqref{initial:ulur}, respectively. In the following analysis, we focus on the $x$–part of the Lax pair~\eqref{Lax:matrix} while treating $t=0$ as a fixed parameter. The Jost solutions corresponding to the initial data~\eqref{initial} satisfy the asymptotic boundary conditions
\begin{equation}\label{Jost:asymp-x}
\begin{aligned}
\Psi^l(x;\lambda) &= \Psi_0^{l}(x;\lambda)\left(I + \mathcal{O}(|x|^{-1})\right), \quad x\to+\infty, \\
\Psi^r(x;\lambda) &= \Psi_0^{r}(x;\lambda)\left(I + \mathcal{O}(|x|^{-1})\right), \quad x\to-\infty.
\end{aligned}
\end{equation}
Suppose the Jost solution $\Psi^s(x;\lambda)$ takes the form
\begin{equation}\label{Jost}
\begin{aligned}
\Psi^s(x;\lambda) &:= \Psi_0^{s}(x;\lambda) \,
e^{i (x-x_0^s) p^{s}(\lambda)\sigma_3} \,
\bJ^{s}(x;\lambda) \,
e^{-i (x-x_0^s) p^{s}(\lambda)\sigma_3}. \\
\end{aligned}
\end{equation}
\subsection*{The normalized eigenfunctions $\bJ^l(x;\lambda)$ and $\bJ^r(x;\lambda)$.}
\begin{pro}\label{prop:laxequationforJ}
Let $\Psi^l(x;\lambda)$ and $\Psi^r(x;\lambda)$ be the Jost solutions of the Lax pair~\eqref{Lax:matrix}. Then $\bJ^{l}(x;\lambda)$ and $\bJ^{r}(x;\lambda)$ satisfy the differential equations:
\begin{equation}\label{eq:J}
\begin{aligned}
&\partial_x\bJ^{s}(x;\lambda) + i p^{s}(\lambda)\,[\sigma_3, \bJ^{s}(x;\lambda)] = U^{s}(x;\lambda) \bJ^{s}(x;\lambda), \\
\end{aligned}
\end{equation}
where 
\[
\begin{aligned}
U^{s}(x;\lambda) &= (O^{s}(x;\lambda))^{-1} \left[\frac{i (u_0(x) - u_0^{s}(x))}{2\lambda}
\begin{pmatrix}
1 & 1 \\[2pt]
-1 & -1
\end{pmatrix} \right] O^{s}(x;\lambda), \\ 
\end{aligned}
\]
Moreover, $\bJ^{l}(x;\lambda)$ and $\bJ^{r}(x;\lambda)$ admit the Volterra integral representations:
\begin{equation}\label{eq:j-Integral}
\begin{aligned}
\bJ^{l}(x;\lambda) &= I + \int_{-\infty}^x 
e^{-i(x-y) p^{l}(\lambda){\sigma}_3} \,
U^{l}(y;\lambda) \,
\bJ^{l}(y;\lambda) e^{i(x-y) p^{l}(\lambda){\sigma}_3}  dy, \\
\bJ^{r}(x;\lambda) &= I + \int_{+\infty}^x 
e^{-i(x-y) p^{r}(\lambda){\sigma}_3} \,
U^{r}(y;\lambda) \,
\bJ^{r}(y;\lambda) e^{i(x-y) p^{r}(\lambda){\sigma}_3}  dy.
\end{aligned}
\end{equation}
\end{pro}
\begin{proof}
Differentiating \eqref{Jost} with respect to $x$, we obtain
\[
\partial_x \Psi^{l}(x;\lambda) = L^{l} \Psi^{l}(x;\lambda) 
+ \Psi_0^{l}(x;\lambda) e^{i (x-x_0^l) p^{l}(\lambda)\sigma_3}
\Big( \partial_x \bJ^{l}(x;\lambda) + i p^{l}(\lambda)[\sigma_3,\bJ^{l}(x;\lambda)] \Big)
e^{-i (x-x_0^l) p^{l}(\lambda)\sigma_3},
\]
where $[A,B] := AB - BA$ denotes the matrix commutator and
\(
L^{l} = \begin{pmatrix}
0 & 1 \\
-\lambda^2 + u_0^l(x) & 0
\end{pmatrix}.
\)
Since $\Psi^l(x;\lambda)$ is a Jost solution of the Lax pair~\eqref{Lax:matrix}, it must satisfy the $x$-part of the system. This implies that $\bJ^{l}(x;\lambda)$ satisfies the differential equation~\eqref{eq:J}.

Moreover, in view of the boundary condition~\eqref{Jost:asymp-x}, $\bJ^{l}(x;\lambda)$ can be expressed in the Volterra integral form~\eqref{eq:j-Integral}. The analysis for $\bJ^{r}(x;\lambda)$ follows analogously.
\end{proof}

Denote the subscript ``1'' for bands on the upper half-plane and ``2'' for bands on the lower half-plane, that is,
\[
\Sigma_1^s:=(i\eta_1^s,i\eta_2^s),
 \quad \Sigma_2^s:=(-i\eta_2^s,-i\eta_1^s).
\]
Since we require that $\Sigma_1^l$ and $\Sigma_1^r$ do not share any endpoints, there are six possible configurations in total.
Figure~\ref{fig:bands} depicts the six cases where $\eta_1^l \neq \eta_1^r$ and $\eta_2^l \neq \eta_2^r$.

\begin{figure}[H]
    \centering
    \begin{minipage}{0.15\textwidth}
        \includegraphics[width=\textwidth]{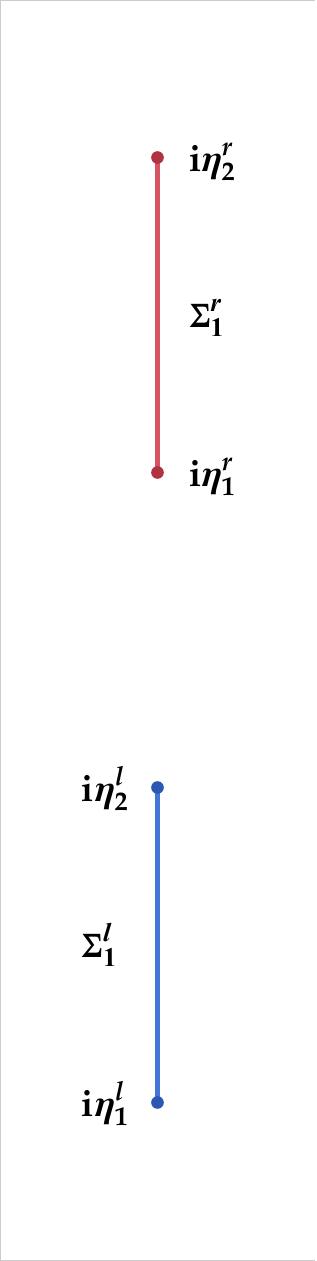}
         \centering
        \textrm{(i)}
    \end{minipage}
    \hfill
    \begin{minipage}{0.15\textwidth}
        \includegraphics[width=\textwidth]{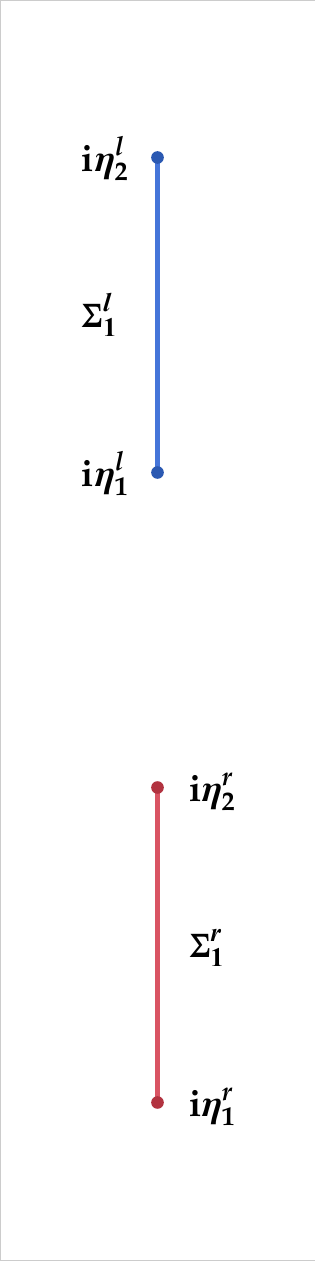}
         \centering
        \textrm{(ii)}
    \end{minipage}
    \hfill
    \begin{minipage}{0.15\textwidth}
        \includegraphics[width=\textwidth]{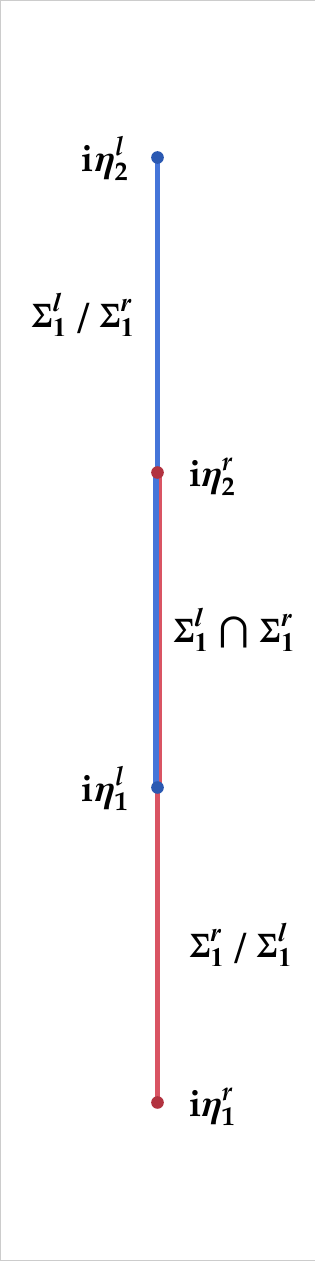}
        \centering
        \textrm{(iii)}
    \end{minipage}
    \hfill
    \begin{minipage}{0.15\textwidth}
        \includegraphics[width=\textwidth]{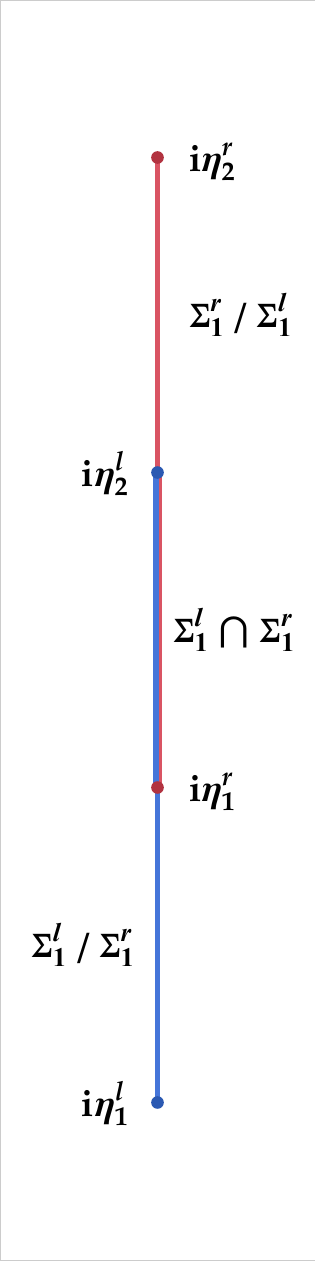}
        \centering
        \textrm{(iv)}
    \end{minipage}
    \hfill
    \begin{minipage}{0.15\textwidth}
        \includegraphics[width=\textwidth]{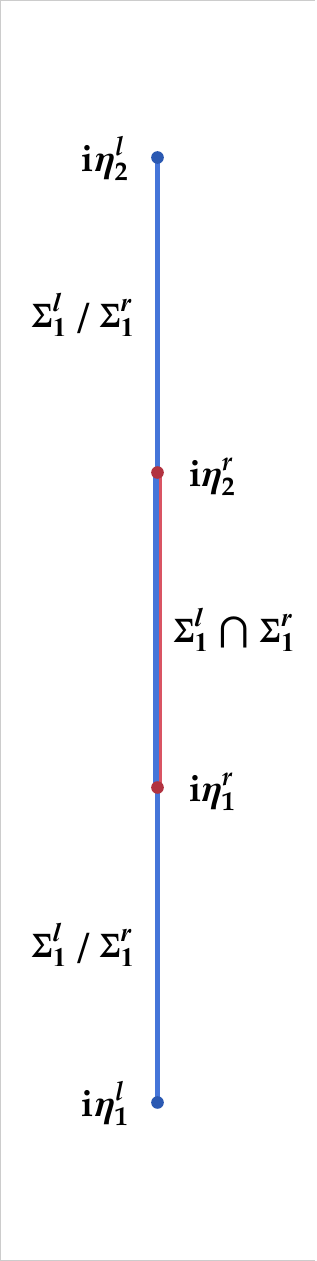}
        \centering
        \textrm{(v)}
    \end{minipage}
    \hfill
    \begin{minipage}{0.15\textwidth}
        \includegraphics[width=\textwidth]{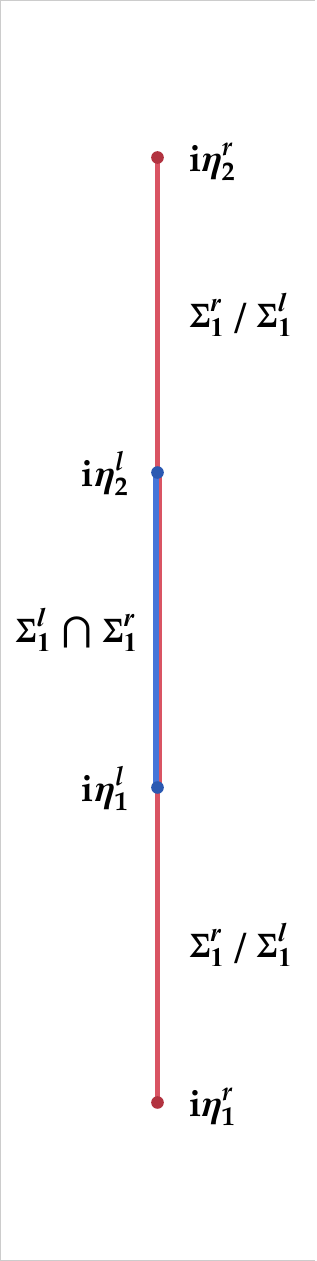}
        \centering
        \textrm{(vi)}
    \end{minipage}
    \caption{\footnotesize The six configurations (i)--(vi) of step-like finite-gap potential bands on the upper half plane. The red ($\Sigma_1^r$) and blue ($\Sigma_1^l$) solid lines represent the spectral bands of $u_0^r(x)$ and $u_0^l(x)$, respectively.}
    \label{fig:bands}
\end{figure}

\begin{pro}\label{prop:J}
   Suppose that the initial data $u_0(x)$ satisfies the condition \eqref{initial} and assumption \ref{assumption}. Then the integral equations \eqref{eq:j-Integral} admits a unique solution $\mathbf{J}^{l}(x;\lambda)$  and $\mathbf{J}^{r}(x;\lambda)$, respectively, which enjoy the following properties:
\begin{enumerate}
    \item The function $\mathbf{J}^{l}(x;\lambda)$ is defined for $x \in \mathbb{R}$ and 
$\lambda \in \big(\overline{\mathbb{C}}_+ \cup \Sigma^l_2, \;\overline{\mathbb{C}}_- \cup \Sigma^l_1\big) \setminus \{0,\pm i\eta_1^l,\pm i\eta_2^l\}$, 
meaning that the first column of $\mathbf{J}^{l}(x;\lambda)$ lies in $\big(\overline{\mathbb{C}}_+ \cup \Sigma^l_2\big)\setminus \{0,\pm i\eta_1^l,\pm i\eta_2^l\}$ 
and the second column lies in $\big(\overline{\mathbb{C}}_- \cup \Sigma^l_1\big)\setminus \{0,\pm i\eta_1^l,\pm i\eta_2^l\}$. For each $\lambda \in \big(\overline{\mathbb{C}}_+ \cup \Sigma^l_2, \;\overline{\mathbb{C}}_- \cup \Sigma^l_1\big) \setminus \{0,\pm i\eta_1^l,\pm i\eta_2^l\}$, $\bJ^l(\cdot,\lambda)$ is $n_0$–times differentiable for any $x\in\mathbb{R}$.

\item For each $x \in \mathbb{R}$, the function $\mathbf{J}^{l}(x,\cdot)$ is continuous for 
$\lambda \in (\overline{\mathbb{C}}_+, \overline{\mathbb{C}}_-) \setminus [-i\eta_2^l, i\eta_2^l]$ 
and analytic in the corresponding interior regions. Moreover, $\mathbf{J}^{l}(x,\cdot)$ is $m_0$ differentiable for the left side and the right side $\lambda \in (\Sigma^l\cup\R) \setminus \{0,\pm i\eta_1^l,\pm i\eta_2^l\}$. In particular, the boundary values of $\mathbf{J}^{l}(x;\lambda)$ for $\lambda \in \Sigma^{l}$ satisfy

    \begin{equation}\label{jumps:Jl-band}
        \mathbf{J}^{l}(x;\lambda_+) = \sigma_1 \mathbf{J}^{l}(x;\lambda_-) \sigma_1, \quad \lambda \in \Sigma^l,
    \end{equation}
    and
    \begin{equation}\label{jumps:Jl-gap}
    \begin{aligned}
        &\bJ^l_1(x;\lambda_+) = \begin{pmatrix}
            1 & 0 \\
            0 & e^{-2i(x\Omega_1^l - \Delta^l)}
        \end{pmatrix} \bJ^l_1(x;\lambda_-), && \lambda \in (0, i\eta_1^l), \\
        &\bJ^l_2(x;\lambda_+) = \begin{pmatrix}
            e^{2i(x\Omega_1^l - \Delta^l)} & 0 \\
            0 & 1
        \end{pmatrix} \bJ^l_2(x;\lambda_-), && \lambda \in (-i\eta_1^l, 0),
    \end{aligned}
    \end{equation}
where $\bJ^l_i(x,\lambda_{\pm})$, $i=1,2$, denote the boundary values (from the left/right side) of the $i$th column of the matrix-valued function $\bJ^l(x;\lambda)$.

    \item $\mathbf{J}^{r}(x;\lambda)$ is defined for $x \in \mathbb{R}$ and $\lambda \in (\overline{\mathbb{C}}_-\cup\Sigma_1^r, \overline{\mathbb{C}}_+\cup\Sigma_2^r) \setminus \{0,\pm i\eta_1^r,\pm i\eta_2^r\}$. For each $\lambda \in (\overline{\mathbb{C}}_-\cup\Sigma_1^r, \overline{\mathbb{C}}_+\cup\Sigma_2^r) \setminus \{0,\pm i\eta_1^r,\pm i\eta_2^r\}$, $\mathbf{J}^{r}(\cdot;\lambda)$ is $n_0$–times differentiable for any $x\in\mathbb{R}$.
    
    \item For each $x \in \mathbb{R}$, the function $\mathbf{J}^{r}(x,\cdot)$ is continuous on $\lambda \in (\overline{\mathbb{C}}_-, \overline{\mathbb{C}}_+ )\setminus[-i\eta_2^r,i\eta_2^r]$ and analytic in their interior regions. Moreover, $\mathbf{J}^{r}(x,\cdot)$ is $m_0$ differentiable for the left side and the right side $\lambda \in (\Sigma^r\cup\R) \setminus \{0,\pm i\eta_1^r,\pm i\eta_2^r\}$. In particular, the boundary values of $\mathbf{J}^{r}(x;\lambda)$ for $\lambda \in \Sigma^{r}$ satisfy
    \begin{equation}
        \mathbf{J}^{r}(x;\lambda_+) = \sigma_1 \mathbf{J}^{r}(x;\lambda_-) \sigma_1, \quad \lambda \in \Sigma^l,
    \end{equation}
    and
    \begin{equation}
    \begin{aligned}
        &\bJ^r_1(x;\lambda_+) = \begin{pmatrix}
            1 & 0 \\
            0 & e^{-2i(x\Omega_1^r - \Delta^r)}
        \end{pmatrix} \bJ^r_1(x;\lambda_-), && \lambda \in (-i\eta_1^r,0), \\
        &\bJ^r_2(x;\lambda_+) = \begin{pmatrix}
            e^{2i(x\Omega_1^r - \Delta^r)} & 0 \\
            0 & 1
        \end{pmatrix} \bJ^r_2(x;\lambda_-), && \lambda \in (0,i\eta_1^r).
    \end{aligned}
    \end{equation}
    \item The functions $\bJ^l(x;\lambda)$ and $\bJ^r(x;\lambda)$ satisfy the following symmetry relations:
    \begin{equation}\label{sym:Jminus}
    \begin{aligned}
    &\bJ^l(x;-\lambda)=\sigma_1\bJ^l(x;\lambda)\sigma_1,&&\lambda \in (\overline{\mathbb{C}}_+, \overline{\mathbb{C}}_- )\setminus[-i\eta_2^l,i\eta_2^l],\\
    &\bJ^r(x;-\lambda)=\sigma_1\bJ^r(x;\lambda)\sigma_1,&&\lambda \in (\overline{\mathbb{C}}_-, \overline{\mathbb{C}}_+ )\setminus[-i\eta_2^r,i\eta_2^r].\\
    \end{aligned}
    \end{equation}
     Moreover, the following complex conjugation symmetries hold:
    \begin{equation}\label{sym:Jconjugate}
    \begin{aligned}
    &\overline{\bJ^l(x,\bar\lambda)}=\sigma_1\bJ^l(x;\lambda)\sigma_1,&&\lambda \in (\overline{\mathbb{C}}_+, \overline{\mathbb{C}}_- )\setminus\{0\},\\
    &\overline{\bJ^r(x,\bar\lambda)}=\sigma_1\bJ^r(x;\lambda)\sigma_1,&&\lambda \in (\overline{\mathbb{C}}_-, \overline{\mathbb{C}}_+ )\setminus\{0\}.\\
    \end{aligned}
    \end{equation}
\item For each $x \in \R$ and for every $\lambda$ belonging to the domains of $\bJ^l(x;\lambda)$ and $\bJ^r(x;\lambda)$, the determinants of $\bJ^l(x;\lambda)$ and $\bJ^r(x;\lambda)$ satisfy
\begin{equation}\label{eq:detJ}
    \det\!\big[\bJ^s(x;\lambda)\big] = 1.
\end{equation}
\end{enumerate}
\end{pro}
\begin{proof}
Before proving the proposition for $\bJ^s(x;\lambda)$, we define the following norms for a matrix-valued function $A(x)$:
\begin{equation*}
|A|
:= \sqrt{\operatorname{Tr}((A^*)^tA)}, 
\quad
\|A\|_{L^p(\gamma)}
:= \left(\int_\gamma |A(x)|^p\,dx\right)^{\frac{1}{p}},
\quad
\|A\|_{L^{\infty}(\gamma)}
:= \sup_{x\in\gamma} |A(x)|.
\end{equation*}

\noindent{\bf The proof of 1.} We begin by recalling the definition of $p^{l}(\lambda)$ from \eqref{integral:pq}. The function $p^{l}(\lambda)$ is analytic for $\lambda \in \mathbb{C} \setminus [-i\eta_2^l,i\eta_2^l]$. Since the imaginary part $\Im p^{l}(\lambda)$ is harmonic and vanishes on $\mathbb{R} \cup \Sigma^{l}$ and $\Im p_{\pm}^{l}(\lambda)>0,\lambda\in(0,i\eta_1^l],\Im p_{\pm}^{l}(\lambda)<0,\lambda\in[-i\eta_1^l,0)$, and since $p^{l}(\lambda) \sim \lambda$ as $\lambda \to \infty$, the maximum principle for harmonic functions yields
\begin{equation}\label{sign:pl}
    \begin{aligned}
        \Im p^{l}(\lambda) &> 0, \quad \lambda \in \mathbb{C}_+ \setminus \Sigma_1^{l}, \\
        \Im p^{l}(\lambda) &< 0, \quad \lambda \in \mathbb{C}_- \setminus \Sigma_2^{l}.
    \end{aligned}
\end{equation}

 Suppose that the initial data $u_0$ satisfies the conditions~\eqref{initial}.
By the Volterra integral equation~\eqref{eq:j-Integral}, the first column of 
$\bJ^l(x;\lambda)$ can be written in the following form:
\begin{equation}\label{eq:Jl1}
\bJ_1^l(x;\lambda)
=
\begin{pmatrix} 1 \\ 0 \end{pmatrix}
+
\int_{-\infty}^{x}
\begin{pmatrix}
1 & 0 \\
0 & e^{2i(x-y)p^l(\lambda)}
\end{pmatrix}
U^l(y;\lambda)\,\bJ_1^l(y;\lambda)\,dy .
\end{equation}

Define
\[
\mathcal{K}(x,y;\lambda)
:=
\begin{pmatrix}
1 & 0 \\
0 & e^{2i(x-y)p^l(\lambda)}
\end{pmatrix}
U^l(y;\lambda),
\]
which induces an integral operator acting on $2\times1$ vector-valued functions of $y$.
For convenience, we rewrite~\eqref{eq:Jl1} in the operator form
\begin{equation}\label{eq:J1l-operator}
    \bJ_1^l = \mathbf{e}_1 + \mathcal{K}\bJ_1^l,
\qquad
\mathbf{e}_1 := \begin{pmatrix} 1 \\ 0 \end{pmatrix}.
\end{equation}

Then by $\Im p^l(\lambda)=0,\ \lambda\in\R\cup\Sigma^l$ and \eqref{sign:pl}, it follows that for any $\lambda\in\overline{\mathbb{C}}_+ \cup \Sigma^l_2\setminus\{0,\pm i\eta_1^l,\pm i\eta_2^l\}$, the operator $\mathcal{K}$ is a bounded linear operator from $L^1{(-\infty,x)}$ to $L^1{(-\infty,x)}$, that is
\begin{equation}\label{ineq:K}
    \|\mathcal{K}\|\le\frac{c}{|\lambda|}\|u_0(x)-u_0^l(x)\|_{L^1(-\infty,x)},\quad \|\mathcal{K}^j\|\le\frac{1}{j!}\left(\frac{c}{|\lambda|}\|u_0(x)-u_0^l(x)\|_{L^1(-\infty,x)}\right)^j.
\end{equation}
Thus, the equation \eqref{eq:J1l-operator} admits a unique solution for any $x\in\R$ and $\lambda\in\overline{\mathbb{C}}_+ \cup \Sigma^l_2\setminus\{0,\pm i\eta_1^l,\pm i\eta_2^l\}$. Moreover, for any fixed $\lambda \in \overline{\mathbb{C}}_+ \cup \Sigma_2^l \setminus \{0,\pm i\eta_1^l,\pm i\eta_2^l\}$,
taking derivatives with respect to $x$ in~\eqref{eq:Jl1} shows that
$\bJ_1^l(x;\lambda)$ is $n_0$-times differentiable for all $x \in \mathbb{R}$.

\noindent{\bf The proof of 2.} 
Notice that for any $x \in \mathbb{R}$ and $\lambda \in \overline{\mathbb{C}}_+ \cup \Sigma_2^l \setminus \{0,\pm i\eta_1^l,\pm i\eta_2^l\}$, the kernel $\partial_\lambda \mathcal{K}(x,y;\lambda)$ satisfies the estimate
\[
\left\|\dot{\mathcal{K}}(x,y;\cdot)\right\|
\le \frac{c}{|\lambda|}
\bigl\|\, x\bigl(u_0(x)-u_0^l(x)\bigr)\bigr\|_{L^1(-\infty,x)} .
\]
where the dot denotes differentiation with respect to $\lambda$.
Thus for the initial data $u_0(x)$ satisfies the condition \eqref{initial}, it follows that $\mathbf{J}_1^{l}(x,\cdot)$ is analytic for $\lambda\in\bar{\C}_{+}\setminus[i\eta_2^l,0]$.
Furthermore, since $\Im p^l(\lambda)=0$ for $\lambda\in\Sigma^l\cup\R$, it follows that
\[
\left\|\frac {d^j}{d\lambda^j}\mathcal{K}(x,t;\cdot)\right\|
\le \frac{c}{|\lambda|}
\bigl\|\, x^j\bigl(u_0(x)-u_0^l(x)\bigr)\bigr\|_{L^1(-\infty,x)},\quad \lambda\in\Sigma^l\cup\R\setminus\{0,\pm i\eta_1^l,\pm i\eta_2^l\}.
\]
Consequently, for the initial data $u_0(x)$ satisfies the condition \eqref{initial}, $\mathbf{J}^{l}(x,\cdot)$ is $m_0$ differentiable for the left side and the right side $\lambda \in (\Sigma^l\cup\R) \setminus \{0,\pm i\eta_1^l,\pm i\eta_2^l\}$.

For $\lambda \in \Sigma^l$, the integral equation gives
\[
    \mathbf{J}^{l}(x;\lambda_+)
    = I + \int_{-\infty}^x 
        e^{-i(x-y) p^{l}(\lambda_+) \sigma_3} \,
        U^{l}(y;\lambda_+) \,
        \mathbf{J}^{l}(y;\lambda_+) e^{i(x-y) p^{l}(\lambda_+) \sigma_3}  dy.
\]
Using the relation $p^{l}(\lambda_+) = -p^{l}(\lambda_-)$ and the jump conditions for $O^{l}(x;\lambda)$ in RH Problem~\ref{RHP:genus1} (for $\lambda \in \Sigma_1^l$), we obtain
\[
    \mathbf{J}^{l}(x;\lambda_+)
    = I + \int_{-\infty}^x 
        e^{i(x-y) p^{l}(\lambda_-) \sigma_3} \,
        (i\sigma_1) U^{l}(y;\lambda_-) (-i\sigma_1) \,
        \mathbf{J}^{l}(y;\lambda_+) e^{-i(x-y) p^{l}(\lambda_-) \sigma_3}  dy.
\]
Multiplying both sides on the left by $i\sigma_1$ and on the right by $-i\sigma_1$ leads to
\[
    (-i\sigma_1) \mathbf{J}^{l}(x;\lambda_+) (i\sigma_1)
    = I + \int_{-\infty}^x 
        e^{-i(x-y) p^{l}(\lambda_-) \sigma_3} \,
        U^{l}(y;\lambda_-) \,
        \left[(-i\sigma_1) \mathbf{J}^{l}(y;\lambda_+) (i\sigma_1)\right] e^{i(x-y) p^{l}(\lambda_-) \sigma_3}  dy.
\]
This identity shows that $(-i\sigma_1) \mathbf{J}^{l}(x;\lambda_+) (i\sigma_1)$ also satisfies the integral equation \eqref{eq:j-Integral} for $\lambda \in \Sigma_1^l$, with $\lambda$ interpreted as the right boundary value.

We now examine the jump conditions in the gap. Note that
\[
p^{l}(\lambda_+) - p^{l}(\lambda_-) = -\Omega_1^l, 
\quad \text{and} \quad 
O^{l}(x;\lambda_+) = O^{l}(x;\lambda_-) e^{-i(x\Omega^l_1 - \Delta^l)\sigma_3}.
\]
From these, we derive
\[
\scalebox{0.95}{$
\begin{aligned}
\bJ_1^l(x;\lambda_+) 
&= \begin{pmatrix} 1 \\ 0 \end{pmatrix} 
+ \int_{-\infty}^{x} 
\begin{pmatrix} 1 & 0 \\ 0 & e^{2i(x-y)p^l(\lambda_+)} \end{pmatrix} 
U^l(y;\lambda_+) \mathbf{J}_1^l(y;\lambda_+)  dy \\
&= \begin{pmatrix} 1 \\ 0 \end{pmatrix} 
+ \int_{-\infty}^{x} 
\begin{pmatrix} 1 & 0 \\ 0 & e^{2i(x-y)p^l(\lambda_-) - 2i(x-y)\Omega_1^l} \end{pmatrix} 
 e^{i(y\Omega_1^l - \Delta^l)\sigma_3} U^l(y;\lambda_+) e^{-i(y\Omega_1^l - \Delta^l)\sigma_3} {\bJ}_1^l(y;\lambda_+)  dy,
\end{aligned}
$}\]
and consequently,
\[
\scalebox{0.925}{$
\begin{aligned}
&\begin{pmatrix} 1 & 0 \\ 0 & e^{-2i(x\Omega_1^l - \Delta^l)} \end{pmatrix} \bJ_1^l(x;\lambda_+) 
= \begin{pmatrix} 1 \\ 0 \end{pmatrix} 
+ \int_{-\infty}^{x} 
\begin{pmatrix} 1 & 0 \\ 0 & e^{2i(x-y)p^l(\lambda_-)} \end{pmatrix} 
U^l(y;\lambda_-)  \begin{pmatrix} 1 & 0 \\ 0 & e^{-2i(y\Omega_1^l - \Delta^l)} \end{pmatrix} {\bJ}_1^l(y;\lambda_-)  dy.
\end{aligned}
$}
\]
The remaining relations follow by similar arguments.

\noindent{\bf The proof of 5.} Regarding the symmetry properties, observe that
\[
p^l(-\lambda) = -p^l(\lambda), \quad \lambda \in \mathbb{C} \setminus [-i\eta_2^l, i\eta_2^l],
\]
and combining this with the symmetry relation in the RH Problem~\ref{RHP:genus1}, we obtain
\[
\partial_x \left( \sigma_1 \mathbf{J}^{l}(x;-\lambda) \sigma_1 \right) + i p^{l}(\lambda) \left[ \sigma_3, \sigma_1 \mathbf{J}^{l}(x;-\lambda) \sigma_1 \right] = U^{l}(x;\lambda) \left( \sigma_1 \mathbf{J}^{l}(x;-\lambda) \sigma_1 \right).
\]
By the uniqueness of the solution to \eqref{eq:J}, we conclude the symmetry relation for $\mathbf{J}^{l}(x;\lambda)$. Similarly, using the Schwarz symmetry $\overline{p(\bar{\lambda})} = p(\lambda)$, we also derive the symmetry relation \eqref{sym:Jconjugate}.

\noindent{\bf The proof of 6.} Recalling that $\det[O^l(x;\lambda)] \equiv 1$ for any $x \in \R$ and $\lambda \in \C$, it follows from~\eqref{Psi} that
\[
\det[\Psi_0^l(x;\lambda)] = 2i\lambda.
\]
Since $\Psi^l(x;\lambda)$ is a Jost solution of the Lax pair~\ref{Lax:matrix}, Liouville’s formula implies that 
$\det[\Psi^l(x;\lambda)]$ is constant with respect to $x$. 
As $x \to -\infty$, we have $\Psi^l(x;\lambda) \to \Psi_0^l(x;\lambda)$; therefore,
\[
\det[\Psi^l(x;\lambda)] = \det[\Psi_0^l(x;\lambda)] = 2i\lambda.
\]
By the definition in~\eqref{Jost}, it then immediately follows that 
\(
\det[\bJ^l(x;\lambda)] = 1.
\)
\end{proof}

\begin{lem}[Asymptotic properties of $\mathbf{J}^{l}(x;\lambda)$ and $\mathbf{J}^{r}(x;\lambda)$ as $\lambda\to\infty$ and $\lambda\to0$]\label{lem:inftyofJ}
Suppose that the initial data $u_0(x)$ satisfies the condition \eqref{initial} and assumption \ref{assumption}. 
\begin{enumerate}
    \item As $\lambda\to\infty$, the matrices $\mathbf{J}^{l}(x;\lambda)$ and $\mathbf{J}^{r}(x;\lambda)$ admit the asymptotic expansions
\begin{equation}\label{expansion:Jinfty}
\begin{aligned}
&\mathbf{J}^{l}(x;\lambda) = I + \sum_{k=1}^{3}\frac{J_k^l(x)}{\lambda^k} + \mathcal{O}\!\left(\lambda^{-4}\right), 
&& \text{with }\ J_1^l(x) = \frac{i}{2} \left( \int_{-\infty}^{x} \left[u_0(y) - u_0^{l}(y)\right]  dy \right) \begin{pmatrix} 1 & 0 \\ 0 & -1 \end{pmatrix}, \\
&\mathbf{J}^{r}(x;\lambda) = I + \sum_{k=1}^{3}\frac{J_{k}^r(x)}{\lambda^k} + \mathcal{O}\!\left(\lambda^{-4}\right),
&&\text{with }\ J_1^r(x) = \frac{i}{2} \left( \int_{+\infty}^{x} \left[u_0(y) - u_0^{r}(y)\right]  dy \right) \begin{pmatrix} 1 & 0 \\ 0 & -1 \end{pmatrix}.
\end{aligned}
\end{equation}
Here, $J_k^s$ satisfy the recursion equation \eqref{eq:J-recursion}.
\item  As $\lambda \to 0$, the functions $\mathbf{J}^{l}(x;\lambda)$ and $\mathbf{J}^{r}(x;\lambda)$ have at most a simple pole. In particular, the asymptotic behavior is given by
\begin{equation}\label{expansion:Jzero}
\begin{aligned}
\mathbf{J}^s(x;\lambda_{\pm}) &= \frac{1}{\lambda} e^{-i(x - x_0^s)p_{\pm}^s(0)\sigma_3} 
\begin{pmatrix}
    \alpha^s(x) & \beta^s(x) \\
    -\alpha^s(x) & -\beta^s(x)
\end{pmatrix} 
e^{i(x - x_0^s)p_{\pm}^s(0)\sigma_3} + \mathcal{O}(1),\ s\in\{r,l\}, \\
\end{aligned}
\end{equation}
where $\alpha^l(x)$ and $\beta^l(x)$ are continuous functions decaying to zero as $x \to -\infty$, and $\alpha^r(x)$ and $\beta^r(x)$ are continuous functions decaying to zero as $x \to +\infty$.
\end{enumerate}

\end{lem}
\begin{proof}
\noindent{\bf The proof of 1.} 
Consider the formal expansion of $\bJ^{s}(x;\lambda)$ together with the asymptotic expansions of
$p^{s}(\lambda)$ and $U^{s}(x;\lambda)$ as $\lambda\to\infty$:
\begin{equation*}
\bJ_{\mathrm{form}}^{s}
:= I + \sum_{k=1}^{\infty}\frac{J_k^{s}}{\lambda^{k}}, \qquad
p^{s}(\lambda)
:= \lambda + \sum_{k=1}^{\infty}\frac{p_k^{s}}{\lambda^{k}}, \qquad
U^{s}(x;\lambda)
:= \sum_{k=1}^{\infty}\frac{U_k^{s}}{\lambda^{k}} .
\end{equation*}
Substituting these expansions into~\eqref{eq:J} yields the following recursion relations:
\begin{equation}
\label{eq:J-recursion}
\begin{cases}
\begin{aligned}
\partial_x \bigl(J_k^{s}\bigr)^{(d)}
&= \left(\displaystyle\sum_{l=1}^{k} U_l^{s} J_{k-l}^{s}\right)^{(d)}, \\[1ex]
i\,[\sigma_3, J_k^{s}]
&= - \bigl(\partial_x J_{k-1}^{s}\bigr)^{(o)}
    - \displaystyle\sum_{l=1}^{k} p_l^{s}\,[\sigma_3, J_{k-l}^{s}]
    + \left(\displaystyle\sum_{l=1}^{k} U_l^{s} J_{k-l}^{s}\right)^{(o)},
\end{aligned}
\end{cases}
\end{equation}
where $(\cdot)^{(d)}$ and $(\cdot)^{(o)}$ denote the diagonal and off-diagonal parts of a matrix,
respectively, and $J_0^{s}=I$.

Define the truncated expansion
\[
\bJ^{s}_{(k)}:=I+\sum_{l=1}^{k}\frac{J^s_l}{\lambda^l}.
\]
A direct computation shows that
\begin{align*}
\partial_x\bJ^s_{(4)}+ip^{s}[\sigma_3,\bJ^s_{(4)}]-U^s\bJ^s_{(4)}
&=\frac{\partial_xJ^s_{4}}{\lambda^4}
 +\sum_{k=4}^{\infty}\frac{1}{\lambda^{k+1}}
 \sum_{l=1}^{4}p^s_{k+1-l}[\sigma_3,J^s_l] 
 -\sum_{k=4}^{\infty}\frac{1}{\lambda^{k+1}}
 \sum_{l=1}^{4}U_{k+1-l}J^s_l .
\end{align*}

By induction, one obtains the estimate
\(
|J_4^s|
\le
\sup_{0\le j\le 2}\bigl|\partial_x^j(u_0-u_0^s)\bigr|.
\)
Moreover, if the initial data satisfy condition~\eqref{initial} and $|\lambda|\gg1$,
then $J^s_{(4)}$ is invertible, with
\[
(\bJ^s_{(4)})^{-1}
=
\sum_{k=0}^{\infty}
\left(-\sum_{l=1}^4\frac{J_{l}^s}{\lambda^l}\right)^k .
\]

Define
\[
\Delta
:=
-\left(
\partial_x\bJ^s_{(4)}
+ip^{s}[\sigma_3,\bJ^s_{(4)}]
-U^s\bJ^s_{(4)}
\right)(\bJ^s_{(4)})^{-1}.
\]
Then, for $|\lambda|$ sufficiently large, $(\bJ^s_{(4)})^{-1}J^s$ satisfies
\begin{equation*}
\partial_x\!\left((\bJ^s_{(4)})^{-1}\bJ^s\right)
+ip^s[\sigma_3,(\bJ^s_{(4)})^{-1}\bJ^s]
=(\bJ^s_{(4)})^{-1}\Delta \bJ^s .
\end{equation*}
Equivalently, $\bJ^s$ solves the following integral equation:
\begin{equation}
\label{eq:J-integral-expansion}
\begin{aligned}
\bJ^s(x;\lambda)
&=\bJ_{(4)}^s(x;\lambda)
+\int_{\pm\infty}^x
\bJ_{(4)}^s(x;\lambda)
e^{-i(x-y)p^s(\lambda)\sigma_3}
(\bJ^s_{(4)}(y;\lambda))^{-1}
\Delta(y;\lambda) \bJ^s(y;\lambda)
e^{i(x-y)p^s(\lambda)\sigma_3}\,dy .
\end{aligned}
\end{equation}

Since
\(
|\Delta|
\le
\frac{\sup_{0\le j\le 3}|\partial_x^j(u_0-u_0^s)|}{|\lambda|^4},
\)
and thus by the inegral equation \eqref{eq:J-integral-expansion}, it follows that, as $\lambda\to\infty$,
\begin{equation*}
|\bJ^s(x;\lambda)-\bJ_{(3)}^s(x;\lambda)|
\le
\frac{
\sup_{0\le j\le 3}
\|\partial_x^j(u_0-u_0^s)\|_{L^1(x,\pm\infty)}
}{|\lambda|^4}.
\end{equation*}

\noindent{\bf The proof of 2.}   As $\lambda \to 0$, we consider specifically the case of $\bJ^l(x;\lambda_+)$; the analysis for other cases is similar. Define
\[
\mathcal{X}(x;\lambda) = \begin{pmatrix}
    1 & 1 \\
    -i \lambda & i \lambda
\end{pmatrix} e^{i(x - x_0^l)p^l(\lambda)\sigma_3} \bJ^l(x;\lambda) e^{-i(x - x_0^l)p^l(\lambda)\sigma_3}.
\]
Then, from the integral equation \eqref{eq:j-Integral}, it follows that
\[
\mathcal{X}(x;\lambda) =  \begin{pmatrix}
    1 & 1 \\
    -i \lambda & i \lambda
\end{pmatrix} + \int_{-\infty}^x \tilde{U}(y;\lambda) \mathcal{X}(y;\lambda)  dy,
\]
where
\[
\tilde{U}(y;\lambda) = \begin{pmatrix}
    1 & 1 \\
    -i \lambda & i \lambda
\end{pmatrix} e^{i(y - x_0^l)p(\lambda)\sigma_3} U(y;\lambda) e^{-i(y - x_0^l)p(\lambda)\sigma_3} \begin{pmatrix}
    1 & 1 \\
    -i \lambda & i \lambda
\end{pmatrix}^{-1}.
\]

Suppose that, as $\lambda \to 0$, the vector-valued solution of the RH problem
\eqref{RHP:g1-vector} and the phase function $p^l(\lambda)$ admit the following
expansions:
\[
\begin{aligned}
    m_1^l(\lambda_+)
    &= m_{1,+}^l(0)
       + \partial_\lambda m_{1,+}^l(0)\,\lambda
       + \mathcal{O}(\lambda^2), \\
    m_2^l(\lambda_+)
    &= m_{2,+}^l(0)
       + \partial_\lambda m_{2,+}^l(0)\,\lambda
       + \mathcal{O}(\lambda^2), \\
    p^l(\lambda_+)
    &= -\frac{\Omega_1^l}{2}
       + p_{1,+}^l(0)\,\lambda
       + \mathcal{O}(\lambda^2),
\end{aligned}
\]
where $m_j^l(x;\lambda)$, $j=1,2$, denote the first and second columns of the
solution to the RH problem \eqref{RHP:g1-vector} associated with the initial data
$u_0^l(x)$. Substituting these expansions into $\tilde{U}(y;\lambda)$ yields
\[
\tilde{U}(y;\lambda_+)
=
\begin{pmatrix}
    \mathcal{O}\!\left((u_0-u_0^l)y^2\right)
    & \mathcal{O}\!\left((u_0-u_0^l)y^2\right) \\
    \mathcal{O}\!\left(u_0-u_0^l\right)
    & \mathcal{O}\!\left((u_0-u_0^l)y\right)
\end{pmatrix}.
\]
Furthermore, as $\lambda\to0$, we also have
$$
\scalebox{0.87}{$
\partial_{\lambda}\tilde{U}(y;\lambda_+)
=
\begin{pmatrix}
    \mathcal{O}\!\left((u_0-u_0^l)y^4\right)
    & \mathcal{O}\!\left((u_0-u_0^l)y^4\right) \\
    \mathcal{O}\!\left((u_0-u_0^l)y\right)
    & \mathcal{O}\!\left((u_0-u_0^l)y^2\right)
\end{pmatrix},\ 
\partial_{\lambda}^2\tilde{U}(y;\lambda_+)
=
\begin{pmatrix}
    \mathcal{O}\!\left((u_0-u_0^l)y^5\right)
    & \mathcal{O}\!\left((u_0-u_0^l)y^5\right) \\
    \mathcal{O}\!\left((u_0-u_0^l)y^2\right)
    & \mathcal{O}\!\left((u_0-u_0^l)y^3\right)
\end{pmatrix}.
$}
$$

As a consequence, for initial data $u_0(x)$ satisfying the condition
\eqref{initial}, the matrix $\mathcal{X}(x;\lambda_+)$ is regular at $\lambda=0$ and admits the following expansion:
\begin{equation*}
    \mathcal{X}(x;\lambda_+)=\mathcal{X}(x;0_+)+\partial_{\lambda}\mathcal{X}(x;0_+)\lambda+\mathcal{O}(\lambda^2).
\end{equation*}
As $\lambda \to 0$, tracking back the transformation of $\bJ^l(x;\lambda)$, we have
\[
\begin{aligned}
\bJ^l(x;\lambda_+)
&=
e^{-i(x-x_0)p_+^l(0)\sigma_3}
\left[
\begin{pmatrix}
    1 & 1 \\
    -i\lambda & i\lambda
\end{pmatrix}^{-1}
\mathcal{X}(x;\lambda)
\right]
e^{i(x-x_0)p_+^l(0)\sigma_3} \\
&=
\frac{1}{\lambda}
e^{-i(x-x_0)p_+^l(0)\sigma_3}
\begin{pmatrix}
    \alpha^l(x) & \beta^l(x) \\
    -\alpha^l(x) & -\beta^l(x)
\end{pmatrix}
e^{i(x-x_0)p_+^l(0)\sigma_3}
+ \mathcal{O}(1),
\end{aligned}
\]
where $\alpha^l(x)$ and $\beta^l(x)$ are continuous functions that decay to zero as
$x \to -\infty$.

\end{proof}
\subsection*{The Jost solutions $\Phi^l(x;\lambda)$ and $\Phi^r(x;\lambda)$.}
We have introduced the Jost solutions $\Psi^l(x;\lambda)$ and $\Psi^r(x;\lambda)$ in \eqref{Jost}, which satisfy the Lax pair \eqref{Lax:matrix}. To construct the RH problem corresponding to the initial data \eqref{initial}, we now define the transformation
\[
\Phi^s(x;\lambda) := \begin{pmatrix}
1 & 1 \\
-i\lambda & i\lambda
\end{pmatrix} \Psi^s(x;\lambda),
\]
which converts the Lax pair \eqref{Lax:matrix} into the new form
\begin{equation}\label{Lax:matrix-new}
\partial_x \Phi^s = \left[
-i\lambda\sigma_3 + \frac{i u_0(x)}{2\lambda}
\begin{pmatrix}
1 & 1 \\
-1 & -1
\end{pmatrix}
\right] \Phi^s := \tilde{L} \Phi^s.
\end{equation}
Correspondingly, we define
\begin{equation}\label{Jost:Phi}
\begin{aligned}
\Phi^{s}(x;\lambda) &:= \begin{pmatrix}
1 & 1 \\
-i\lambda & i\lambda
\end{pmatrix}^{-1} \Psi^s(x;\lambda)
= O^{s}(x;\lambda) \mathbf{J}^{s}(x;\lambda) e^{-i(x - x^{s}_0)p^{s}(\lambda)\sigma_3}, \\
\end{aligned}
\end{equation}
and note that these functions satisfy the transformed Lax equation \eqref{Lax:matrix-new} with the initial datum $u_0(x)$ given in \eqref{initial}.
\begin{pro}\label{prop:Phi}
Suppose that $u_0 $ satisfies the condition \eqref{initial} and assumption \ref{assumption}. Then $\Phi^l(x;\lambda)$ and $\Phi^r(x;\lambda)$ are Jost solutions of the Lax pair \eqref{Lax:matrix-new} with initial data \eqref{initial} and possess the following properties:
\begin{enumerate}
    \item For each $x\in\mathbb{R}$, $\Phi^l(x;\lambda)$ is defined on
$(\overline{\mathbb{C}}_+\cup\Sigma_2^l)\cup(\overline{\mathbb{C}}_-\cup\Sigma_1^l)\setminus\{0,\pm i\eta_1^l,\pm i\eta_2^l\}$,
and is analytic in the interior of these regions. Moreover, $\Phi^l(x;\lambda)$ admits continuous boundary values on both sides of
$\Sigma_1^l\cup\Sigma_2^l$ away from $\{\pm i\eta_1^l,\pm i\eta_2^l\}$. Moreover, these boundary values are $m_0$-times continuously differentiable.

    \item For each $x\in\mathbb{R}$, $\Phi^r(x;\lambda)$ is defined on
$(\overline{\mathbb{C}}_-\cup\Sigma_1^r)\cup(\overline{\mathbb{C}}_+\cup\Sigma_2^r)\setminus\{0,\pm i\eta_1^r,\pm i\eta_2^r\}$,
and is analytic in the interior of these regions. Moreover, $\Phi^r(x;\lambda)$ admits continuous boundary values on both sides of
$\Sigma_1^r\cup\Sigma_2^r$ away from $\{\pm i\eta_1^r,\pm i\eta_2^r\}$ and these boundary values are $m_0$-times continuously differentiable.

In particular, the boundary values of $\Phi^s(x;\lambda)$ for $\lambda \in \Sigma^s$ satisfy
    \begin{equation}\label{jump:Phir}
        \begin{aligned}
            \Phi^s(x;\lambda_+) &= \Phi^s(x;\lambda_-)(-i\sigma_1), && \lambda \in \Sigma_1^s, \\
            \Phi^s(x;\lambda_+) &= \Phi^s(x;\lambda_-) i\sigma_1, && \lambda \in \Sigma_2^s.
        \end{aligned}
    \end{equation}
   \item The Jost solutions $\Phi^l(x;\lambda)$ and $\Phi^r(x;\lambda)$ satisfy the following symmetry conditions:
\begin{equation}\label{sym:Phi-pm}
\begin{aligned}
&\Phi^l(x;-\lambda) = \sigma_1 \Phi^l(x;\lambda) \sigma_1, 
&& \lambda \in (\overline{\mathbb{C}}_+, \overline{\mathbb{C}}_-) \setminus [-i\eta_2^l, i\eta_2^l], \\
&\Phi^r(x;-\lambda) = \sigma_1 \Phi^r(x;\lambda) \sigma_1, 
&& \lambda \in (\overline{\mathbb{C}}_-, \overline{\mathbb{C}}_+) \setminus [-i\eta_2^r, i\eta_2^r].
\end{aligned}
\end{equation}
Furthermore, the following complex-conjugation symmetries hold:
\begin{equation}\label{sym:Phi-conjuate}
\begin{aligned}
&\overline{\Phi^l(x,\bar{\lambda})} = \sigma_1 \Phi^l(x;\lambda) \sigma_1, 
&& \lambda \in (\overline{\mathbb{C}}_+, \overline{\mathbb{C}}_-) \setminus \{0\}, \\
&\overline{\Phi^r(x,\bar{\lambda})} = \sigma_1 \Phi^r(x;\lambda) \sigma_1, 
&& \lambda \in (\overline{\mathbb{C}}_-, \overline{\mathbb{C}}_+) \setminus \{0\}.
\end{aligned}
\end{equation}
\item For each $x \in \R$ and for every $\lambda$ belonging to the domains of $\Phi^l(x;\lambda)$ and $\Phi^r(x;\lambda)$, the determinants of $\Phi^l(x;\lambda)$ and $\Phi^r(x;\lambda)$ satisfy
\begin{equation}\label{eq:detPhi}
   \det\!\big[\Phi^s(x;\lambda)\big] = 1.
\end{equation}
\end{enumerate}
\end{pro}
\begin{proof}
Recall the properties of $\mathbf{J}^l(x;\lambda)$ from Proposition~\ref{prop:J}, which determine the domain, continuation, and analytic properties of $\Phi^l(x;\lambda)$.
Combining the jump relations of $\mathbf{J}^l(x;\lambda)$ in~\eqref{jumps:Jl-band} with the jump matrix from RH Problem~\ref{RHP:genus1}, we obtain, for $\lambda \in \Sigma_1^l$,
\begin{equation*}
\begin{aligned}
\Phi^{l}(x,\lambda_+) 
&= O^{l}(x,\lambda_+) \, \mathbf{J}^l(x,\lambda_+) \, e^{-i(x-x_0^l)p^l(\lambda_+)\sigma_3} \\
&= O^{l}(x,\lambda_-) \, (-i\sigma_1) \, \sigma_1 \mathbf{J}^l(x,\lambda_-)\sigma_1 \, e^{i(x-x_0^l)p^l(\lambda_-)\sigma_3} \\
&= \Phi^{l}(x,\lambda_-) \, (-i\sigma_1).
\end{aligned}
\end{equation*}
For $\lambda \in [0,i\eta_1^l]$, combining the jump condition~\eqref{jumps:Jl-gap} with RH Problem~\ref{RHP:genus1}, we have
\begin{equation*}
\begin{aligned}
\Phi_1^{l}(x,\lambda_+)
&= O^{l}(x,\lambda_-) \,
e^{-i(x\Omega_1^l-\Delta^l)\sigma_3}
\begin{pmatrix}
1 & 0 \\[2pt]
0 & e^{-2i(x\Omega_1^l-\Delta^l)}
\end{pmatrix}
\mathbf{J}_1^l(x,\lambda_-)
e^{-i(x-x_0^l)p^l(\lambda_-)+i(x\Omega_1^l-\Delta^l)} \\
&= \Phi_1^{l}(x,\lambda_-).
\end{aligned}
\end{equation*}
The remaining jump relations can be derived in an analogous manner. 

Furthermore, the symmetry conditions follow directly from the symmetry properties of $\mathbf{J}^l(x;\lambda)$ and $\mathbf{J}^r(x;\lambda)$ stated in~\eqref{sym:Jminus} and~\eqref{sym:Jconjugate}.
In addition, the determinant relations for $\Phi^l(x;\lambda)$ and $\Phi^r(x;\lambda)$ are obtained straightforwardly from those of $\bJ^l(x;\lambda)$ and $\bJ^r(x;\lambda)$, together with~\eqref{eq:detJ} and~\eqref{Jost:Phi}.
\end{proof}

\begin{lem}[Asymptotic properties of $\Phi^{l}(x;\lambda)$ and $\Phi^{r}(x;\lambda)$ as $\lambda\to\infty$]\label{lem:asymptotic}
Suppose that $u_0 $ satisfies the condition \eqref{initial} and assumption \ref{assumption}. 
Then, as $\lambda \to \infty$, the matrices $\Phi^{l}(x;\lambda)$ and $\Phi^{r}(x;\lambda)$ admit the asymptotic expansions
\begin{subequations}\label{expansion:Phi}
\begin{align}
\Phi^{s}(x;\lambda) 
&= \left[I + \sum_{j=1}^{3}\frac{\Gamma^s_j(x)}{\lambda^j} + \mathcal{O}\!\left(\lambda^{-4}\right)\right]e^{-i(x-x_0^s)\lambda\sigma_3},&&\lambda\to\infty.
\end{align}
\text{with}
\begin{align}
\Gamma^l_1(x)
&= \left(
m_{1,1}^l(x)
- i(x - x_0^l)p_1^l
+ \frac{i}{2} \int_{-\infty}^{x} \!\!\big[u_0(y) - u_0^{l}(y)\big]\, dy
\right)
\begin{pmatrix} 1 & 0 \\[2pt] 0 & -1 \end{pmatrix}, \\[4pt]
\Gamma^r_1(x)
&= \left(
m_{1,1}^r(x)
- i(x - x_0^r)p_1^r
+ \frac{i}{2} \int_{+\infty}^{x} \!\!\big[u_0(y) - u_0^{r}(y)\big]\, dy
\right)
\begin{pmatrix} 1 & 0 \\[2pt] 0 & -1 \end{pmatrix}.
\end{align}
\end{subequations}
Here $m_{1,1}^l(x)$ and $m_{1,1}^r(x)$ denote the coefficients in the expansions of 
$m_1^{l}(x;\lambda)$ and $m_1^{r}(x;\lambda)$ in \eqref{jump:g1-vector}, respectively, of the solutions to RH problem~\ref{RHP:g1-vector} associated with $u_0^l(x)$ and $u_0^r(x)$, respectively.
The quantities $p_1^l$ and $p_1^r$ are the coefficients in the large-$\lambda$ expansion of the quasimomenta $p^l(\lambda)$ and $p^r(\lambda)$, given by
\[
p_1^s
= \frac{(\eta_1^s)^2 - (\eta_2^s)^2}{2} 
+ (\eta_2^s)^2 \frac{E\!\left(m^s\right)}{K\!\left(m^s\right)},\ m^s=\frac{\eta_1^s}{\eta_2^s}.
\]
In particular, the step-like potential $u_0(x)$ can be expressed as
\begin{equation}\label{eq:u0}
   u_0(x) =-2i\,\partial_x [ \Gamma_1^s(x)]_{11}.
\end{equation}
where $[A]_{11}$ denotes the $(1,1)$ entry of the matrix~$A$.
\end{lem}
\begin{proof}
Rewrite $\Phi^s(x;\lambda)$ as $\Phi^s(x;\lambda):=\Gamma^s(x;\lambda)e^{-i(x-x_0^s)p^s(\lambda)\sigma_3}$, and since $\Phi^s(x;\lambda)$ satisfies the Lax pair \eqref{Lax:matrix-new}, it follows that 
\begin{equation}\label{eq:gamma}
 \partial_x\Gamma^s(x;\lambda)-ip^s(\lambda)\Gamma^s(x;\lambda)\sigma_3=\left[
-i\lambda\sigma_3 + \frac{i u_0(x)}{2\lambda}
\begin{pmatrix}
1 & 1 \\
-1 & -1
\end{pmatrix}
\right]\Gamma^s(x;\lambda).  
\end{equation}

Suppose that $\Gamma^s(x;\lambda):=I+\sum_{k=1}^\infty\frac{\Gamma_k^s}{\lambda^k}$ and $p^s(\lambda)=\lambda+\frac{p_j^s}{\lambda^k}$, and plugging them into \eqref{eq:gamma}, it follows that
\begin{equation}\label{eq:gamma-exp}
\partial_x\Gamma_k^s+i[\sigma_3,\Gamma_{k+1}^s]-i\sum_{l=1}^kp_l ^s\Gamma_{k-l}^s= \frac{i u_0(x)}{2\lambda}
\begin{pmatrix}
1 & 1 \\
-1 & -1
\end{pmatrix}\Gamma_{k-1}^s,~k\ge1.
\end{equation}
Collecting the expansion of $\mathbf{J}^l(x;\lambda)$ from~\eqref{expansion:Jinfty} and that of $O^l(x;\lambda)$ from~\eqref{expansion:O}, together with the expansion of the quasi-momentum
\[
p^l(\lambda) = \lambda + \frac{p_1^l}{\lambda} + \mathcal{O}\!\left(\lambda^{-2}\right), \qquad
p_1^l = \frac{(\eta_1^l)^2 - (\eta_2^l)^2}{2} 
+ (\eta_2^l)^2 \frac{E\!\left(\frac{\eta_1^l}{\eta_2^l}\right)}{K\!\left(\frac{\eta_1^l}{\eta_2^l}\right)}.
\]
By the lemma \ref{lem:inftyofJ} and the definition of $\Phi^s(x;\lambda)$ in \eqref{Jost:Phi}, we obtain \eqref{expansion:Phi} and  as $\lambda \to \infty$,
\[
\Gamma_1^l(x)=(O^l)^{(1)} + \bJ^l_1 - i(x - x_0^l)p_1^l\sigma_3.
\]
Finally, by Corollary~\ref{cor:u}, the recovery formula for $u_0(x)$ given in~\eqref{eq:u0} follows.
\end{proof}
\begin{lem}[Asymptotic properties of $\Phi^{l}(x;\lambda)$ and $\Phi^{r}(x;\lambda)$ as $\lambda \to 0$]\label{lem:Phiat0}
Suppose that $u_0 $ satisfies the condition \eqref{initial}.
As $\lambda \to 0$, the Jost solutions $\Phi^l(x;\lambda)$ and $\Phi^r(x;\lambda)$ may have at most a simple pole.  
In particular, if we left-multiply $\Phi^{l}(x;\lambda)$ and $\Phi^{r}(x;\lambda)$ by the row vector $ \begin{bmatrix} 1 & 1 \end{bmatrix} $, then the resulting vectors
\[
\phi^s(x;\lambda) := \begin{bmatrix} 1 & 1 \end{bmatrix}  \Phi^s(x;\lambda), 
\]
are regular at $\lambda = 0$.
\end{lem}
\begin{proof}
Notice that, as $\lambda \to 0$,
\[
\begin{bmatrix} 1 & 1 \end{bmatrix}  O^l(x;\lambda_{\pm}) e^{-i(x-x_0^l)p^l_{\pm}(0)}
= \begin{pmatrix} m^l_{1,\pm}(0) & m^l_{2,\pm}(0) \end{pmatrix} e^{\pm (x\Omega_1 - \Delta)\sigma_3 / 2}.
\]
Recalling the relationship~\eqref{eq:m1m2}, it follows that
\[
\begin{bmatrix} 1 & 1 \end{bmatrix}  O^l(x;\lambda_{\pm}) e^{-i(x-x_0^l)p^l_{\pm}(0)}
= m_{2,\pm}(0) \, e^{\mp (x\Omega_1 - \Delta)/2} \begin{bmatrix} 1 & 1 \end{bmatrix} .
\]

Substituting the above equation into 
\(\begin{bmatrix} 1 & 1 \end{bmatrix}  \Phi^l(x;\lambda_{\pm})\) in~\eqref{Jost:Phi} 
and combining it with the asymptotic expansion of $\mathbf{J}^r(x;\lambda)$
in~\eqref{expansion:Jzero}, we conclude that the vector $\phi^l(x;\lambda)$ is regular at $\lambda = 0$.  
The proof for $\phi^r(x;\lambda)$ follows by the same argument.
\end{proof}

\begin{lem}
[Local behaviour of $\Phi^{l}(x;\lambda)$ and $\Phi^{r}(x;\lambda)$ at the endpoints]
\label{lem:Phi-endpoints}
Suppose that the initial datum $u_0$ satisfies the condition~\eqref{initial}.
As $\lambda \to \pm i\eta_j^l$, $j=1,2$, the Jost solutions $\Phi^{l}(x;\lambda)$ may exhibit at most a fourth--root singularity.
Similarly, as $\lambda \to \pm i\eta_j^r$, $j=1,2$, the Jost solutions $\Phi^{r}(x;\lambda)$ may exhibit at most a fourth--root singularity.
\end{lem}
\begin{proof}
    By recalling the definitions of $\Psi^l(x;\lambda)$ in~\eqref{Jost} and $\Psi_0^l(x;\lambda)$ in~\eqref{def:Psi0}, we obtain the following integral equation for $\Psi^l(x;\lambda)$:
\begin{equation}\label{eq:Phil}
    \Psi^l(x;\lambda)=\Psi_0^l(x;\lambda)+\int_{-\infty}^x\Psi_0^l(x;\lambda)(\Psi_0^l(y;\lambda))^{-1}\Delta\tilde U(y)\Psi^l(y;\lambda)dy,
\end{equation}
where $\Delta \tilde{U}(x):=\begin{pmatrix}
    0&0\\
    u_0^l(x)-u_0(x)&0
\end{pmatrix}$. 
Since $\Psi_0^l(x;\lambda)$ has only constant jump conditions on $\Sigma^l$, 
$\Psi_0^l(x;\lambda)\bigl(\Psi_0^l(y;\lambda)\bigr)^{-1}$ is continuous along $\Sigma^l$ and has at most square--root singularities at the endpoints $\pm i\eta_j^l$, $j=1,2$. 
In particular, it can be continuously extended to each endpoint.

In particular, as $\lambda\to i\eta_2$, we have the following expansion of $m^l(x;\lambda)$ in \eqref{solution:m}
\begin{equation}\label{eq:expansion of eta2}
    \begin{aligned}
    &m_1^l(x;\lambda)=(\lambda-i\eta_2^l)^{-\frac{1}{4}}(\sum_{k=0}^{\infty}\gamma_k(\lambda-i\eta_2^l)^k)\left(\sum_{j=0}^{\infty}S^{(1)}_j(\lambda-i\eta_2^l)^j+\sqrt{\lambda-i\eta_2^l}\sum_{j=0}^{\infty}S^{(2)}_j(\lambda-i\eta_2^l)^j\right),\\
    &m_2^l(x;\lambda)=(\lambda-i\eta_2^l)^{-\frac{1}{4}}(\sum_{k=0}^{\infty}\gamma_k(\lambda-i\eta_2^l)^k)\left(\sum_{j=0}^{\infty}S^{(1)}_j(\lambda-i\eta_2^l)^j-\sqrt{\lambda-i\eta_2^l}\sum_{j=0}^{\infty}S^{(2)}_j(\lambda-i\eta_2^l)^j\right),\\
\end{aligned}
\end{equation}
where $\gamma_k$ is the expansion of $\gamma^l(\lambda)$ and $S^{(1)}_j,S^{(2)}_j$ are the coefficients of the Puiseux expansion of the $\vartheta_3$ in \eqref{solution:m}. On the other hand, the quasi-momentum integral $p^l(\lambda)$ has the following expansion:
$$
p^l(\lambda)=2\sqrt{\lambda-i\eta_2^l}\sum_{k=0}^{\infty}p_k(\lambda-i\eta_2^l).
$$

Substituting the above expansion into the kernel, we obtain
\[
\Psi_0^l(x;\lambda)\bigl(\Psi_0^l(y;\lambda)\bigr)^{-1} \,\Delta \tilde U(y)
= \mathcal{O}\Bigl((1+|x|+|y|)\,|u_0-u_0^l|\Bigr)
+ \mathcal{O}\Bigl(\sqrt{\lambda-i\eta_2^l}\Bigr).
\]

Since the initial datum $u_0(x)$ satisfies the condition~\eqref{initial}, the Neumann series then implies
\begin{equation*}
    |\Psi^l(x;\lambda)| \le 
    |\Psi_0^l(x;\lambda)|\,
    \exp\Bigl(C (1+|x|)\, \sup_{0\le j\le1}\|x^j(u_0-u_0^l)\|_{L^1(-\infty,x)}\Bigr).
\end{equation*}

Hence, $\Psi^l(x;\lambda)$ exhibits the same local behaviour as $\Psi_0^l(x;\lambda)$. 
Since $\Phi^l(x;\lambda)$ is obtained from $\Psi^l(x;\lambda)$ via a linear transformation, it inherits the same local behaviour as well. The analysis for the other endpoints is entirely analogous. 
\end{proof}

\subsection*{Scattering data}
Notice that $\Phi^l_j(x;\lambda)$ and $\Phi^r_j(x;\lambda)$ ($j=1,2$) can be regarded as eigenvectors of equation~\eqref{Lax:matrix-new}, subject to different boundary conditions (\ref{expansion:Phi}). 
By Proposition~\ref{prop:Phi}, there exist four eigenvectors when $\lambda \in \mathbb{R} \cup (\Sigma^{l} \cap \Sigma^{r})$, 
three eigenvectors when $\lambda \in \Sigma^{r} \setminus \Sigma^{l}$ or $\lambda \in \Sigma^{l} \setminus \Sigma^{r}$, 
and generically two eigenvectors elsewhere. 
Assuming that $u_0 - u_0^{l} \in L^{1}(\mathbb{R}_{-})$ and $u_0 - u_0^{r} \in L^{1}(\mathbb{R}_{+})$, 
we obtain the following scattering relation:

\begin{itemize}
    \item For $\lambda \in \mathbb{R} \setminus \{0\}$, both $\Phi^r(x;\lambda)$ and $\Phi^l(x;\lambda)$ exist. 
    Hence, there exists a scattering matrix $\mathbf{S}(\lambda)$ such that
    \begin{equation}\label{eq:Phirl0}
        \Phi^l(x;\lambda) = \Phi^r(x;\lambda) \mathbf{S}(\lambda),\quad \lambda\in\R\setminus\{0\}.
    \end{equation}
    and by the symmetry relation \eqref{sym:Jconjugate}, the scattering matrix takes the form
    \begin{equation*}
        \mathbf{S}(\lambda) = \begin{pmatrix}
             a(\lambda) & b^{*}(\lambda) \\
             b(\lambda) & a^{*}(\lambda)
        \end{pmatrix},\quad \lambda\in\R\setminus\{0\}.
    \end{equation*}
    where $f^{*}(\lambda) := \overline{f(\bar{\lambda})}$ denotes the Schwartz conjugate. Notice that, by \eqref{eq:Phirl0}, 
\begin{equation}\label{a:det}
    a(\lambda) = \det\bigl[\Phi_1^l(x;\lambda), \Phi_2^r(x;\lambda)\bigr],\ b(\lambda) = \det\bigl[\Phi_1^r(x;\lambda), \Phi_1^l(x;\lambda)\bigr], \quad \lambda\in\R\setminus\{0\}.
\end{equation}
Combining this with the properties of $\Phi(x;\lambda)$ in Proposition~\ref{prop:Phi}, 
it follows that $a(\lambda)$ admits an analytic continuation to 
$\mathbb{C}_+ \setminus (\Sigma_1^r \cup \Sigma_1^l)$. In the following, we still use $a(\lambda)$, since it continues to satisfy \eqref{a:det}.

\item For $\lambda \in \Sigma^r \cap \Sigma^l$, the boundary values of the Jost solutions $\Phi^{l}(x;\lambda_{\pm})$ and $\Phi^{r}(x;\lambda_{\pm})$ exist. 
Hence, there exists a scattering matrix $S(\lambda_{\pm})$ such that
\begin{equation}\label{eq:Phirl1}
    \Phi^l(x;\lambda_{\pm}) = \Phi^r(x;\lambda_{\pm}) S(\lambda_{\pm}),\quad \lambda \in \Sigma^r \cap \Sigma^l,
\end{equation}
and by the symmetry relation~\eqref{sym:Jconjugate}, the scattering matrix takes the form
\[
S(\lambda_{\pm}) = 
\begin{pmatrix}
    a(\lambda_{\pm}) & b_1^*(\lambda_{\pm}) \\
    b_1(\lambda_{\pm}) & a^*(\lambda_{\pm})
\end{pmatrix},\quad \lambda \in \Sigma^r \cap \Sigma^l,
\]
with
\begin{equation}\label{def:b1}
    b_1(\lambda) = \det\bigl[\Phi_1^r(x;\lambda), \Phi_1^l(x;\lambda)\bigr]\,\quad \lambda \in \Sigma^r \cap \Sigma^l.
\end{equation}
\item For $\lambda \in \Sigma_1^{r} \setminus \Sigma_1^{l}$, only the eigenvectors 
$\Phi_1^{l}(x;\lambda)$, $\Phi_1^{r}(x;\lambda_{\pm})$, and $\Phi_2^{r}(x;\lambda_{\pm})$ exist, 
giving the relation
\begin{equation}\label{eq:Phirl2}
    \Phi_1^{l}(x;\lambda) = \Phi^{r}(x;\lambda_{\pm})
    \begin{bmatrix}
        a(\lambda_{\pm}) \\[2pt]
        b_1(\lambda_{\pm})
    \end{bmatrix},\quad \lambda \in \Sigma_1^{r} \setminus \Sigma_1^{l}.
\end{equation}

\item For $\lambda \in \Sigma_1^{l} \setminus \Sigma_1^{r}$, only the eigenvectors 
$\Phi_2^{r}(x;\lambda)$, $\Phi_1^{l}(x;\lambda_{\pm})$, and $\Phi_2^{l}(x;\lambda_{\pm})$ exist, 
satisfying
\begin{equation}\label{eq:Phirl3}
    \Phi_2^{r}(x;\lambda) = \Phi^{l}(x;\lambda_{\pm})
    \begin{bmatrix}
        -b_1^{*}(\lambda_{\pm}) \\[2pt]
        a(\lambda_{\pm})
    \end{bmatrix},\quad \lambda \in \Sigma_1^{l} \setminus \Sigma_1^{r}.
\end{equation}

\item For $\lambda \in \Sigma_2^{r} \setminus \Sigma_2^{l}$, only the eigenvectors 
$\Phi_2^{l}(x;\lambda)$, $\Phi_1^{r}(x;\lambda_{\pm})$, and $\Phi_2^{r}(x;\lambda_{\pm})$ exist, 
with the corresponding relation
\begin{equation}\label{eq:Phirl4}
    \Phi_2^{l}(x;\lambda) = \Phi^{r}(x;\lambda_{\pm})
    \begin{bmatrix}
        b_1^{*}(\lambda_{\pm}) \\[2pt]
        a^{*}(\lambda_{\pm})
    \end{bmatrix},\quad \lambda \in \Sigma_2^{r} \setminus \Sigma_2^{l}.
\end{equation}

\item For $\lambda \in \Sigma_2^{l} \setminus \Sigma_2^{r}$, only the eigenvectors 
$\Phi_1^{r}(x;\lambda)$, $\Phi_1^{l}(x;\lambda_{\pm})$, and $\Phi_2^{l}(x;\lambda_{\pm})$ exist, 
and we have
\begin{equation}\label{eq:Phirl5}
    \Phi_1^{r}(x;\lambda) = \Phi^{l}(x;\lambda_{\pm})
    \begin{bmatrix}
        a^{*}(\lambda_{\pm}) \\[2pt]
        -b_1(\lambda_{\pm})
    \end{bmatrix},\quad \lambda \in \Sigma_2^{l} \setminus \Sigma_2^{r}.
\end{equation}
\end{itemize}
\begin{lem}\label{lem:ab}
Suppose that the initial data $u_0$ satisfies the condition \eqref{initial} and assumption \ref{assumption}. Then the functions $a(\lambda)$ and $b_1(\lambda)$ have the following properties:  

\begin{enumerate}
\item The function $a(\lambda)$ is defined on $\overline{\mathbb{C}_{+}}\cup(\Sigma_2^r\cap\Sigma_2^l)\setminus\{0,\pm i\eta^r_j,\pm \eta^l_j\}$, admits boundary values for $\lambda \in \Sigma_1^{r}\cup\Sigma_1^{l}$ and $\lambda\in\Sigma_2^r\cap\Sigma_2^l$, and is analytic in $\mathbb{C}_{+}\setminus\bigl(\Sigma_1^{r}\cup\Sigma_1^{l}\bigr)$. 
The function $b_1(\lambda)$ is defined and admits boundary values for $\lambda \in \Sigma_1^{r}\cup\Sigma_2^{l}\setminus\{\pm i\eta_j^r,\pm i\eta_j^l\},~j=1,2$.
\item  $a(\lambda)$ is $m_0$ differentiable for the left side and the right side $\lambda \in (\Sigma_1^r\cup\Sigma_1^l)\cup(\Sigma_2^l\cap\Sigma_2^r)\setminus \{\pm i\eta_j^r,\pm i\eta_j^r\}$ and $b_1(\lambda)$ is $m_0$ differentiable for the left side and the right side $\lambda\in\Sigma_1^r\cup\Sigma_2^l\setminus \{0,\pm i\eta_j^r,\pm i\eta_j^r\},~j=1,2$. $b(\lambda)$ is defined on $\lambda\in\R\setminus\{0\}$ and is $m_0$ differentiable for $\lambda\in\R\setminus\{0\}$.

\item For $\lambda \in \Sigma^{r}\cap\Sigma^{l}$, the scattering matrix $S(\lambda)$ satisfies
\begin{equation}\label{jump:S}
    S(\lambda_{+}) = \sigma_{1} S(\lambda_{-}) \sigma_{1}.
\end{equation}
More precisely, it follows that 
\begin{equation*}
\begin{aligned}
    a(\lambda_+)=a^*(\lambda_-),\quad b_1(\lambda_+)=b_1^*(\lambda_-),\quad \lambda\in\Sigma^r\cap\Sigma^l.
\end{aligned}
\end{equation*}
For $\lambda \in \Sigma^{r}\setminus\Sigma^{l}$ and $\lambda \in \Sigma^{l}\setminus\Sigma^{r}$, we have
\begin{equation}\label{eq:ba-jumps}
    \begin{aligned}
        & b_1(\lambda_+) = i\,a(\lambda_-), 
        && a(\lambda_+) = i\,b_1(\lambda_-),
        &&&\lambda \in \Sigma_1^r \setminus \Sigma_1^l, \\[1ex]
         & b_1^*(\lambda_+) = -i\,a(\lambda_-), 
        && a(\lambda_+) = -i\,b_1^*(\lambda_-),
        &&&\lambda \in \Sigma_1^l \setminus \Sigma_1^r, \\[1ex] 
        & b_1^*(\lambda_+) = -i\,a^*(\lambda_-), 
        &&a^*(\lambda_+) = -i\,b_1^*(\lambda_-),
        &&& \lambda \in \Sigma_2^r \setminus \Sigma_2^l,\\[1ex] 
        & b_1(\lambda_+) = i\,a^*(\lambda_-), 
        && a^*(\lambda_+) = i\,b_1(\lambda_-),
        &&& \lambda \in \Sigma_2^l \setminus \Sigma_2^r. \\[1ex]
    \end{aligned}
\end{equation}
\item For $\lambda \in \Sigma^l \cap \Sigma^r$ and for $\lambda \in \R$, the scattering matrices $S(\lambda)$ and $\mathbf{S}(\lambda)$ satisfy the following Schwartz conjugate symmetry:
\begin{equation}\label{sym:S}
\begin{aligned}
    & S(\lambda) = \sigma_{1} S^*(\lambda) \sigma_{1}, && \lambda \in \Sigma^l \cap \Sigma^r,\\
    & \mathbf{S}(\lambda) = \sigma_{1} \mathbf{S}^*(\lambda) \sigma_{1}=\sigma_1\mathbf{S}(-\lambda)\sigma_1, && \lambda \in \R.
\end{aligned}
\end{equation}
\item For each $x \in \R$, the scattering matrices $S(\lambda)$ and $\mathbf{S}(\lambda)$ satisfy
\begin{equation}\label{det:S}
    \det S(\lambda) = 1, \quad \lambda \in \Sigma^r \cap \Sigma^l, \qquad
    \det \mathbf{S}(\lambda) = 1, \quad \lambda \in \R.
\end{equation}
\item As $\lambda \to \infty$, the function $a(\lambda)$ admits the following asymptotic expansion:
\begin{equation}\label{expansion:a}
    a(\lambda) = e^{\,i(x_0^{l}-x_0^{r})\lambda} \bigl(1 + \mathcal{O}(\lambda^{-1})\bigr), \quad \lambda \to \infty,
\end{equation}
and if $u_0$ satisfies the condition \eqref{initial}, and then $b(\lambda)$ admits the following asymptotic expansion:
\begin{equation}\label{expansion:b}
    b(\lambda) = e^{\,i(x_0^{l}-x_0^{r})\lambda} \bigl( \mathcal{O}(\lambda^{-4})\bigr), \quad \lambda\in\R\text{ and }\lambda \to \infty.
\end{equation}
Moreover, as $\lambda \to 0$, the functions $a(\lambda)$ and $b(\lambda)$ have at most a simple pole.

\end{enumerate}
\end{lem}
\begin{proof}
Collecting the relations~\eqref{eq:Phirl1}--\eqref{eq:Phirl5}, it follows that the functions $a(\lambda)$ and $b_1(\lambda)$ can be characterized by
\begin{equation}\label{def:b}
    a(\lambda) = \det\bigl[\Phi_1^l(x;\lambda), \Phi_2^r(x;\lambda)\bigr], \qquad
    b_1(\lambda) = \det\bigl[\Phi_1^r(x;\lambda), \Phi_1^l(x;\lambda)\bigr].
\end{equation}
Then, by Proposition~\ref{prop:Phi} concerning $\Phi^l(x;\lambda)$ and $\Phi^r(x;\lambda)$, the definition, continuation, and analyticity properties of $a(\lambda)$ and $b_1(\lambda)$ are established.

Regarding the jump conditions, for $\lambda \in \Sigma_1^l \cap \Sigma_1^r$, using the relations in~\eqref{jump:Phir}, it follows that
\begin{equation*}
\begin{aligned}
    \Phi^l(x;\lambda_+) 
        &= \Phi^l(x;\lambda_-)(-i\sigma_1) 
        = \Phi^r(x;\lambda_-) S(\lambda_-)(-i\sigma_1) \\
        &= \Phi^r(x;\lambda_+)(i\sigma_1) S(\lambda_-)(-i\sigma_1) 
        = \Phi^r(x;\lambda_+) S(\lambda_+).
\end{aligned}
\end{equation*}
For $\lambda \in \Sigma_1^l \setminus \Sigma_1^r$, we have
\begin{equation*}
\begin{aligned}
    \Phi_1^l(x;\lambda) 
        &= \Phi^r(x;\lambda_+) 
            \begin{bmatrix}
                a(\lambda_+) \\[2pt]
                b_1(\lambda_+)
            \end{bmatrix} 
        = \Phi^r(x;\lambda_-) (-i\sigma_1)
            \begin{bmatrix}
                a(\lambda_+) \\[2pt]
                b_1(\lambda_+)
            \end{bmatrix} 
        = \Phi^r(x;\lambda_-) 
            \begin{bmatrix}
                a^*(\lambda_-) \\[2pt]
                b_1(\lambda_-)
            \end{bmatrix}.
\end{aligned}
\end{equation*}
The remaining jump relations can be derived by similar computations.

From the determinants of $\Phi^l(x;\lambda)$ and $\Phi^r(x;\lambda)$ in~\eqref{eq:detPhi}, the determinants of the scattering matrices $S(\lambda)$ and $\mathbf{S}(\lambda)$ follow straightforwardly from the relations~\eqref{eq:Phirl0} and~\eqref{eq:Phirl1}.

Substituting the expansions of $\Phi^l(x;\lambda)$ and $\Phi^r(x;\lambda)$ as $\lambda \to \infty$ in~\eqref{expansion:Phi} into the representation of $a(\lambda)$ in~\eqref{a:det}, we obtain
\[
a(\lambda) = \det\bigl[\Phi_1^l(x;\lambda), \Phi_2^r(x;\lambda)\bigr]
= \det
\begin{pmatrix}
    e^{-i(x-x_0^l)} & 0 \\
    0 & e^{i(x-x_0^r)}
\end{pmatrix} + \mathcal{O}(\lambda^{-1}), \quad \lambda \to \infty.
\]

By the symmetry condition in \eqref{sym:Phi-conjuate}, we can assume that
$
\Gamma_k^s=\begin{pmatrix}
    \alpha_k^s&\beta_k^s\\
    \bar{\beta_k^s}&\bar{\alpha_k^s}
\end{pmatrix},
$ and then we have the following equation:
\begin{equation}\label{eq:alpha-beta}
   \begin{cases}
       \begin{aligned}
           &\partial_x\alpha_k^s-i\sum_{l=1}^k p_l^s\alpha_{k-l}^s=\frac{u_0}{2}(\alpha_{k-1}^s+\bar{\beta}_{k-1}^s),\\
           &\partial_x\beta_k^s+2i\beta_{k+1}^s+i\sum_{l=1}^k p_l^s\beta_{k-l}^s=\frac{u_0}{2}(\bar{\alpha}_{k-1}^s+{\beta}_{k-1}^s),\\
       \end{aligned}
   \end{cases} 
\end{equation}
with $\beta_0^s=\beta_1^s=0$
 and $\alpha_0^s=1$.

For $b(\lambda)$, by the definition in \eqref{def:b}, it follows that
\begin{equation*}
    b(\lambda)e^{-i(x_0^l-x_0^r)\lambda}=\left[\frac{(\bar{\beta}_2^l-\bar{\beta}_2^r)}{\lambda^2}+\frac{b_3}{\lambda^3}\right]+\mathcal{O}(\lambda^{-4}),
\end{equation*}
with
$b_3=\alpha^{r}_{1}\,\bar{\beta}^{l}_{2} + \bar{\beta}^{l}_{3}
- i\Bigl(
 -i\,\alpha^{l}_{1}\,\bar{\beta}^{r}_{2}
 - i\,\beta^{r}_{3}
 + \bar{\beta}^{l}_{2}\,p_{1}^r\,x^{l}_0
 - \bar{\beta}^{r}_{2}\,p_{1}^r\,x^{l}_0
 - \bar{\beta}^{l}_{2}\,p_{1}^r\,x^{r}_0
 + \bar{\beta}^{r}_{2}\,p_{1}^r\,x^{r}_0
\Bigr).
$
By the system in \eqref{eq:alpha-beta}, we obtain that
$$
\beta_2^l=\beta_2^r=\frac{u_0}{4i},\beta_3^l-\beta_3^r=\beta_2^l(\bar{\alpha}_1^l-\bar{\alpha}_1^r).
$$
Thus, if $\Gamma^s(x;\lambda)$ admits the expansion at fourth order, then $ b(\lambda) = e^{\,i(x_0^{l}-x_0^{r})\lambda} \bigl( \mathcal{O}(\lambda^{-4})\bigr),\ \lambda \to \infty$.

Regarding the behaviour of $a(\lambda)$ as $\lambda \to 0$, Lemma~\ref{lem:Phiat0} implies that $a(\lambda),b(\lambda)$  have at most a simple pole.
\end{proof}

\begin{lem}
[Local behaviour of $a(\lambda)$ and $b_1(\lambda)$ at the endpoints]
\label{lem:ab-endpoints}
Suppose that the initial datum $u_0$ satisfies the condition~\eqref{initial}.
As $\lambda \to \pm i\eta_j^s$, $j=1,2$, the scattering data $a(\lambda)$ and $b_1(\lambda)$ admit at most fourth--root singularities in their respective domains.

In particular, define
\[
\widehat{\Phi}^s_1(\lambda):=\gamma^s(\lambda)\Phi^s_1(\lambda),
\qquad
\widehat{\Phi}^s_2(\lambda):=(\gamma^s(\lambda))^{-1}\Phi^s_2(\lambda).
\]
Then the following limits exist:
\begin{equation}\label{eq:limit1}
\begin{aligned}
    &\lim_{\substack{\lambda\to i\eta_2^r \\ \lambda\in\Sigma_1^r}}\frac{b_1(\lambda_+)}{b_1(\lambda_-)}
    =
    -i
    \frac{
        \det\!\left[\widehat{\Phi}_1^r(i\eta_{2}^r),\Phi_1^l(i\eta_{2,+}^r)\right]
    }{
        \det\!\left[\widehat{\Phi}_1^r(i\eta_{2}^r),\Phi_1^l(i\eta_{2,-}^r)\right]
    },
    \quad
    \lim_{\substack{\lambda\to i\eta_1^r \\ \lambda\in\Sigma_1^r}}\frac{b_1(\lambda_+)}{b_1(\lambda_-)}
    =
    i
    \frac{
        \det\!\left[\widehat{\Phi}_1^r(i\eta_{1}^r),\Phi_1^l(i\eta_{1,+}^r)\right]
    }{
        \det\!\left[\widehat{\Phi}_1^r(i\eta_{1}^r),\Phi_1^l(i\eta_{1,-}^r)\right]
     },\\
    &\lim_{\substack{\lambda\to i\eta_2^l \\ \lambda\in\Sigma_1^l}}\frac{b_1^*(\lambda_+)}{b_1^*(\lambda_-)}
    =
    -i
    \frac{
        \det\!\left[{\Phi}_2^r(i\eta_{2,+}^l),\widehat{\Phi}_2^l(i\eta_{2}^l)\right]
    }{
        \det\!\left[{\Phi}_2^r(i\eta_{2,+}^l),\widehat{\Phi}_2^l(i\eta_{2}^l)\right]
    },
    \quad
    \lim_{\substack{\lambda\to i\eta_1^l \\ \lambda\in\Sigma_1^l}}\frac{b_1(\lambda_+)}{b_1(\lambda_-)}
    =
    i
    \frac{
        \det\!\left[\widehat{\Phi}_1^r(i\eta_{1}^r),\Phi_1^l(i\eta_{1,+}^r)\right]
    }{
        \det\!\left[\widehat{\Phi}_1^r(i\eta_{1}^r),\Phi_1^l(i\eta_{1,-}^r)\right]
     },
\end{aligned}
\end{equation}
and 
\begin{equation}\label{eq:limit2}
\begin{aligned}
   &\lim_{\substack{\lambda\to i\eta_j^l \\ \lambda\in\Sigma_1^r\setminus\Sigma_1^l}}\frac{b_1(\lambda_+)}{b_1(\lambda_-)}
    =
    \frac{
        \det\!\left[{\Phi}_1^r(i\eta_{j,+}^l),\widehat{\Phi}_1^l(i\eta_{j}^l)\right]
    }{
        \det\!\left[{\Phi}_1^r(i\eta_{j,-}^l),\widehat{\Phi}_1^l(i\eta_{j}^l)\right]
    }
    \quad
    \lim_{\substack{\lambda\to i\eta_j^r \\ \lambda\in\Sigma_1^l\setminus\Sigma_1^r}}\frac{b_1^*(\lambda_+)}{b_1^*(\lambda_-)}
    =
    \frac{
        \det\!\left[\widehat{\Phi}_2^r(i\eta_{j}^l),{\Phi}_2^l(i\eta_{j,+}^l)\right]
    }{
        \det\!\left[\widehat{\Phi}_2^r(i\eta_{j}^l),{\Phi}_2^l(i\eta_{j,-}^l)\right]
    },~j=1,2.
\end{aligned}
\end{equation}
\end{lem}
\begin{proof}
   By Lemma~\ref{lem:Phi-endpoints}, the Jost solutions $\Phi^l(x;\lambda)$ possess fourth--root singularities at $\pm i\eta_j^l$, $j=1,2$, while $\Phi^r(x;\lambda)$ possess fourth--root singularities at $\pm i\eta_j^r$, $j=1,2$. 
By the definitions of $a(\lambda)$ in~\eqref{a:det} and $b_1(\lambda)$ in~\eqref{def:b1}, it then follows that the scattering data $a(\lambda)$ and $b_1(\lambda)$ admit at most a fourth--root singularity at each endpoint.

The analysis of the remaining limits follows in the same way, so we only consider the limit at $\lambda=i\eta_2^r$. 
Recall the expansions of $m_1^r(\lambda)$ and $m_2^r(\lambda)$ in \eqref{eq:expansion of eta2}. 
It follows that
\[
\lim_{\lambda \to i\eta_2^r} (\gamma^r(\lambda_+))^{-1} m_j^r(\lambda_+)
=
\lim_{\lambda \to i\eta_2^r} (\gamma^r(\lambda_-))^{-1} m_j^r(\lambda_-),
\]
and, by Lemma~\ref{lem:Phi-endpoints},
\[
\widehat{\Phi}_1^r(i\eta_2^r)
:=
\lim_{\lambda \to i\eta_2^r} \widehat{\Phi}_1^r(\lambda_+)
=
\lim_{\lambda \to i\eta_2^r} \widehat{\Phi}_1^r(\lambda_-).
\]
Substituting the above local behaviour into \eqref{def:b1}, we obtain
\begin{equation*}
\begin{aligned}
\lim_{\lambda \to i\eta_2^r}
\frac{b_1(\lambda_+)}{b_1(\lambda_-)}
&=
\lim_{\lambda \to i\eta_2^r}
\frac{\det[\Phi_1^r(\lambda_+),\Phi_1^l(\lambda_+)]}
{\det[\Phi_1^r(\lambda_-),\Phi_1^l(\lambda_-)]} \\
&=
\lim_{\lambda \to i\eta_2^r}
\frac{
\gamma^r(\lambda_+)
\det[(\gamma^r(\lambda_+))^{-1}\Phi_1^r(\lambda_+),\Phi_1^l(\lambda_+)]
}{
\gamma^r(\lambda_-)
\det[(\gamma^r(\lambda_-))^{-1}\Phi_1^r(\lambda_-),\Phi_1^l(\lambda_-)]
}.
\end{aligned}
\end{equation*}
\end{proof}

\section{RH problem for the step-like KdV finite gap potentials}\label{section:RHP}
Define a $2\times2$ matrix-valued meromorphic function $M(x;\lambda)$ by
\begin{equation}\label{def:M}
M(x;\lambda)=
\begin{cases}
\begin{pmatrix}
\dfrac{\Phi_1^l(x;\lambda)}{a(\lambda)} & \Phi_2^r(x;\lambda)
\end{pmatrix}
e^{\,i(x-x_0^r)\lambda\sigma_3}, & \lambda \in \C_+\setminus\{0\}, \\[10pt]
\begin{pmatrix}
\Phi_1^r(x;\lambda) & \dfrac{\Phi_2^{l}(x;\lambda)}{a^{*}(\lambda)}
\end{pmatrix}
e^{\,i(x-x_0^r)\lambda\sigma_3}, & \lambda \in \C_-\setminus\{0\} .
\end{cases}
\end{equation}
Here $\Phi_j^r(x;\lambda)$ and $\Phi_j^l(x;\lambda)$ ($j=1,2$) denote the $j$-th column of $\Phi^l(x;\lambda)$ and $\Phi^r(x;\lambda)$ in (\ref{Jost:Phi}), respectively. Then the $2\times2$ matrix-valued meromorphic $M(x;\lambda)$ satisfies the following RH problem:
\begin{RHP}\label{RHP:M}
The matrix function $M(x;\lambda)$ is analytic for 
$\lambda \in \C \setminus (\R \cup \Sigma^{l}\cup\Sigma^r)$, 
with at most a simple singularity at $\lambda=0$.
\begin{enumerate}
    \item $M(x;\lambda)$ satisfies the following jump conditions:
    \begin{equation}
         M_+(x;\lambda)=M_-(x;\lambda)V^{(M)}(x;\lambda),
    \end{equation}
    where
    \begin{equation*}
    V^{(M)}(x;\lambda)=\begin{cases} \begin{aligned} &\begin{pmatrix} \frac{a(\lambda_-)}{ib_1(\lambda_-)}&-ie^{-2i(x-x^r_0)\lambda}\\ 0&\frac{ib_1(\lambda_-)}{a(\lambda_-)} \end{pmatrix},&&\lambda\in\Sigma_1^r\setminus\Sigma_1^l,\\ 
    &\begin{pmatrix} \frac{-ib_1^*(\lambda_-)}{a(\lambda_+)}&-ie^{-2i(x-x^r_0)\lambda}\\ \frac{-i e^{2i(x-x^r_0)\lambda}}{a(\lambda_+)a(\lambda_-)}&\frac{ib_1(\lambda_-)}{a(\lambda_-)} \end{pmatrix},&&\lambda\in\Sigma_1^r\cap\Sigma_1^l,\\ &\begin{pmatrix} \frac{-ib_1^*(\lambda_-)}{a(\lambda_+)}&0\\ \frac{-ie^{2i(x-x^r_0)\lambda}}{a(\lambda_+)a(\lambda_-)}&1 \end{pmatrix},&&\lambda\in\Sigma_1^l\setminus\Sigma_1^r,\\ &\begin{pmatrix} \frac{1}{a(\lambda)a^*(\lambda)}&-\frac{b^*(\lambda)}{a^*(\lambda)}e^{-2i(x-x^r_0)\lambda}\\ \frac{b(\lambda)}{a(\lambda)}e^{2i(x-x^r_0)\lambda}&1 \end{pmatrix},&&\lambda\in \R,\\ &\begin{pmatrix} 1&\frac{ie^{-2i(x-x^r_0)}}{a^*(\lambda_+)a^*(\lambda_-)}\\ 0&i\frac{b_1(\lambda_-)}{a^*(\lambda_+)} \end{pmatrix},&&\lambda\in \Sigma_2^{l}\setminus\Sigma_2^{r},\\ &\begin{pmatrix} -\frac{ib_1^*(\lambda_-)}{a^{*}(\lambda_-)}&\frac{ie^{-2i(x-x^r_0)}}{a^*(\lambda_+)a^*(\lambda_-)}\\ ie^{2i(x-x^r_0)}&\frac{ib_1(\lambda_-)}{a^*(\lambda_+)} \end{pmatrix},&&\lambda\in\Sigma_2^{r}\cap\Sigma_2^{l},\\ &\begin{pmatrix} -\frac{ib_1^*(\lambda_-)}{a^*(\lambda_-)}&0\\ ie^{2i(x-x^r_0)}&\frac{ia^*(\lambda_-)}{b_1^*(\lambda_-)} \end{pmatrix},&&\lambda\in\Sigma_2^r\setminus\Sigma_2^l. \end{aligned} \end{cases} 
    \end{equation*}

    \item Normalization condition:
    \[
        M(x;\lambda) = I+\mathcal{O}(\lambda^{-1}), \qquad \lambda \to \infty.
    \]

    \item Symmetry condition:
    \[
       M^*(x;\lambda)= M(x;-\lambda) = \sigma_1 M(x;\lambda) \sigma_1.
    \]
\end{enumerate}   
\end{RHP}
\begin{proof}
For $\lambda \in \Sigma_1^r \setminus \Sigma_1^l$, the vector function $\Phi_1^l(x;\lambda)$ has no jump, and hence
\begin{equation*}
\frac{\Phi_1^l(x;\lambda) e^{i(x-x_0^r)\lambda}}{a(\lambda_+)}
= \frac{\Phi_1^l(x;\lambda) e^{i(x-x_0^r)\lambda}}{a(\lambda_-)}
  \frac{a(\lambda_-)}{a(\lambda_+)}
= \frac{\Phi_1^l(x;\lambda) e^{i(x-x_0^r)\lambda}}{a(\lambda_-)}
  \frac{a(\lambda_-)}{i\,b_1(\lambda_-)},
\end{equation*}
where the last equality follows from~\eqref{eq:ba-jumps}.
As for $\Phi_2^r(x;\lambda)$, we have
\begin{equation*}
\begin{aligned}
    \Phi_2^r(x;\lambda_+) e^{-i(x-x_0^r)\lambda}
        &= \frac{e^{-i(x-x_0^r)\lambda}}{b_1(\lambda_+)}
           \bigl[\Phi_1^l(x;\lambda) - a(\lambda_+)\Phi_1^r(x;\lambda_+)\bigr] \\
        &= \frac{e^{-i(x-x_0^r)\lambda}}{b_1(\lambda_+)}
           \bigl[\Phi_1^l(x;\lambda) + i\,a(\lambda_+)\Phi_2^r(x;\lambda_-)\bigr] \\
        &= -\,i\,e^{-2i(x-x_0^r)\lambda}
           \frac{\Phi_1^l(x;\lambda)e^{i(x-x_0^r)\lambda}}{a(\lambda_-)}
           + \frac{i\,b_1(\lambda_-)}{a(\lambda_-)}\,
             \Phi_2^r(x;\lambda_-) e^{-i(x-x_0^r)\lambda},
\end{aligned}
\end{equation*}
where the first equality follows from~\eqref{eq:Phirl2}, the second from the jump condition of $\Phi^r(x;\lambda)$ in~\eqref{jump:Phir}, and the last from the jump relations of $a(\lambda)$ and $b_1(\lambda)$ in~\eqref{eq:ba-jumps}.

For $\lambda \in \Sigma_1^r \cap \Sigma_1^l$, by the jump conditions of the Jost solutions in~\eqref{jump:Phir}, together with the relation~\eqref{eq:Phirl1}, we obtain
\begin{equation*}
\begin{aligned}
    \frac{\Phi_1^l(x;\lambda_+)}{a(\lambda_+)} e^{i(x-x_0^r)\lambda}
        &= \frac{-i\,e^{i(x-x_0^r)\lambda}}{a(\lambda_+)}\,\Phi_2^l(x;\lambda_-) \\
        &= \frac{-i\,e^{i(x-x_0^r)\lambda}}{a(\lambda_+)a(\lambda_-)}
           \bigl[b_1^*(\lambda_-)\Phi_1^l(x;\lambda_-) + \Phi_2^r(x;\lambda_-)\bigr] \\
        &= \frac{-i\,b_1^*(\lambda_-)}{a(\lambda_+)}
           \frac{\Phi_1^l(x;\lambda_-)}{a(\lambda_-)} e^{i(x-x_0^r)\lambda}
           - \frac{i\,e^{2i(x-x_0^r)\lambda}}{a(\lambda_+)a(\lambda_-)}
             \Phi_2^r(x;\lambda_-) e^{-i(x-x_0^r)\lambda},
\end{aligned}
\end{equation*}
and
\begin{equation*}
\begin{aligned}
    \Phi_2^r(x;\lambda_+) e^{-i(x-x_0^r)\lambda}
        &= -\,i\,\Phi_2^l(x;\lambda_-) e^{-i(x-x_0^r)\lambda} \\
        &= \frac{-i\,e^{-i(x-x_0^r)\lambda}}{a(\lambda_-)}
           \bigl[\Phi_1^l(x;\lambda_-) - b_1(\lambda_-)\Phi_2^r(x;\lambda_-)\bigr] \\
        &= -\,i\,e^{-2i(x-x_0^r)\lambda}
           \frac{\Phi_1^l(x;\lambda_-)}{a(\lambda_-)} e^{i(x-x_0^r)\lambda}
           + \frac{i\,b_1(\lambda_-)}{a(\lambda_-)}\,
             \Phi_2^r(x;\lambda_-) e^{-i(x-x_0^r)\lambda}.
\end{aligned}
\end{equation*}

For $\lambda \in \Sigma_1^l \setminus \Sigma_1^r$, the function $\Phi_2^r(x;\lambda)$ has no jump, whereas for $\Phi_1^l(x;\lambda)$ we have
\begin{equation*}
\begin{aligned}
    \frac{\Phi_1^l(x;\lambda_+) e^{i(x-x_0^r)\lambda}}{a(\lambda_+)}
    &= \frac{e^{i(x-x_0^r)\lambda}}{a(\lambda_+)}
        \left[\frac{a(\lambda_+)\Phi_2^l(x;\lambda_+) - \Phi_2^r(x;\lambda)}{b_1^*(\lambda_+)}\right] \\
    &= \frac{-e^{i(x-x_0^r)\lambda}}{a(\lambda_+)b_1^*(\lambda_+)}
        \left[i\,a(\lambda_+)\Phi_1^l(x;\lambda_-) + \Phi_2^r(x;\lambda)\right] \\
    &= \frac{-i\,b_1^*(\lambda_-)}{a(\lambda_+)}
        \frac{\Phi_1^l(x;\lambda_-) e^{i(x-x_0^r)\lambda}}{a(\lambda_-)}
        - \frac{ie^{2i(x-x_0^r)\lambda}}{\,a(\lambda_+)a(\lambda_-)}\,
          \Phi_2^r(x;\lambda)e^{-i(x-x_0^r)\lambda}.
\end{aligned}
\end{equation*}
Here, the first equality follows from~\eqref{eq:Phirl2}, the second from the jump condition of $\Phi^l(x;\lambda)$ in~\eqref{jump:Phir}, and the last from the jump relations of $a(\lambda)$ and $b_1^*(\lambda)$ in~\eqref{eq:ba-jumps}.
The jump conditions on the real line are obtained through standard analysis, while those on the lower half-plane can be derived either from the symmetry relations or by direct computation.

The normalization condition is a direct consequence of the asymptotic expansions of $\Phi^l(x;\lambda)$ and $\Phi^r(x;\lambda)$ in~\eqref{expansion:Phi}, together with the asymptotic behavior of $a(\lambda)$ as $\lambda \to \infty$ in~\eqref{expansion:a}.
Moreover, the symmetry condition follows directly from the symmetry properties of the Jost solutions given in~\eqref{sym:Phi-pm} and~\eqref{sym:Phi-conjuate}.
\end{proof}
Let $r(\lambda):=\frac{b(\lambda)}{a(\lambda)}$ for $\lambda\in\R$ and introduce the following auxiliary function $h(\lambda),~\lambda\in\C_+$:
\begin{equation}\label{eq:h-lambda}
\begin{aligned}
h(\lambda)
&= \exp\Biggl\{
   -i(x_0^l - x_0^r)\lambda
   + \frac{1}{2\pi i}\Biggl[
      \int_{\Sigma_1^r \setminus \Sigma_1^l}
        \frac{\ln\!\left(-\tfrac{b_1(\zeta_-)}{b_1(\zeta_+)}\right)}{\zeta-\lambda}\, d\zeta
    + \int_{\Sigma_1^l \cap \Sigma_1^r}
        \frac{\ln\!\left(-\tfrac{b_1(\zeta_-)}{b_1^*(\zeta_-)}\right)}{\zeta-\lambda}\, d\zeta \\[4pt]
&\qquad\quad
    + \int_{\Sigma_1^l \setminus \Sigma_1^r}
        \frac{\ln\!\left(-\tfrac{b_1^*(\zeta_+)}{b_1^*(\zeta_-)}\right)}{\zeta-\lambda}\, d\zeta
    + \int_{\R}
        \frac{\ln\!\bigl(1-|r(\zeta)|^2\bigr)}{\zeta-\lambda}\, d\zeta
    + \int_{\Sigma_2^l \setminus \Sigma_2^r}
        \frac{\ln\!\left(-\tfrac{b_1(\zeta_-)}{b_1(\zeta_+)}\right)}{\zeta-\lambda}\, d\zeta \\[4pt]
&\qquad\quad
    + \int_{\Sigma_2^l \cap \Sigma_2^r}
        \frac{\ln\!\left(-\tfrac{b_1(\zeta_-)}{b_1^*(\zeta_-)}\right)}{\zeta-\lambda}\, d\zeta
    + \int_{\Sigma_2^r \setminus \Sigma_2^l}
        \frac{\ln\!\left(-\tfrac{b_1^*(\zeta_+)}{b_1^*(\zeta_-)}\right)}{\zeta-\lambda}\, d\zeta
   \Biggr]
\Biggr\}.
\end{aligned}
\end{equation}
By using the auxiliary function $h(\lambda)$, we decompose the function $a(\lambda):=a_1(\lambda)a_2(\lambda)$ as follows:
\begin{equation}\label{def:a1a2}
    \begin{aligned}
        a_1(\lambda)&=(a(\lambda)/h(\lambda))^{1/2},&&\lambda\in\C_+,\\
        a_2(\lambda)&=(a(\lambda)h(\lambda))^{1/2},&&\lambda\in\C_+.\\
    \end{aligned}
\end{equation}
\begin{pro}\label{pro:a12}
The functions $a_1(\lambda)$ and $a_2(\lambda)$ defined in~\eqref{def:a1a2} satisfy the following properties:
\begin{enumerate}
    \item The function $a_1(\lambda)$ is defined for $\lambda \in \overline{\mathbb{C}}_+ \setminus \{0,\, i\eta_1^l,\, i\eta_2^l\}$ and is analytic for $\lambda \in \mathbb{C}_+ \setminus \Sigma_1^l$. 
Similarly, the function $a_2(\lambda)$ is defined for $\lambda \in \overline{\mathbb{C}}_+ \setminus \{0,i\eta_1^r,\, i\eta_2^r \}$ and is analytic for $\lambda \in \mathbb{C}_+ \setminus \Sigma_1^r$. 
    
    \item In particular, for $\lambda \in \Sigma_1^r \cap \Sigma_1^l$, we have
    \begin{equation*}
        \frac{a_2(\lambda_+)}{a_1(\lambda_+)}
        = -\,\frac{a_2(\lambda_-)b_1(\lambda_-)}{a_1(\lambda_-)b_1^*(\lambda_-)}, 
        \qquad \lambda \in \Sigma_1^l \cap \Sigma_1^r.
    \end{equation*}
    Moreover, for $\lambda \in \R$, it holds that
    \begin{equation*}
        \frac{1}{|a_1(\lambda)|^2} = 1 - |r(\lambda)|^2, 
        \qquad 
        |a_2(\lambda)|^2 = 1, 
        \qquad \lambda \in \R.
    \end{equation*}

    \item As $\lambda \to \infty$, the functions $a_1(\lambda)$ and $a_2(\lambda)$ admit the following asymptotic expansions:
    \begin{equation}\label{asym:a1a2}
        a_1(\lambda) = e^{i(x_0^l - x_0^r)\lambda} + \mathcal{O}(\lambda^{-1}),
        \qquad 
        a_2(\lambda) = 1 + \mathcal{O}(\lambda^{-1}),
        \qquad \lambda \to \infty.
    \end{equation}
   \item For each endpoint $\lambda = i\eta_j^s$, the modified scattering data 
$a_1(\lambda)$ and $a_2(\lambda)$ admit the following local behaviour:
\begin{equation}\label{endpoints:a1a2}
a_1(\lambda)=
\begin{cases}
\mathcal{O}(1), & \lambda \to i\eta_2^r,\\[0.3ex]
(\lambda-i\eta_2^l)^{-\frac{1}{4}}, & \lambda \to i\eta_2^l,\\[0.3ex]
\mathcal{O}(1), & \lambda \to i\eta_1^r,\\[0.3ex]
(\lambda-i\eta_1^l)^{-\frac{1}{4}}, & \lambda \to i\eta_1^l,
\end{cases}
\qquad
a_2(\lambda)=
\begin{cases}
(\lambda-i\eta_2^r)^{-\frac{1}{4}}, & \lambda \to i\eta_2^r,\\[0.3ex]
\mathcal{O}(1), & \lambda \to i\eta_2^l,\\[0.3ex]
(\lambda-i\eta_1^r)^{-\frac{1}{4}}, & \lambda \to i\eta_1^r,\\[0.3ex]
\mathcal{O}(1), & \lambda \to i\eta_1^l.
\end{cases}
\end{equation}
\end{enumerate}
\end{pro}

\begin{proof}
    Observe that $h(\lambda)$ is defined via the Cauchy-type integral representation, from which the following jump relations hold:
    \begin{equation}
    \label{Jumps:h}
\begin{aligned}
h(\lambda_+) &= -h(\lambda_-)
\begin{cases}
\dfrac{b_1(\lambda_-)}{b_1(\lambda_+)}, & \lambda \in \Sigma_1^r \setminus \Sigma_1^l,\\[3pt]
\dfrac{b_1(\lambda_-)}{b_1^*(\lambda_-)}, & \lambda \in \Sigma_1^r \cap \Sigma_1^l,\\[3pt]
\dfrac{b_1^*(\lambda_+)}{b_1^*(\lambda_-)}, & \lambda \in \Sigma_1^l \setminus \Sigma_1^r,
\end{cases}\\[6pt]
h(\lambda)h^*(\lambda) &= 1 - |r(\lambda)|^2, \qquad \lambda \in \R.
\end{aligned}
\end{equation}
    Moreover, as $\lambda \to \infty$, we have
    \[
        h(\lambda) = e^{-i(x_0^l - x_0^r)\lambda} + \mathcal{O}(\lambda^{-1}), 
        \qquad \lambda \to \infty.
    \]
    
    Combining these jump relations with those of $a(\lambda)$ and $b_1(\lambda)$ in~\eqref{eq:ba-jumps}, one obtains
    \[
    \begin{aligned}
        &\frac{a(\lambda_+)}{h(\lambda_+)}
        = -\,\frac{i b_1(\lambda_-)}{h(\lambda_-)}\frac{b_1(\lambda_+)}{b_1(\lambda_-)}
        = \frac{a(\lambda_-)}{h(\lambda_-)}, 
        && \lambda \in \Sigma_1^r \setminus \Sigma_1^l, \\[3pt]
        &a(\lambda_+)h(\lambda_+)
        = i b_1^*(\lambda_-)h(\lambda_-)\frac{b_1^*(\lambda_+)}{b_1^*(\lambda_-)}
        = a(\lambda_-)h(\lambda_-), 
        && \lambda \in \Sigma_1^l \setminus \Sigma_1^r.
    \end{aligned}
    \]
    Consequently, $a_1(\lambda)$ has no jump for $\lambda \in \Sigma_1^r \setminus \Sigma_1^l$, and $a_2(\lambda)$ has no jump for $\lambda \in \Sigma_1^l \setminus \Sigma_1^r$.
    
   \noindent {\bf{The proof of 4.}} By Lemma~\ref{lem:ab-endpoints}, the scattering coefficient $a(\lambda)$ exhibits a fourth--root singularity at each endpoint $i\eta_j^s$, $j=1,2$.
In order to study the local behaviour of $a_1(\lambda)$ and $a_2(\lambda)$ at the endpoints, it therefore suffices to determine the local behaviour of $h(\lambda)$. 

For $\lambda \in \mathbb{D}_{\epsilon}(i\eta_2^r)$, a sufficiently small disk centered at $i\eta_2^r$ with radius $\epsilon$, we consider two cases:
$i\eta_2^r \notin \Sigma_1^l$ and $i\eta_2^r \in \Sigma_1^l$.

\medskip
\noindent
\textbf{Case 1:} $i\eta_2^r \notin \Sigma_1^l$.

By the definition of $h(\lambda)$ in \eqref{eq:h-lambda}, the only interval of integration is 
$\Sigma_1^r \setminus \Sigma_1^l$, which can be written as
\[
\Sigma_1^r \setminus \Sigma_1^l = (i\eta, i\eta_2^r),
\]
where $\eta$ depends on the configuration in Figure~\ref{fig:bands}.
Hence,
\begin{equation*}
\begin{aligned}
\int_{\Sigma_1^r \setminus \Sigma_1^l}
\frac{\ln\!\left(-\tfrac{b_1(\zeta_-)}{b_1(\zeta_+)}\right)}{\zeta-\lambda}\, d\zeta
&=
\ln\!\left(-\tfrac{b_1(i\eta_{2,-}^r)}{b_1(i\eta_{2,+}^r)}\right)
\int_{i\eta}^{i\eta_2^r} \frac{d\zeta}{\zeta-\lambda} \\
&\quad +
\int_{i\eta}^{i\eta_2^r}
\left[
\ln\!\left(-\tfrac{b_1(\zeta_-)}{b_1(\zeta_+)}\right)
-
\ln\!\left(-\tfrac{b_1(i\eta_{2,-}^r)}{b_1(i\eta_{2,+}^r)}\right)
\right]
\frac{d\zeta}{\zeta-\lambda}.
\end{aligned}
\end{equation*}
For the second term in the above equation, the appendix of 
\cite{grava_direct_nodate} implies that it is 
$o\!\left(\log|\lambda - i\eta_2^r|\right)$ 
as $\lambda \to i\eta_2^r$.

\medskip
\noindent
\textbf{Case 2:} $i\eta_2^r \in \Sigma_1^l$.

In this case, $i\eta_2^r$ is a common endpoint of 
$\Sigma_1^r \cap \Sigma_1^l$ and $\Sigma_1^l \setminus \Sigma_1^r$.
We decompose the integrals as follows:
\begin{equation*}
\begin{aligned}
\int_{\Sigma_1^l \cap \Sigma_1^r}
\frac{\ln\!\left(-\tfrac{b_1(\zeta_-)}{b_1^*(\zeta_-)}\right)}{\zeta-\lambda}\, d\zeta
&=
\ln\!\left(-\tfrac{b_1(\eta_{2,-}^r)}{b_1(\eta_{2,+}^r)}\right)
\int_{i\eta_1^l}^{i\eta_2^r} \frac{d\zeta}{\zeta-\lambda} \\
&\quad +
\int_{i\eta_1^l}^{i\eta_2^r}
\left[
\ln\!\left(-\tfrac{b_1(\zeta_-)}{b_1(\zeta_+)}\right)
-
\ln\!\left(-\tfrac{b_1(i\eta_{2,-}^r)}{b_1(i\eta_{2,+}^r)}\right)
\right]
\frac{d\zeta}{\zeta-\lambda},
\end{aligned}
\end{equation*}

and
\begin{equation*}
\begin{aligned}
\int_{\Sigma_1^l \setminus \Sigma_1^r}
\frac{\ln\!\left(-\tfrac{b_1^*(\zeta_+)}{b_1^*(\zeta_-)}\right)}{\zeta-\lambda}\, d\zeta
&=
\ln\!\left(-\tfrac{b_1^*(\eta_{2,+}^r)}{b_1^*(\eta_{2,-}^r)}\right)
\int_{i\eta_2^r}^{i\eta_2^l} \frac{d\zeta}{\zeta-\lambda} \\
&\quad +
\int_{i\eta_2^r}^{i\eta_2^l}
\left[
\ln\!\left(-\tfrac{b_1^*(\zeta_+)}{b_1^*(\zeta_-)}\right)
-
\ln\!\left(-\tfrac{b_1^*(i\eta_{2,+}^r)}{b_1^*(i\eta_{2,-}^r)}\right)
\right]
\frac{d\zeta}{\zeta-\lambda}.
\end{aligned}
\end{equation*}
Combining \eqref{eq:limit1}, \eqref{eq:limit2}, and the jump condition in \eqref{jump:Phir}, we obtain
\[
\begin{aligned}
&\ln\!\left(-\frac{b_1(\eta_{2,-}^r)}{b_1(\eta_{2,+}^r)}\right)
-
\ln\!\left(-\frac{b_1^*(\eta_{2,+}^r)}{b_1^*(\eta_{2,-}^r)}\right) \\
&=
\ln\!\left(
-i
\frac{
\det\!\left[\widehat{\Phi}_1^r(i\eta_{2}^r),\Phi_1^l(i\eta_{2,-}^r)\right]
}{
\det\!\left[\widehat{\Phi}_1^r(i\eta_{2}^r),\Phi_1^l(i\eta_{2,+}^r)\right]
}
\right) 
-
\ln\!\left(
-
\frac{
\det\!\left[\widehat{\Phi}_2^r(i\eta_{2}^l),{\Phi}_2^l(i\eta_{2,+}^l)\right]
}{
\det\!\left[\widehat{\Phi}_2^r(i\eta_{2}^l),{\Phi}_2^l(i\eta_{2,-}^l)\right]
}
\right)
= -\frac{1}{4}.
\end{aligned}
\]

Consequently, as $\lambda \to i\eta_2^r$, the function $h(\lambda)$ admits the expansion
\[
h(\lambda) = (\lambda - i\eta_2^r)^{-\frac{1}{4}} e^{h_0(\lambda)},
\]
where $h_0(\lambda)$ remains bounded as $\lambda \to i\eta_2^r$.

More precisely, we have the following local behaviour.
\[
\begin{aligned}
h(\lambda) &= \mathcal{O}\!\left((\lambda - i \eta_2^r)^{-\frac14}\right),
&\qquad & \lambda \to i \eta_2^r,
&\qquad
h(\lambda) &= \mathcal{O}\!\left((\lambda - i \eta_2^l)^{\frac14}\right),
&\qquad & \lambda \to i \eta_2^l, \\[0.6ex]
h(\lambda) &= \mathcal{O}\!\left((\lambda - i \eta_1^r)^{-\frac14}\right),
&\qquad & \lambda \to i \eta_1^r,
&\qquad
h(\lambda) &= \mathcal{O}\!\left((\lambda - i \eta_1^l)^{\frac14}\right),
&\qquad & \lambda \to i \eta_1^l .
\end{aligned}
\]
By the definition of $a_1(\lambda),a_2(\lambda)$ in \eqref{def:a1a2}, it follows the local behaviour in \eqref{endpoints:a1a2}.
\end{proof}

Introduce the vector-valued function $X(x;\lambda)$ by
\begin{equation}
    X(x;\lambda)
    =
    \begin{bmatrix}
        1 & 1
    \end{bmatrix}
    M(x;\lambda)
    \begin{cases}
        a_2(\lambda)^{\sigma_3}, & \lambda \in \C_+, \\[2pt]
        a_2^*(\lambda)^{-\sigma_3}, & \lambda \in \C_-.
    \end{cases}
\end{equation}

Define the following reflection coefficients $r_1(\lambda),r_2(\lambda)$ and $\rho(\lambda)$ by
\begin{equation}\label{def:r1r2}
\begin{aligned}
    &r_1(\lambda) =\frac{a_2(\lambda_-)}{a_1(\lambda_-)}\frac{e^{-2ix_0^r\lambda}}{a_2(\lambda_-)a_1(\lambda_+)-ib_1^*(\lambda_-)}, &&\lambda\in\Sigma_1^l,\\
    &r_2(\lambda) =\frac{a_1(\lambda_-)}{a_2(\lambda_-)}\frac{e^{2ix_0^r\lambda}}{a_1(\lambda_-)a_2(\lambda_+)+ib_1(\lambda_-)}, &&\lambda\in\Sigma_1^r,\\
    &\rho(\lambda)=\frac{b(\lambda)}{a_1(\lambda)a_2^*(\lambda)}e^{-2ix_0^r\lambda},&&\lambda\in\R.
\end{aligned}
\end{equation}



\subsection*{The proof of the lemma \ref{lem:r1r2}}
\begin{proof}
    By the proposition of $a_1(\lambda),a_2(\lambda)$ in \ref{pro:a12} and the local behavior of $b_1(\lambda)$ in lemma~\ref{lem:ab-endpoints}, it follows that $r_1(\lambda)$ exhibits square-root behaviour at $i\eta_1^l,i\eta_2^l$. 
    If an endpoint of $\Sigma_1^r$ lies on $\Sigma_1^l$, such as $i\eta_2^r$ in pattern~$\mathrm{(iii)}$ of Figure~\ref{fig:bands}, then $r_1(\lambda)$ is bounded at this point.

For $\lambda \in \Sigma_1^l \setminus \Sigma_1^r$, by the jump conditions for $a(\lambda)$ in~\eqref{eq:ba-jumps} and those for $h(\lambda)$ in~\eqref{Jumps:h}, it follows that
\[
a_2(\lambda_-)a_1(\lambda_+) = -i\, b_1^*(\lambda_-).
\]
Consequently, by the jump condition of $b_1^*(\lambda)$ in \eqref{eq:ba-jumps}, $r_1(\lambda)$ can be rewritten as
\[
r_1(\lambda)
= \frac{e^{-2i\lambda x_0^r}}{2\,a_1(\lambda_-)a_1(\lambda_+)}.
\]

Notice that for $\lambda\in\mathbb{R}$ and as $\epsilon\to0^+$, we have
\(
\lim_{\epsilon\to 0^+} h(\lambda+i\epsilon)=\frac{\exp({\alpha})}{a(\lambda)},
\)
where $\alpha$ is the exponential part in \eqref{eq:h-lambda} without the term $\int_{\mathbb{R}}
\frac{\ln(1-|r(\zeta)|^2)}{\zeta-\lambda}\,d\zeta$. Thus, by the local behaviour of $b(\lambda)$ and
$a(\lambda)$ at $\lambda=0$ in lemma \ref{lem:ab}, it follows that $\rho(\lambda)$ is regular at $\lambda=0$.

\end{proof}
Consequently, we formulate the RH problem \ref{RHP:X} for $X(x;\lambda)$ with $t=0$, which is obtained from RH problem~\ref{RHP:M} by substituting $r_1(\lambda)$ and $r_2(\lambda)$ into the corresponding jump matrices.

\section{Time Evolution and Inverse Scattering}\label{section:inverse}
Now consider $\theta(x,t;\lambda)=x\lambda+4t\lambda^{3}$ in the RH problem~\ref{RHP:X}. 
In this section, we show that the RH problem $X(x,t;\lambda)$ admits a unique solution. 
Moreover, as $t$ evolves, the solution of the KdV equation with initial data~\eqref{initial} 
can be recovered from $X(x,t;\lambda)$.

\begin{lem}\label{lem:vanishing}
	If the spectral function $\rho(\lambda)\in L^2(\mathbb{R})$, $r_1(\lambda)\in L^2(\Sigma_1^l)$, and $r_2(\lambda)\in L^2(\Sigma_1^r)$, then the RH problem for $X(x,t;\lambda)$ admits a unique solution.
\end{lem}

\begin{proof}
Consider the $1\times2$ vector-valued RH problem for $N(x,t;\lambda)$, which satisfies the same jump conditions as $X(x,t;\lambda)$, but with the normalization $N(x,t;\lambda)=\mathcal{O}(\lambda^{-1})$ as $\lambda\to\infty$. More precisely,
\begin{equation*}
    \begin{cases}
        N_+(x,t;\lambda) = N_-(x,t;\lambda)\, V(x,t;\lambda), & \lambda \in \mathbb{R} \cup \Sigma^r \cup \Sigma^l, \\[6pt]
        N(x,t;\lambda) = \mathcal{O}(\lambda^{-1}), & \lambda \to \infty .
    \end{cases}
\end{equation*}
Suppose that 
\[
    H(\lambda)=N(\lambda)\bigl(N^*(\lambda)\bigr)^{t},
\]
where $(\cdot)^t$ denotes transpose. It is immediate that $H(\lambda)$ is analytic for 
$\lambda\in\mathbb{C}^+\setminus(\Sigma_1^r\cup\Sigma_1^l)$ and satisfies 
$H(\lambda)=\mathcal{O}(\lambda^{-2})$ as $\lambda\to\infty$. Consequently,
\begin{equation*}
    \int_{\mathbb{R}} H(\lambda)\,d\lambda 
    + \int_{\Sigma_1^l\cup\Sigma_1^r} \bigl(H_+(\lambda)-H_-(\lambda)\bigr)\, d\lambda 
    = 0.
\end{equation*}

For $\lambda\in\Sigma_1^r\cup\Sigma_1^l$, we have
\[
    H_+(\lambda)
    = N_+(\lambda) N_+^*(\lambda)^t
    = N_-(\lambda)V(\lambda_+)N_+^*(\lambda)^t,
\]
and
\[
    H_-(\lambda)
    = N_-(\lambda)N_-^*(\lambda)^t
    = N_-(\lambda)V^{-1}(\lambda_-)N_+^*(\lambda)^t,
\]
where we used the symmetry $N^*(\lambda)=N(\lambda)\sigma_1$.

From the definitions of $r_1(\lambda)$ and $r_2(\lambda)$ in \eqref{def:r1r2} and \eqref{def:r1r2}, we recall that
\[
    r_1(\lambda_+) = -r_1(\lambda_-),\quad \lambda\in\Sigma_1^l,
    \qquad 
    r_2(\lambda_+) = -r_2(\lambda_-),\quad \lambda\in\Sigma_1^r.
\]
Hence for $\lambda\in\Sigma_1^l\cup\Sigma_1^r$, it follows that 
\(
    V^{-1}(\lambda_-) = V(\lambda_+),
\)
and thus
\[
    \int_{\Sigma_1^l\cup\Sigma_1^r} \bigl(H_+(\lambda)-H_-(\lambda)\bigr)\,d\lambda = 0.
\]

By standard arguments, this implies that the only solution to the RH problem for $N(x,t;\lambda)$ is the trivial zero solution. Then, by the vanishing lemma of Zhou \cite{zhou_riemannhilbert_1989}, the RH problem for $X(x,t;\lambda)$ admits a unique solution.

\end{proof}

\begin{rmk}
	It is shown in \cite{boutet_de_monvel_inverse_2008} that when $n=0$ and $m_0=2$, the initial data $u_0(x)$ can be uniquely reconstructed from the scattering data. 
	In Lemma \ref{lem:vanishing}, we only use the fact that $\rho(\lambda)\in L^2(\mathbb{R})$, together with the boundary values of $r_1(\lambda)$ and $r_2(\lambda)$ belonging to $L^2(\Sigma_1^l)$ and $L^2(\Sigma_1^r)$, respectively. 
	This indicates that it suffices to assume that the initial data satisfy the condition $n=0$ and $m_0=2$ in \eqref{initial}, which is consistent with the result in \cite{boutet_de_monvel_inverse_2008}.
\end{rmk}

By Lemma~\ref{lem:asymptotic} and the definition of $M(x;\lambda)$ in 
\ref{def:M}, it follows that the solution $u(x,t)$ exists and is unique. 
Moreover, it can be recovered from the large--$\lambda$ asymptotics of 
$X(x,t;\lambda)$ via
\begin{equation*}
    u(x,t) = -2i\,\frac{d}{dx} 
    \left( \lim_{\lambda \to \infty} 
    \lambda \bigl( X_1(x,t;\lambda) - 1 \bigr) \right),
\end{equation*}
where $X_1(x,t;\lambda)$ denotes the first component of the vector 
$X(x,t;\lambda)$.

\appendix
\section{The periodic travelling wave solution of KdV equation}\label{appendix:periodic solution}
Consider a periodic travelling-wave solution of the KdV equation~\eqref{eq:KdV}. Such solutions have been widely investigated in the context of Whitham modulation theory; see, for example, \cite{grava_whitham_2016,gong_modulation_2025}.
Let \(u(x,t) = f(\xi)\), where  
\[
\xi = kx - \omega t + \phi_0,
\]
and \(f(\xi)\) is a \(2\pi\)-periodic function with wavenumber \(k\), frequency \(\omega\), and an arbitrary phase constant \(\phi_0\). Substituting the travelling–wave ansatz \(f(\xi)\) into~\eqref{eq:KdV} and integrating twice, we obtain
\begin{equation*}
    \frac{k^{2}}{2} f_{\xi}^{2}
    = f^{3} + \frac{\omega}{2k} f^{2} + A f + B,
\end{equation*}
where \(A\) and \(B\) are integration constants.  
Assume \(\beta_1 > \beta_2 > \beta_3\). Then the above identity can be rewritten as
\begin{equation}\label{eq:f}
    \frac{k^{2}}{2} f_{\xi}^{2}
    = (f + \beta_1)(f + \beta_2)(f + \beta_3).
\end{equation}

Since the KdV solution is real-valued, it follows from~\eqref{eq:f} that
\(-\beta_1 \le f \le -\beta_2\). Moreover, from~\eqref{eq:f} we have
\begin{equation}\label{eq:dudxi}
    \frac{k\, du}{\sqrt{2 (u+\beta_1)(u+\beta_2)(u+\beta_3)}} = d\xi .
\end{equation}
Because the period of the solution is \(2\pi\), we obtain
\begin{equation*}
    2k \int_{-\beta_1}^{-\beta_2}
    \frac{du}{\sqrt{2 (u+\beta_1)(u+\beta_2)(u+\beta_3)}}
    = \oint d\xi = 2\pi,
\end{equation*}
which implies
\[
k = \frac{\pi \sqrt{\beta_1 - \beta_3}}{\sqrt{2}\, K(m)},
\qquad
m^{2} = \frac{\beta_1 - \beta_2}{\beta_1 - \beta_3}.
\]

Let
\[
f = -\beta_1 + (\beta_1 - \beta_2) \sin^{2}\phi, \qquad 0 \le \phi \le \frac{\pi}{2}.
\]
Then, using~\eqref{eq:dudxi}, we obtain
\begin{equation*}
    \frac{\sqrt{2}\,k}{\sqrt{\beta_1 - \beta_3}}
    \int_{0}^{\phi} \frac{d\phi}{\sqrt{1 - m^{2} \sin^{2}\phi}}
    = \xi - \xi_{0}.
\end{equation*}
By the definition of the Jacobi elliptic function \(\mathrm{sn}\), we find that
\begin{equation*}
    \sin(\phi)=\mathrm{sn}\left(\frac{\sqrt{\beta_1 - \beta_3}}{\sqrt{2}k}(\xi-\xi_0)\, \middle|\, m\right),
\end{equation*}
Consequently, we can obtain that
\begin{equation}\label{solution:sn}
    f(\xi)
    = -\beta_1 + (\beta_1 - \beta_2)
      \mathrm{sn}^{2}\!\left(
        \frac{\sqrt{\beta_1 - \beta_3}}{\sqrt{2}}
        \left[x - 2(\beta_1 + \beta_2 + \beta_3)t + \frac{\phi_0}{k}\right]
        - \xi_0 \frac{\sqrt{\beta_1 - \beta_2}}{\sqrt{2}\,k}
      \, \middle|\, m\right).
\end{equation}
Using the arbitrariness of \(\phi_0\) and the identity \(\mathrm{dn}^{2} + m^{2}\mathrm{sn}^{2} = 1\), we may rewrite~\eqref{solution:sn} as
\begin{equation*}
    u(x,t)
    = f(\xi)
    = -\beta_3 - (\beta_1 - \beta_3)\,
    \mathrm{dn}^{2}\!\left(
        \frac{\sqrt{\beta_1 - \beta_3}}{\sqrt{2}}
        \left( x - 2(\beta_1 + \beta_2 + \beta_3)t - x_{0} \right)
        + K(m)\, \middle|\, m
    \right),
\end{equation*}
where \(x_{0} = -(\phi_0 - \xi_0 - \pi)/k\).
\section{The Lax spectrum of periodic finite gap solution}\label{appendix:Bloch spectrum}
For convenience, consider the stationary periodic solution of the KdV equation \eqref{eq:KdV}. 
Introduce the moving coordinate
\[
\zeta := x - 2(\beta_1+\beta_2+\beta_3)\,t,\qquad \tau := t,
\]
so that the solution \eqref{solution:periodic solution} can be written as
\begin{equation}\label{solution:staionary}
    u(\zeta,\tau)=u(\zeta)
    = -\beta_3 - (\beta_1-\beta_3)\,
      \mathrm{dn}^{2}\!\left(
        \frac{\sqrt{\beta_1-\beta_3}}{\sqrt{2}}(\zeta-x_0) + K(m)
        \,\middle|\, m
      \right),
\end{equation}
with $\beta_1>\beta_2>\beta_3$, and which is independent of \(\tau\) (stationary in the moving frame). This profile satisfies the stationary KdV equation
\begin{equation}\label{eq:kdv-stationary}
    v\,u_{\zeta} + 6\,u\,u_{\zeta} - u_{\zeta\zeta\zeta} = 0,\ v = 2(\beta_1+\beta_2+\beta_3).
\end{equation}

Furthermore, denote \( z = \lambda^{2} \). Then the Lax pair~\eqref{Lax:matrix} can be rewritten as
\begin{equation}\label{Lax-staionary}
\begin{aligned}
&\begin{cases}
\Psi_{\zeta} = \mathcal{L}\, \Psi, \\
\Psi_{\tau} = \mathcal{A}\, \Psi,
\end{cases} \\[4pt]
\mathcal{L} =
\begin{pmatrix}
0 & 1 \\
- z + u & 0
\end{pmatrix}, 
\qquad
&\mathcal{A} =
\begin{pmatrix}
- u_{\zeta} & 4z + 2u + v \\
 (u - z)(4z + 2u + v) - u_{\zeta\zeta} & u_{\zeta}
\end{pmatrix}.
\end{aligned}
\end{equation}
Since the time–evolution matrix \( \mathcal{A} \) in~\eqref{Lax-staionary} is independent of \( \tau \), we may look for a \(2\times 1\) eigenvector solution of the form
\begin{equation}\label{eigenvector}
    \psi(\zeta,\tau)
= e^{\, iy \tau}
\begin{pmatrix}
    \psi_{1}(\zeta) \\
    \psi_{2}(\zeta)
\end{pmatrix}.
\end{equation}

Plugging the eigenvector solution~\eqref{eigenvector} into the time–part equation in~\eqref{Lax-staionary}, we obtain
\begin{equation}\label{eq:phi12}
    (\mathcal{A}-yI)
    \begin{pmatrix}
        \psi_{1}(\zeta) \\
        \psi_{2}(\zeta)
    \end{pmatrix}
    =
    \begin{pmatrix}
        -u_{\zeta} - iy & 4z + 2u + v \\
        (u - z)(4z + 2u + v) - u_{\zeta\zeta} & u_{\zeta} - iy
    \end{pmatrix}
    \begin{pmatrix}
        \psi_{1}(\zeta) \\
        \psi_{2}(\zeta)
    \end{pmatrix}
    = 0.
\end{equation}
The Lax spectrum is defined as the set of all \(x \in \mathbb{R}\) for which the spectral problem~\eqref{Lax-staionary} admits a bounded eigenvector solution \cite{deconinck_orbital_2020,levitan_introduction_1975}.  
This requires that~\eqref{eq:phi12} have a nontrivial solution, which is equivalent to 
\begin{equation}
    \det(\mathcal{A} - yI) = 0.
\end{equation}

Thus, the parameter \(y\) must satisfy
\begin{equation}\label{eq:yu}
    y^{2}
    = 16 z^{3} + 8 v z^{2} + (v^{2} - C_{1}) z - C_{2},
\end{equation}
where
\[
C_{1} = 12u^{2} - 4u_{\zeta\zeta} + 4vu,
\qquad
C_{2} = 4u^{3} + u_{\zeta}^{2} + 4u^{2}v - u_{\zeta\zeta}v
        + u\left(-2u_{\zeta\zeta} + v^{2}\right).
\]
Since \(u\) satisfies the stationary KdV equation~\eqref{eq:kdv-stationary}, it follows that
\(C_{1}\) and \(C_{2}\) are constants. More precisely, substituting the periodic solution
\eqref{solution:staionary} into~\eqref{eq:yu}, we obtain
\begin{equation}\label{curve}
    y^{2}
    = 2\,(2z+\beta_{1}+\beta_{2})(2z+\beta_{1}+\beta_{3})
        (2z+\beta_{2}+\beta_{3}),
\end{equation}
which defines a two-sheeted genus-one Riemann surface. 
The branch cut is given by
\begin{equation*}
    \Big(-\infty, -\frac{\beta_1 + \beta_2}{2}\Big] 
      \;\cup\; 
      \Big[-\frac{\beta_1 + \beta_3}{2}, -\frac{\beta_1 + \beta_2}{2}\Big] .
\end{equation*}

\section*{Acknowledgments}
The author is deeply grateful to Professor Tamara Grava and Professor Robert Jenkins for their guidance and for proposing the research problem. The author also thanks Xiaofan Zhang and Zechuan Zhang for insightful discussions and valuable comments. Their related work~\cite{grava_direct_nodate} was prepared in parallel with the present manuscript. 

The  author acknowledge the support of the scholarship provided by the China Scholarship Council (CSC) under Grant No. 202406040149 and the GNFM-INDAM group and the research project Mathematical Methods in NonLinear Physics (MMNLP), Gruppo 4-Fisica Teorica of INFN.   The author further thanks Professor Deng-shan Wang and Cheng Zhu for their lecture notes on the KdV equation.

\bibliographystyle{plain}
\bibliography{references}

@unpublished{grava_direct_nodate,
	title = {Direct {Scattering} of the {Focusing} {Nonlinear} {Schrödinger} {Equation} with {Step}-like {Oscillatory} {Initial} {Data}},
	author = {Grava, Tamara and Jenkins, Robert and Zhang, Xiaofan and Zhang, Zechuan},
}

@article{jenkins_approximation_2025,
	title = {Approximation of the {Thermodynamic} {Limit} of {Finite}-{Gap} {Solutions} to the {Focusing} {NLS} {Hierarchy} by {Multisoliton} {Solutions}},
	volume = {406},
	number = {9},
	journal = {Communications in Mathematical Physics},
	publisher = {Springer},
	author = {Jenkins, Robert and Tovbis, Alexander},
	year = {2025},
	note = {ISBN: 0010-3616},
	pages = {214},
}

@book{levitan_introduction_1975,
	title = {Introduction to spectral theory: selfadjoint ordinary differential operators},
	volume = {39},
	isbn = {0-8218-8663-0},
	publisher = {American Mathematical Soc.},
	author = {Levitan, Boris Moiseevich and Sargsian, Ishkhan Saribekovich},
	year = {1975},
}

@article{gong_modulation_2025,
	title = {Modulation theory of soliton−mean flow in {Korteweg}–de {Vries} equation with box type initial data},
	volume = {134},
	issn = {0165-2125},
	url = {https://www.sciencedirect.com/science/article/pii/S0165212524001975},
	doi = {10.1016/j.wavemoti.2024.103467},
	journal = {Wave Motion},
	author = {Gong, Ruizhi and Wang, Deng-Shan},
	year = {2025},
	pages = {103467},
}

@article{gesztesy_one-dimensional_1997,
	title = {One-dimensional scattering theory for quantum systems with nontrivial spatial asymptotics},
	volume = {10},
	url = {https://doi.org/10.57262/die/1367525666},
	doi = {10.57262/die/1367525666},
	number = {3},
	journal = {Differential and Integral Equations},
	publisher = {Khayyam Publishing, Inc.},
	author = {Gesztesy, F. and Nowell, R. and Pötz, W.},
	year = {1997},
	pages = {521 -- 546},
}

@article{trogdon_numerical_2014,
	title = {A numerical dressing method for the nonlinear superposition of solutions of the {KdV} equation},
	volume = {27},
	copyright = {http://iopscience.iop.org/info/page/text-and-data-mining},
	issn = {0951-7715, 1361-6544},
	url = {https://iopscience.iop.org/article/10.1088/0951-7715/27/1/67},
	doi = {10.1088/0951-7715/27/1/67},
	language = {en},
	number = {1},
	urldate = {2025-11-25},
	journal = {Nonlinearity},
	author = {Trogdon, Thomas and Deconinck, Bernard},
	month = jan,
	year = {2014},
	pages = {67--86},
}

@article{deconinck_orbital_2020,
	title = {The {Orbital} {Stability} of {Elliptic} {Solutions} of the {Focusing} {Nonlinear} {Schrödinger} {Equation}},
	volume = {52},
	issn = {0036-1410, 1095-7154},
	url = {https://epubs.siam.org/doi/10.1137/19M1240757},
	doi = {10.1137/19M1240757},
	language = {en},
	number = {1},
	urldate = {2025-11-25},
	journal = {SIAM Journal on Mathematical Analysis},
	author = {Deconinck, Bernard and Upsal, Jeremy},
	month = jan,
	year = {2020},
	pages = {1--41},
}

@book{belokolos_algebro-geometric_nodate,
	address = {Berlin ;},
	title = {Algebro-geometric approach to nonlinear integrable equations},
	isbn = {3540502653 (Berlin : acid-free), 0387502653 (New York : acid-free), 9783540502654 (Berlin : acid-free), 9780387502656 (New York : acid-free)},
	publisher = {Springer-Verlag,},
	editor = {Belokolos, E. D.},
	note = {Publication Title: Springer series in nonlinear dynamics},
}

@article{dyachenko_primitive_2016,
	title = {Primitive potentials and bounded solutions of the {KdV} equation},
	volume = {333},
	issn = {01672789},
	url = {https://linkinghub.elsevier.com/retrieve/pii/S0167278916301518},
	doi = {10.1016/j.physd.2016.04.002},
	language = {en},
	urldate = {2025-10-02},
	journal = {Physica D: Nonlinear Phenomena},
	author = {Dyachenko, S. and Zakharov, D. and Zakharov, V.},
	month = oct,
	year = {2016},
	pages = {148--156},
}

@article{gurevich_nonstationary_1974,
	title = {Nonstationary structure of a collisionless shock wave},
	volume = {38},
	doi = {http://jetp.ras.ru/cgi-bin/dn/e_038_02_0291.pdf},
	language = {en},
	number = {2},
	journal = {Soviet Physics – JETP},
	author = {Gurevich, A. V. and PitaevskiT, L. P.},
	month = feb,
	year = {1974},
	pages = {291--297},
}

@article{gardner_method_1967,
	title = {Method for {Solving} the {Korteweg}-{deVries} {Equation}},
	volume = {19},
	copyright = {http://link.aps.org/licenses/aps-default-license},
	issn = {0031-9007},
	url = {https://link.aps.org/doi/10.1103/PhysRevLett.19.1095},
	doi = {10.1103/PhysRevLett.19.1095},
	language = {en},
	number = {19},
	urldate = {2025-08-21},
	journal = {Physical Review Letters},
	author = {Gardner, Clifford S. and Greene, John M. and Kruskal, Martin D. and Miura, Robert M.},
	month = nov,
	year = {1967},
	pages = {1095--1097},
}

@article{egorova_cauchy_2009,
	title = {On the {Cauchy} problem for the {Korteweg}–de {Vries} equation with steplike finite-gap initial data: {I}. {Schwartz}-type perturbations},
	volume = {22},
	issn = {0951-7715, 1361-6544},
	shorttitle = {On the {Cauchy} problem for the {Korteweg}–de {Vries} equation with steplike finite-gap initial data},
	url = {https://iopscience.iop.org/article/10.1088/0951-7715/22/6/009},
	doi = {10.1088/0951-7715/22/6/009},
	language = {en},
	number = {6},
	urldate = {2025-06-10},
	journal = {Nonlinearity},
	author = {Egorova, Iryna and Grunert, Katrin and Teschl, Gerald},
	month = jun,
	year = {2009},
	pages = {1431--1457},
}

@article{boutet_de_monvel_inverse_2008,
	title = {Inverse scattering theory for one-dimensional {Schrödinger} operators with steplike finite-gap potentials},
	volume = {106},
	copyright = {http://www.springer.com/tdm},
	issn = {0021-7670, 1565-8538},
	url = {http://link.springer.com/10.1007/s11854-008-0050-4},
	doi = {10.1007/s11854-008-0050-4},
	language = {en},
	number = {1},
	urldate = {2025-06-10},
	journal = {Journal d'Analyse Mathématique},
	author = {Boutet De Monvel, Anne and Egorova, Iryna and Teschl, Gerald},
	month = jan,
	year = {2008},
	pages = {271--316},
}

@article{trogdon_riemannhilbert_2013,
	title = {A {Riemann}–{Hilbert} problem for the finite-genus solutions of the {KdV} equation and its numerical solution},
	volume = {251},
	copyright = {https://www.elsevier.com/tdm/userlicense/1.0/},
	issn = {01672789},
	url = {https://linkinghub.elsevier.com/retrieve/pii/S0167278913000444},
	doi = {10.1016/j.physd.2013.01.018},
	language = {en},
	urldate = {2025-06-06},
	journal = {Physica D: Nonlinear Phenomena},
	author = {Trogdon, Thomas and Deconinck, Bernard},
	month = may,
	year = {2013},
	pages = {1--18},
}

@article{nabelek_algebro-geometric_2020,
	title = {Algebro-geometric finite gap solutions to the {Korteweg}–de {Vries} equation as primitive solutions},
	volume = {414},
	issn = {01672789},
	url = {https://linkinghub.elsevier.com/retrieve/pii/S0167278919304610},
	doi = {10.1016/j.physd.2020.132709},
	language = {en},
	urldate = {2025-05-08},
	journal = {Physica D: Nonlinear Phenomena},
	author = {Nabelek, Patrik V.},
	month = dec,
	year = {2020},
	pages = {132709},
}

@article{egorova_cauchy_2011,
	title = {On the {Cauchy} problem for the {Kortewegde} {Vries} equation with steplike finite-gap initial data {II}. {Perturbations} with finite moments},
	volume = {115},
	copyright = {http://www.springer.com/tdm},
	issn = {0021-7670, 1565-8538},
	url = {http://link.springer.com/10.1007/s11854-011-0024-9},
	doi = {10.1007/s11854-011-0024-9},
	language = {en},
	number = {1},
	urldate = {2025-05-08},
	journal = {Journal d'Analyse Mathématique},
	author = {Egorova, Iryna and Teschl, Gerald},
	month = jun,
	year = {2011},
	pages = {71--101},
}

@article{zhou_riemannhilbert_1989,
	title = {The {Riemann}–{Hilbert} {Problem} and {Inverse} {Scattering}},
	volume = {20},
	issn = {0036-1410, 1095-7154},
	url = {http://epubs.siam.org/doi/10.1137/0520065},
	doi = {10.1137/0520065},
	language = {en},
	number = {4},
	urldate = {2024-12-16},
	journal = {SIAM Journal on Mathematical Analysis},
	author = {Zhou, Xin},
	month = jul,
	year = {1989},
	pages = {966--986},
}

@book{lawden_elliptic_1989,
	address = {New York, NY},
	series = {Applied {Mathematical} {Sciences}},
	title = {Elliptic {Functions} and {Applications}},
	volume = {80},
	copyright = {http://www.springer.com/tdm},
	isbn = {978-1-4419-3090-3 978-1-4757-3980-0},
	url = {http://link.springer.com/10.1007/978-1-4757-3980-0},
	doi = {10.1007/978-1-4757-3980-0},
	language = {en},
	urldate = {2024-12-02},
	publisher = {Springer New York},
	author = {Lawden, Derek F.},
	editor = {John, F. and Marsden, J. E. and Sirovich, L.},
	year = {1989},
}

@incollection{grava_whitham_2016,
	title = {Whitham modulation equations and application to small dispersion asymptotics and long time asymptotics of nonlinear dispersive equations},
	volume = {926},
	url = {http://arxiv.org/abs/1701.00069},
	doi = {10.1007/978-3-319-39214-1_10},
	language = {en},
	urldate = {2024-11-07},
	author = {Grava, Tamara},
	year = {2016},
	note = {arXiv:1701.00069 [math-ph]},
	pages = {309--335},
}

@article{girotti_rigorous_2021,
	title = {Rigorous {Asymptotics} of a {KdV} {Soliton} {Gas}},
	volume = {384},
	issn = {1432-0916},
	url = {https://doi.org/10.1007/s00220-021-03942-1},
	doi = {10.1007/s00220-021-03942-1},
	number = {2},
	journal = {Communications in Mathematical Physics},
	author = {Girotti, M. and Grava, T. and Jenkins, R. and McLaughlin, K. D. T.-R.},
	month = jun,
	year = {2021},
	pages = {733--784},
}

\end{document}